\documentclass[11pt,leqno]{amsart}

\usepackage{amsmath,amssymb,amsthm,cite,mathrsfs,amsfonts,stmaryrd,bm}

\usepackage[abbrev]{amsrefs}
\usepackage{graphicx,tikz,makecell}

\usepackage{mathtools}
\mathtoolsset{showonlyrefs=true}

\usepackage[top=30mm, bottom=30mm, left=25mm, right=25mm]{geometry}

\usepackage{color}

\allowdisplaybreaks[1]

\newtheorem{theo}{Theorem}[section]
\newtheorem{lemm}[theo]{Lemma}
\newtheorem{corr}[theo]{Corollary}
\newtheorem{prop}[theo]{Proposition}

\numberwithin{equation}{section}

\theoremstyle{definition}

\newtheorem{rema}[theo]{Remark}

\begin{document}
\title[]{Global well-posedness for the Keller--Segel--Navier--Stokes system with nonlinear boundary conditions}

\author[]{Taiki Takeuchi}
\address[Taiki Takeuchi]{Institute of Mathematics for Industry, Kyushu University, 744 Motooka, Nishi-ku, Fukuoka, 819-0395, Japan}		
\email{takeuchi.taiki.643@m.kyushu-u.ac.jp}

\author[]{Keiichi Watanabe}
\address[Keiichi Watanabe]{School of General and Management Studies, Suwa University of Science, 5000-1, Toyohira, Chino, Nagano, 391-0292, Japan}		
\email{watanabe\_keiichi@rs.sus.ac.jp}

\thanks{}

\subjclass[2020]{Primary; 35Q92, Secondary; 35K05, 35K61, 35Q35, 76D03, 92C17}

\keywords{Keller--Segel--Navier--Stokes system; Nonlinear boundary conditions; Maximal regularity; Global well-posedness}

\date{}

\maketitle

\begin{abstract}
In this paper, we consider the Keller--Segel--Navier--Stokes system with nonlinear boundary conditions in a bounded smooth (and not necessarily convex) domain $\Omega \subset \mathbb{R}^N$, $N \ge 2$, where the chemotactic sensitivity $S$ is assumed to have values in $\mathbb{R}^{N \times N}$ which accounts for rotational fluxes.
In contrast to the case where $S$ is a scalar-valued function (or $S$ is the identity matrix), in our system, the normal derivative for the density $n$ of the cell is given as the product of the unknown functions, i.e., the function $n$ satisfies the \textit{nonlinear} boundary condition.
We show the existence and uniqueness of global strong solutions to the system under the smallness assumptions of given data, where the Lipschitz continuity of the solution mapping and the asymptotic stability of the solution are also shown.
The proof is based on maximal regularity results for the linear heat equation and the Stokes system, where we establish a new maximal regularity theorem for some linear heat equation with an \textit{inhomogeneous} Neumann boundary condition.
Since we develop a direct approach to construct the solutions (i.e., \textit{without} considering limiting procedure in certain regularized problem with homogeneous linear boundary conditions), our solutions indeed satisfy the boundary conditions, which was not addressed clearly in the previous contributions by Cao and Lankeit \cite{CaoLankeit} as well as Yu, Wang, and Zheng \cite{YuWangZheng18}.
Under suitable regularity conditions on given data, the solution may be understood in a classical sense and, as a by-product, the non-negativity result is proved via the maximum principle of a new type.
\end{abstract}

\section{Introduction}
\label{sec-1}
Let us consider the following Keller--Segel--Navier--Stokes system with \textit{nonlinear} boundary conditions in a bounded smooth domain $\Omega \subset \mathbb{R}^N$, $N \ge 2$;
\begin{equation}\label{KSNS}
\left\{\begin{aligned}
\partial_tn &= \Delta n-\nabla \cdot (nS(t,x)\nabla c)-\bm{u} \cdot \nabla n, & t>0, \, x &\in \Omega, \\
\partial_tc &=\Delta c-c+n-\bm{u} \cdot \nabla c, & t>0, \, x &\in \Omega, \\
\partial_t\bm{u}+(\bm{u}\cdot\nabla)\bm{u} &= \Delta \bm{u}-\nabla p+n\nabla \varphi+\bm{f}, & t>0, \, x &\in \Omega, \\
\nabla \cdot \bm{u} &=0, & t>0, \, x &\in \Omega, \\
\nabla n \cdot \bm{\nu} &=nS(t,x)\nabla c \cdot \bm{\nu}, \quad \nabla c \cdot \bm{\nu}=\bm{u}=0, & t>0, \, x &\in \partial\Omega, \\
(n,c,\bm{u})(0,x) &= (n_0,c_0,\bm{u}_0)(x), & x &\in \Omega,
\end{aligned}\right.
\end{equation}
where $n=n(t,x)$, $c=c(t,x)$, $\bm{u}=\bm{u}(t,x)$, and $p=p(t,x)$ are the unknown functions standing for the density of the cell, the concentration of the chemo-attractant, the velocity of the fluid, and the pressure, respectively, whereas the potential function $\varphi = \varphi (x)$ and the vector-valued function $\bm{f}=\bm{f}(t,x)$ are assumed to be given and $(n_0,c_0,\bm{u}_0) (x)$ denotes the given initial data.
In addition, the \textit{tensor-valued} function $S = S(t,x)$ stands for the chemotactic sensitivity, which is also assumed to be given.
A unit outer normal to $\partial\Omega$ is denoted by $\bm{\nu}=\bm{\nu}(x)$.
From the biological viewpoint, it is reasonable to suppose that $n_0$ and $c_0$ are non-negative in $\Omega$ and expect that $n$ and $c$ are non-negative in $(0, \infty) \times \Omega$, where we will discuss this issue in the sequel.

System \eqref{KSNS} may be regarded as a generalization of the standard Keller--Segel--Navier--Stokes system in the case that the evolution of the chemo-attractant is (essentially) governed by production through cells, which is motivated by recent modeling results \cite{xueothmer,xue}.
In fact, those studies suggest that the motion of bacteria near surfaces has rotational components in the cross-diffusion flux.
To the best of the authors' knowledge, there are very few results on the global existence and uniqueness result for \eqref{KSNS} with a general $S$ and the asymptotic behavior of the global solutions, in particular, for the \textit{general dimensions} $N \ge 2$.
Note that, in the case that $S \equiv I$ (the identity matrix) or $S$ is replaced by a scalar function, there are abundant results for \eqref{KSNS} or the system with the second equation replaced by
\begin{equation}\label{consumption}
\partial_tc=\Delta c -nc-\bm{u}\cdot\nabla c,
\end{equation}
e.g., on the local/global existence, large-time behavior, and blow-up of solutions.
For the details, we refer to \cite{dilorz, duan, liulorz, winkler2012, chaekanglee2013, chaekanglee2014, JWZ15, Winkler16, KMS16, winkler2020, watanabe, kangleewinkler, Winkler19} and references therein (see also \cite{LL1,LL2,LL3} for fairly recent results on related models).
However, there are few studies in the case that $S$ is a general matrix.
Indeed, the introduction of general tensor-valued sensitivities $S$ induces difficulty due to the destruction of the natural energy structure arising from \eqref{KSNS}, and hence the proof of the existence of (weak) solutions to \eqref{KSNS} requires technical efforts.
For instance, the nonlinear boundary condition $\nabla n \cdot \bm{\nu} =nS(t,x)\nabla c \cdot \bm{\nu}$ on $(0, \infty) \times \Omega$ is harmful when we use integration by parts, i.e., when we establish the energy estimates, unlike the case where the homogeneous Neumann boundary condition for $n$ is imposed \cite{winkler2012,Winkler15,Winkler16}.
Note that this nonlinear boundary condition also prevents us from formulating \eqref{KSNS} as an abstract Cauchy problem so that we may \textit{not} rewrite System \eqref{KSNS} in the integral form.
To overcome this difficulty, an approximation argument was often used.
Precisely speaking, by introducing the family of cut-off functions that vanish near the boundary, one may regularize \eqref{KSNS} as the corresponding system with the homogeneous boundary conditions, i.e., $n_\eta$ (the regularization of $n$) satisfies $\nabla n_\eta \cdot \bm{\nu} = 0$ on $\partial \Omega$.
Such an approximation procedure was used by, e.g., Wang, Winkler, and Xiang \cite{WangWinklerXiang18}, Wang \cite{Wang17}, and Yu, Wang, and Zheng \cite{YuWangZheng18}: Wang, Winkler, and Xiang \cite{WangWinklerXiang18} considered the case of a 2D bounded smooth and convex domain $\Omega \subset \mathbb{R}^2$ and assumed that the tensor-valued sensitivity $S$ depends on $x$, $n$, and $c$ and satisfies $|S(x,n,c)| \le C(1+n)^{-\alpha}$ with some $C,\alpha>0$. Then they showed the global existence of classical solutions to \eqref{KSNS} for arbitrarily large initial data and obtained the uniform boundedness of the global solutions in space and time.
Wang \cite{Wang17} assumed a similar condition to that of \cite{WangWinklerXiang18} in the case $N=3$ and obtained global (very) weak solutions to \eqref{KSNS} without any smallness conditions of the initial data.
Yu, Wang, and Zheng \cite{YuWangZheng18} assumed $|S(x,n,c)| \le C$ and certain smallness conditions on the initial data to prove the global existence of classical solutions to \eqref{KSNS} in the case of $N \in \{2,3\}$, where the exponential decay of the global solutions was verified as well.
The assumption appearing in \cite{Wang17} was relaxed by Ke and Zheng \cite{KeZheng19}.
We also mention that a more general model, such as the quasilinear Keller--Segel--Navier--Stokes system modeling coral fertilization, was investigated by Zheng \cite{Z21} and Zheng and Ke \cite{ZK22}.
We remark that the assumptions of \cite{WangWinklerXiang18, Wang17, KeZheng19, Z21, ZK22} for $S$ differ from ours in the sense that while $S$ depends on $n$ and $c$, it is necessary to \textit{restrict} the growth of $S$ so that $|S(x,n,c)| \le C(1+n)^{-\alpha}$, in particular, to exclude the standard case $S \equiv I$. Although our result given below requires some smallness conditions, we do not have to assume any growth condition of $S$, and thus our result covers the case $S \equiv I$ as well. In that sense, our assumption is similar to that of \cite{YuWangZheng18}. However, we refrain from further discussion on the conditions of $S$ in this paper.
Furthermore, as mentioned above, System \eqref{KSNS} with the second equation replaced by \eqref{consumption} has been studied as well as our present problem \eqref{KSNS}; in this case, it should be noted that Cao and Lankeit \cite{CaoLankeit} had shown the corresponding result to that of \cite{YuWangZheng18} before they did.
We also refer to, e.g., Winkler \cite{Winkler18, Winkler21}, Black \cite{Black19}, Zheng \cite{Zheng22}, and Heihoff \cite{Heihoff23} for related results on System \eqref{KSNS} or with the second equation replaced by \eqref{consumption}.
It should be emphasized that the uniqueness of the global classical or weak solutions was not discussed in \textit{all} the aforementioned studies, even for the case $N = 2$.

The aim of this paper is to prove the global well-posedness of \eqref{KSNS} in the sense of Hadamard, i.e., we also verify that the solution depends continuously on the data.
In particular, we show that \eqref{KSNS} admits a \textit{unique} global strong solution provided that the initial data are small, and if given functions are suitably regular, we also investigate the smoothness and non-negativity issues for the global solution.
In contrast to the previous works \cite{Winkler21, Winkler18, CaoLankeit, KeZheng19, Wang17, WangWinklerXiang18, Heihoff23, Black19, YuWangZheng18, Zheng22, Z21, ZK22}, we rely on the maximal regularity technique so that we do \textit{not} need to approximate \eqref{KSNS} by the corresponding system with the homogeneous boundary conditions.

\subsection{Main results}

Before giving our main results, we shall introduce the abstract form of System \eqref{KSNS}.
To this end, we shall introduce several notations.
For $1 \le r \le \infty$ and $k \in \mathbb{N}_0 \coloneqq \mathbb{N} \cup \{0\}$, let $L^r(\Omega)$ and $W^{k,r}(\Omega)$ denote the usual Lebesgue spaces and Sobolev spaces defined on $\Omega$, respectively.
We agree that $W^{0,r}(\Omega) \coloneqq L^r(\Omega)$ in the case of $k=0$.
Moreover, we also set
\begin{align}
L_0^r(\Omega) &\coloneqq \left\{ \psi \in L^r(\Omega) \, \left| \, \int_{\Omega}\psi(x)\,dx=0 \right.\right\}, \\
W_{\nu}^{k,r}(\Omega) &\coloneqq \left\{\left. \psi \in W^{k,r}(\Omega) \, \right| \, \nabla \psi \cdot \bm{\nu}=0 \text{ on $\partial\Omega$} \right\} \quad \text{for $k \ge 2$}
\end{align}
and introduce the Neumann--Laplacian $-\Delta_{\mathrm{N},1} \coloneqq -\Delta$ in $W^{1,r} (\Omega)$ with the domain $D(-\Delta_{\mathrm{N},1}) \coloneqq W_{\nu}^{3,r}(\Omega)$, where the boundary condition $\nabla \psi \cdot \bm{\nu}=0$ on $\partial\Omega$ is regarded as $(\gamma\nabla \psi)(x) \cdot \bm{\nu}(x)=0$ a.e.~$x \in \partial\Omega$ with a suitable trace operator $\gamma$ (see Remark \ref{globalrema} (ii) below).
Let $C_0^{\infty}(\Omega)$ be the set of smooth functions which have a compact support in $\Omega$ and let $C_{0,\sigma}^{\infty}(\Omega)$ be the set of $\bm{\psi} \in C_0^{\infty}(\Omega)^N$ such that $\nabla \cdot \bm{\psi}=0$ on $\Omega$.
Then $W_0^{1,r}(\Omega)$ and $L_{\sigma}^r(\Omega)$ denote the $W^{1,r}(\Omega)$-closures of $C_0^{\infty}(\Omega)$ and the $L^r(\Omega)$-closures of $C_{0,\sigma}^{\infty}(\Omega)$ for $1<r<\infty$, respectively.
With the aid of the Helmholtz projection operator $P:L^r(\Omega)^N \to L_{\sigma}^r(\Omega)$ for $1<r<\infty$, we define the Stokes operator in $L^r_\sigma (\Omega)$ by $A \coloneqq -P\Delta : D(A) \to L_{\sigma}^r(\Omega)$ with the domain $D(A) \coloneqq L_{\sigma}^r(\Omega) \cap W_0^{1,r}(\Omega)^N \cap W^{2,r}(\Omega)^N$.
Then, we may rewrite System \eqref{KSNS} to the following form:
\begin{equation}\label{ABS}
\left\{\begin{aligned}
\partial_tn-\Delta n &= -\nabla \cdot (nS(t,x)\nabla c)-\bm{u} \cdot \nabla n, & t>0, \, x &\in \Omega, \\
\partial_tc+(1-\Delta_{\mathrm{N},1})c &=n-\bm{u} \cdot \nabla c, & t>0, \, x &\in \Omega, \\
\partial_t\bm{u}+A\bm{u} &= -P(\bm{u}\cdot\nabla)\bm{u}+P(n\nabla \varphi)+P\bm{f}, & t>0, \, x &\in \Omega, \\
\nabla n \cdot \bm{\nu} &=nS(t,x)\nabla c \cdot \bm{\nu}, & t>0, \, x &\in \partial\Omega, \\
(n,c,\bm{u})(0,x) &= (n_0,c_0,\bm{u}_0)(x), & x &\in \Omega.
\end{aligned}\right.
\end{equation}
Here, the initial data belong to subspaces of the Besov spaces $B_{r,q}^s(\Omega)$ whose definition will be given in the next section.
In this paper, we intend to seek unique global strong solutions to System \eqref{ABS} via maximal regularity theory and prove some properties of the solution provided that the given data are small in their natural norms.

Our first result is on the unique existence of global strong solutions to \eqref{ABS} for small given data.
To simplify the notation, for $1 \le q \le \infty$ and a Banach space $X$, we abbreviate $\|\,\cdot\,\|_{L^q(X)} \coloneqq \|\,\cdot\,\|_{L^q(0,\infty;X)}$ and $\|\,\cdot\,\|_{W^{1,q}(X)} \coloneqq \|\,\cdot\,\|_{W^{1,q}(0,\infty;X)}$.
We also set the mean value $\overline{n}_0$ of the mass of the initial density $n_0$ as $\overline{n}_0 \coloneqq |\Omega|^{-1}\int_{\Omega}n_0(x)\,dx$.

\begin{theo}[Global well-posedness]\label{globalsol}
Let $\Omega$ be a bounded smooth domain in $\mathbb{R}^N$, $N \ge 2$.
In addition, let $N<r<\infty$ and $2<q<\infty$ satisfy $1/r+2/q \neq 1$ and let $0<\lambda_1<\min\{1,\lambda_{\mathrm{N}}(\Omega)/q\}$ and $\lambda_2 \in [0,\lambda_1] \cap [0,\lambda_{\mathrm{D}}(\Omega)/q)$, where $\lambda_{\mathrm{N}}(\Omega),\lambda_{\mathrm{D}}(\Omega) \in (0,\infty)$ are given by
\begin{equation}\label{lamdef}
\lambda_{\mathrm{N}}(\Omega) \coloneqq \inf_{\psi \in (W^{1,2} \cap L_0^2)(\Omega) \setminus \{0\}}\frac{\|\nabla \psi\|_{L^2(\Omega)}^2}{\|\psi\|_{L^2(\Omega)}^2}, \quad \lambda_{\mathrm{D}}(\Omega) \coloneqq \inf_{\psi \in W_0^{1,2}(\Omega) \setminus \{0\}}\frac{\|\nabla \psi\|_{L^2(\Omega)}^2}{\|\psi\|_{L^2(\Omega)}^2}.
\end{equation}
Assume that the initial data $n_0 \in B_{r,q}^{2-2/q}(\Omega)$, $c_0 \in B_{r,q}^{3-2/q}(\Omega)$, and $\bm{u}_0 \in L_{\sigma}^r(\Omega) \cap B_{r,q}^{2-2/q}(\Omega)^N$ satisfy $\nabla c_0 \cdot \bm{\nu}=\bm{u}_0=0$ on $\partial\Omega$ and the given functions $\varphi$, $\bm{f}$, and $S$ satisfy
\begin{alignat}{2}
\nabla\varphi &\in L^r(\Omega)^N, & e^{\lambda_2 t}\bm{f} &\in L^q(0,\infty;L^r(\Omega)^N), \\
\label{scond}
S &\in L^{\infty}(0,\infty;W^{1,r}(\Omega)^{N^2}), \quad& \partial_tS &\in L^q(0,\infty;L^r(\Omega)^{N^2}).
\end{alignat}
Furthermore, in the case $1/r+2/q<1$, suppose that $n_0$, $c_0$, and $S$ fulfill the additional condition
\begin{equation}\label{addcompcond}
\nabla n_0\cdot\bm{\nu}=n_0S(0,\,\cdot\,)\nabla c_0 \cdot \bm{\nu} \quad \text{on $\partial\Omega$}.
\end{equation}
Then there exists a constant $\varepsilon>0$ independent of $n_0$, $c_0$, $\bm{u}_0$, $\varphi$, $\bm{f}$, and $S$ such that if
\begin{equation}
\begin{aligned}
\llbracket (n_0,c_0,\bm{u}_0,\bm{f}) \rrbracket &\coloneqq \|n_0\|_{B_{r,q}^{2-2/q}(\Omega)}+\|c_0\|_{B_{r,q}^{3-2/q}(\Omega)}+\|\bm{u}_0\|_{B_{r,q}^{2-2/q}(\Omega)}+\|e^{\lambda_2 t}\bm{f}\|_{L^q(L^r(\Omega))} \\
&\le \varepsilon (1+\|S\|_{L^{\infty}(W^{1,r}(\Omega))}+\|\partial_tS\|_{L^q(L^r(\Omega))})^{-1}(1+\|\nabla\varphi\|_{L^r(\Omega)})^{-2},
\end{aligned}
\end{equation}
then System \eqref{ABS} has a unique global strong solution $(n,c,\bm{u})$ satisfying
\begin{equation}\label{solclass}
\left\{\begin{aligned}
e^{\lambda_1 t} (n-\overline{n}_0) &\in L^q(0,\infty; (W^{2,r} \cap L_0^r)(\Omega)) \cap W^{1,q}(0,\infty;L_0^r(\Omega)), \\
e^{\lambda_1 t}\{c-(1-e^{-t})\overline{n}_0\} &\in L^q(0,\infty;W_{\nu}^{3,r}(\Omega)) \cap W^{1,q}(0,\infty;W^{1,r}(\Omega)), \\
e^{\lambda_2 t}\bm{u} &\in L^q(0,\infty;D(A)) \cap W^{1,q}(0,\infty;L_{\sigma}^r(\Omega)).
\end{aligned}\right.
\end{equation}
The nonlinear boundary condition is satisfied in the sense that
\begin{equation}\label{VNBC}
(\gamma(\nabla n(t,\,\cdot\,)-n(t,\,\cdot\,)S(t,\,\cdot\,)\nabla c(t,\,\cdot\,))(x))\cdot \bm{\nu}(x)=0
\end{equation}
for a.e.~$(t,x) \in (0,\infty) \times \partial\Omega$, where $\gamma:W^{1,r}(\Omega)^N \to W^{1-1/r,r}(\partial\Omega)^N$ denotes the standard trace operator.
Moreover, it holds that
\begin{equation}\label{solest}
\|(n-\overline{n}_0,c-(1-e^{-t})\overline{n}_0,\bm{u})\|_{\mathbb{E}} \le C(1+\|\nabla\varphi\|_{L^r(\Omega)})\llbracket (n_0,c_0,\bm{u}_0,\bm{f}) \rrbracket,
\end{equation}
where
\begin{equation}
\begin{aligned}
\|(\widehat{n},\widehat{c},\widehat{\bm{u}})\|_{\mathbb{E}} &\coloneqq \|e^{\lambda_1 t}\widehat{n}\|_{L^q(W^{2,r}(\Omega)) \cap W^{1,q}(L^r(\Omega))}+\|e^{\lambda_1 t}\widehat{c}\|_{L^q(W^{3,r}(\Omega)) \cap W^{1,q}(W^{1,r}(\Omega))} \\
&\quad +\|e^{\lambda_2 t}\widehat{\bm{u}}\|_{L^q(W^{2,r}(\Omega)) \cap W^{1,q}(L^r(\Omega))}
\end{aligned}
\end{equation}
and $C>0$ is a constant independent of $\varepsilon$, $n_0$, $c_0$, $\bm{u}_0$, $\varphi$, $\bm{f}$, $S$, $n$, $c$, and $\bm{u}$.

Suppose additionally that $(n^*,c^*,\bm{u}^*)$ is a global solution to \eqref{ABS} with the given data $(n_0,c_0,\bm{u}_0,\bm{f})$ replaced by $(n_0^*,c_0^*,\bm{u}_0^*,\bm{f}^*)$ such that
\begin{equation}
\llbracket (n_0^*,c_0^*,\bm{u}_0^*,\bm{f}^*) \rrbracket \le \varepsilon (1+\|S\|_{L^{\infty}(W^{1,r}(\Omega))}+\|\partial_tS\|_{L^q(L^r(\Omega))})^{-1}(1+\|\nabla\varphi\|_{L^r(\Omega)})^{-2}.
\end{equation}
Then the Lipschitz continuity of the solution mapping is obtained in the following sense:
\begin{equation}\label{sollip}
\begin{aligned}
&\|((n-\overline{n}_0)-(n^*-\overline{n}_0^*),c-(1-e^{-t})\overline{n}_0-\{c^*-(1-e^{-t})\overline{n}_0^*\},\bm{u}-\bm{u}^*)\|_{\mathbb{E}} \\
&\le C(1+\|\nabla\varphi\|_{L^r(\Omega)})\llbracket (n_0-n_0^*,c_0-c_0^*,\bm{u}_0-\bm{u}_0^*,\bm{f}-\bm{f}^*) \rrbracket.
\end{aligned}		
\end{equation}
\end{theo}

\begin{rema}\label{globalrema}
We make some comments on Theorem \ref{globalsol}.
\begin{enumerate}
\item The conditions $r>N$ and $q>2$ yield the embedding results $B_{r,q}^{2-2/q}(\Omega) \subset C(\overline{\Omega})$ and $B_{r,q}^{3-2/q}(\Omega) \subset C^1(\overline{\Omega})$ (see \cite{muramatu}*{Theorem 2}), and hence the initial data $(n_0,c_0,\bm{u}_0)$ necessarily satisfy $n_0 \in C(\overline{\Omega})$, $c_0 \in C^1(\overline{\Omega})$, and $\bm{u}_0 \in C(\overline{\Omega})^N$.
Since the solution $(n,c,\bm{u})$ belongs to the maximal regularity class \eqref{solclass}, we infer from the above embeddings and the usual property of the trace space \cite{lunardi}*{Corollary 1.14} that
\begin{equation}\label{solbound}
\left\{\begin{aligned}
e^{\lambda_1 t}(n-\overline{n}_0) &\in BUC([0,\infty); C(\overline{\Omega})), \\
e^{\lambda_1 t}\{c-(1-e^{-t})\overline{n}_0\} &\in BUC([0,\infty); C^1(\overline{\Omega})), \\
e^{\lambda_2 t}\bm{u} &\in BUC([0,\infty);C(\overline{\Omega})^N).
\end{aligned}\right.		
\end{equation}
Thus, we have the following \textit{exponential decay} for $n$ and $c$:
$$\|n(t,\,\cdot\,)-\overline{n}_0\|_{L^{\infty}(\Omega)}+\|c(t,\,\cdot\,)-(1-e^{-t})\overline{n}_0\|_{L^{\infty}(\Omega)}=O(e^{-\lambda_1 t}) \quad \text{as $t \to \infty$}.$$
Concerning the asymptotic behaviors of $\bm{u}$, we also obtain $\|\bm{u}(t,\,\cdot\,)\|_{L^{\infty}(\Omega)}=O(e^{-\lambda_2 t})$ as $t \to \infty$ under a suitable exponential time-weighted assumption for $\bm{f}$.
We remark that our result includes the case of $\lambda_2=0$; although we may not expect any decay properties of $\bm{u}$, global solutions would be constructed even under weaker conditions $\bm{f} \in L^q(0,\infty;L^r(\Omega)^N)$.

\item The assumption $1/r+2/q \neq 1$ is just due to the exclusion of the critical case; in the case of $1/r+2/q<1$, since $\nabla n_0 \in B_{r,q}^{1-2/q}(\Omega)^N$ and since $B_{r,q}^{1-2/q}(\Omega) \subset W^{1/r+\eta,r}(\Omega)$ for $0<\eta<1-2/q-1/r$ \cite{muramatu}*{Theorem 2}, we may apply the usual trace theorem \cite{grisvard}*{Theorem~1.5.1.2} to find the trace operator $\gamma:W^{1/r+\eta,r}(\Omega) \to L^r(\partial\Omega)$ so that $\gamma \nabla n_0$ makes sense.
Thus we have to assume the compatibility condition \eqref{addcompcond}.
In the case of $1/r+2/q>1$, such a trace operator $\gamma$ does not exist.

\item The tensor-valued function $S$ is motivated by the model including bacterial chemotaxis near surfaces that contain rotational components orthogonal to the signal gradient \cite{xueothmer}, where the typical choice of $S$ is 
\begin{equation}
S = a \begin{pmatrix} 1 & 0 \\ 0 & 1 \end{pmatrix} + b \begin{pmatrix} 0 & - 1 \\ 1 & 0 \end{pmatrix}, \qquad a > 0, \; b \in \mathbb{R}
\end{equation}
in the case of $N = 2$.
Clearly, if $S \equiv I$ (the identity matrix), then System \eqref{KSNS} reduces to the standard Keller--Segel--Navier--Stokes system with chemotactic cross-diffusion being directed to increasing signal concentrations.
Moreover, note that we may obtain unique global strong solutions \textit{without} assuming any smallness of $\varphi$ and $S$ by taking $n_0$, $c_0$, $\bm{u}_0$, and $\bm{f}$ sufficiently small depending on $\varphi$ and $S$. Such a relaxation of the smallness condition of $\varphi$ relies on the method of \cite{choe, takeuchi}.
\end{enumerate}
\end{rema}

Under the same assumption as in Theorem \ref{globalsol}, we may prove that the unique global strong solution $(n,c,\bm{u})$ to \eqref{ABS} constructed in Theorem \ref{globalsol} is indeed a \textit{classical solution} in the pointwise sense provided that the given functions $\varphi$, $\bm{f}$, and $S$ satisfy some additional regularity conditions.
As a by-product of such a result, we also obtain the \textit{non-negativity} of solutions $n$ and $c$, a fact that must be observed from the viewpoint of the biological model since $n$ and $c$ stand for the density and the concentration, respectively.

\begin{theo}[Regularities and non-negativity of solutions]\label{solreg}
Assume that all assumptions in Theorem \ref{globalsol} are satisfied.
Suppose additionally that
\begin{gather}\label{givenreg1}
\varphi \in C^{1+\theta_0}(\overline{\Omega}), \quad \bm{f} \in C((0,\infty);C^{\theta_0}(\overline{\Omega})^N), \\
\label{givenreg2}
S \in C((0,\infty);C^{1+\theta_0}(\overline{\Omega})^{N^2})
\end{gather}
for some $0<\theta_0<\min\{1,2-2/q-N/r\}$.
Then, the solution $(n,c,\bm{u})$ obtained in Theorem \ref{globalsol} also admits the regularities
\begin{equation}\label{ncureg}
\left\{\begin{aligned}
n &\in BUC([0,\infty);C(\overline{\Omega})) \cap C((0,\infty);C^{2+\theta}(K)) \cap C^1((0,\infty);C^{\theta}(K)), \\
c &\in BUC([0,\infty);C^1(\overline{\Omega})) \cap C((0,\infty);C^{4+\theta}(K)) \cap C^1((0,\infty);C^{2+\theta}(K)), \\
\bm{u} &\in BUC([0,\infty);C(\overline{\Omega})^N) \cap C((0,\infty);C^{2+\theta}(\overline{\Omega})^N) \cap C^1((0,\infty);C^{\theta}(\overline{\Omega})^N)
\end{aligned}\right.
\end{equation}
for all $0<\theta<\theta_0$, where $K \subset \Omega$ is an arbitrary compact subset.
Moreover, there holds
$$\int_{\Omega}n(t,x)\,dx=\int_{\Omega}n_0(x)\,dx, \quad \int_{\Omega}c(t,x)\,dx=e^{-t}\int_{\Omega}c_0(x)\,dx+(1-e^{-t})\int_{\Omega}n_0(x)\,dx$$
for all $0<t<\infty$.
In particular, if $N/r+2/q<1$ and the initial data $n_0$ and $c_0$ are non-negative in $\Omega$, then $n$ and $c$ are non-negative in $(0,\infty) \times \Omega$.
\end{theo}

\begin{rema}\label{regrema}
There are a few comments on Theorem \ref{solreg}.
\begin{enumerate}
\item The regularity result \eqref{ncureg} implies that the global solution $(n,c,\bm{u})$ constructed in Theorem \ref{globalsol} becomes a classical solution to \eqref{ABS} under the additional assumptions \eqref{givenreg1} and \eqref{givenreg2}.
Here, we will prove the regularity result \eqref{ncureg} by the standard bootstrap argument but it is \textit{not} so simple to show the smoothing effects acting on $n$ due to the nonlinearity of the boundary condition $\nabla n \cdot \bm{\nu}=nS(t,x)\nabla c \cdot \bm{\nu}$ on $\partial \Omega$.
To overcome this difficulty, we will introduce a suitable cut-off function to ignore the effects near the boundary $\partial\Omega$.
Then, extending the function on $\Omega$ to the function on $\mathbb{R}^N$, we may consider the integral form with the aid of the usual heat semigroup $e^{t\Delta}$ defined on $\mathbb{R}^N$ to prove that $n$ is regular in space and time.
This is the reason the compact subset $K \subset \Omega$ appears in \eqref{ncureg}.

\item In the last assertion of Theorem \ref{solreg}, the condition $N/r+2/q<1$ is mainly used to ensure that the nonlinear boundary condition $\nabla n \cdot \bm{\nu}=nS(t,x)\nabla c \cdot \bm{\nu}$ makes sense in $C(\partial\Omega)$.
To be more precise, as we have $B_{r,q}^{2-2/q}(\Omega) \subset C^1(\overline{\Omega})$ from \cite{muramatu}*{Theorem 2}, a similar argument to that in Remark \ref{globalrema} (i) provides that $\nabla n \in BUC([0,\infty);C^1(\overline{\Omega})^N)$.
Therefore, noting that $nS\nabla c \in C((0,\infty);C(\overline{\Omega})^N)$ because of \eqref{givenreg2} and \eqref{ncureg}, we observe that \eqref{VNBC} is satisfied in the $C(\partial\Omega)$-sense.
\end{enumerate}
\end{rema}

We now draw attention to what is the most pivotal contribution of our findings in Theorems \ref{globalsol} and \ref{solreg}. As noted earlier, while the global existence of solutions under smallness assumptions has been established in prior works such as \cite{CaoLankeit, YuWangZheng18}, our work not only reproduces such results but also ensures the \textit{uniqueness} of global solutions within an appropriate analytical framework. A further noteworthy aspect of our results lies in the fact that the global solutions that we construct are rigorously subject to the nonlinear boundary condition in the sense of \eqref{VNBC} (cf.~Remark~\ref{globalrema}~(ii)).
This achievement attains particular significance when juxtaposed with earlier contributions, such as those in \cite{CaoLankeit, YuWangZheng18}. For instance, the methodology employed in \cite{CaoLankeit} relies on solving a certain regularized problem involving homogeneous linear boundary conditions $\nabla n_{\eta}\cdot\bm{\nu}=0$ on $\partial\Omega$ instead of the nonlinear boundary condition. While \cite{CaoLankeit}*{Lemma 5.4} demonstrates the convergence of a sequence of approximate solutions to a weak solution (defined via test functions in $C_0^{\infty}(\Omega)$), the subsequent justification that such a weak solution is indeed classical primarily hinges on parabolic regularity arguments. However, the weak solution framework delineated in \cite{CaoLankeit}*{Definition 1} is devoid of explicit information regarding the nonlinear boundary condition, since it relies exclusively on test functions compactly supported in $\Omega$. Consequently, the results in \cite{CaoLankeit} leave unresolved the question of whether the nonlinear boundary condition $\nabla n\cdot\bm{\nu}=nS(t,x)\nabla c\cdot\bm{\nu}$ is certainly satisfied. In that sense, we believe that our
results seem to be the \textit{first ones} that provide the unique existence of global solutions that actually satisfy the nonlinear boundary conditions. As a trade-off for our framework necessitating more stringent assumptions compared to those in \cite{CaoLankeit, YuWangZheng18}, we emphasize that these conditions allow us to construct a unique regular global solution, thereby enabling a more robust characterization of the interplay between the nonlinear boundary conditions and the governing dynamics.

\subsection{Summary of strategies and plan of the paper}
\label{morerema}

This subsection is devoted to an overview of both the main difficulties in proving Theorems \ref{globalsol} and \ref{solreg} and our strategies for overcoming them, including a description of the structure of this paper.
Note that the notations and fundamental properties of embedding are summarized in Section \ref{sec-2}.

In contrast to the standard case $S \equiv I$, the main difficulty here is that System \eqref{ABS} contains the \textit{nonlinear} boundary condition $\nabla n \cdot \bm{\nu}=nS(t,x)\nabla c \cdot \bm{\nu}$ on $\partial \Omega$.
In fact, as seen in \eqref{ABS}, we may not directly formulate \eqref{KSNS} as an abstract Cauchy problem.
In this paper, we rely on some maximal regularity theorem to construct a strong solution to \eqref{ABS}, where, roughly speaking, maximal regularity means that \textit{each} term in the equation has the same regularity in space and time.
One of the novelties of this paper is that we establish the maximal regularity theorems for the Neumann--Laplacian in some subspace of $L^r (\Omega)$.
To be precise, it is well-known that the Neumann--Laplacian defined on $W^{2,r}_\nu (\Omega) = \{\psi \in W^{2,r}(\Omega) \mid \nabla \psi \cdot \bm{\nu} = 0 \text{ on $\partial \Omega$}\}$ contains 0 as its spectrum, and hence we may \textit{not} expect to observe that the Neumann--Laplacian has maximal regularity on the semi-infinite time interval $(0, \infty)$.
However, if we restrict the domain of the Neumann--Laplacian such that 
\begin{equation}
(W_{\nu}^{2,r} \cap L_0^r)(\Omega)=\left\{\psi \in W^{2,r}(\Omega) \, \left| \, \nabla \psi \cdot \bm{\nu}=0 \text{ on $\partial\Omega$}, \quad \int_{\Omega}\psi(x)\,dx=0 \right.\right\},
\end{equation}
then the Neumann--Laplacian generates an analytic $C_0$-semigroup of negative exponential type so that the Neumann--Laplacian has maximal regularity on $(0, \infty)$.
We also extend this result to the case of inhomogeneous boundary data (cf.~Theorem \ref{maxneu0} below), which is \textit{different} from the standard result on the maximal regularity theorem for the linear heat equation with an inhomogeneous Neumann boundary condition (cf.~Pr\"uss and Simonett \cite{prusssimonett}*{Theorem 6.3.2}).
In fact, unlike the case of \cite{prusssimonett}*{Theorem 6.3.2}, we see that the solution decays exponentially $O(e^{- \lambda t})$ as $t \to \infty$, where $\lambda$ is a positive number.
This assertion seems to be standard nowadays but we may not deduce the fact from a general theory of maximal regularity.
Hence, we will serve the detailed proof with the aid of a combination of the general theory of maximal regularity \cite{prusssimonett}*{Theorem 6.3.2} and the argument used in \cite{shibata}*{Section 4} since it is partly hard to find its proof in the common literature.
Moreover, we verify that the range of $\lambda$ may be given as $0 < \lambda < \min\{1, \lambda_{\mathrm{N}}(\Omega) \slash q\}$, where $\lambda_{\mathrm{N}}(\Omega)$ stands for the least upper bound of the Poincar\'e inequality \cite{galdi}*{Theorem II 5.4}:
$$\|\psi\|_{L^2(\Omega)} \le \sqrt{\lambda_{\mathrm{N}}(\Omega)}\|\nabla \psi\|_{L^2(\Omega)}$$
for all $\psi \in (W^{1,2} \cap L_0^2)(\Omega)$.
Here, $q \in (1, \infty)$ is an index of the time integrability of the solution.
We shall remark that the same method yields an exponential decay of solutions to the usual Stokes system as well, provided that given external forces satisfy exponential time-weighted assumptions.
The aforementioned results are recorded in Section \ref{sec-3} with their proofs.
In addition, the fundamental property of the Neumann--Laplacian is summarized in Proposition \ref{neumann0} below.
We also prepare the maximal regularity result for a shifted Neumann heat equation in the same section.

In Section \ref{sec-4}, we prove Theorem \ref{globalsol} with a combination of the maximal regularity results given in Section \ref{sec-3}.
Since we construct a strong solution to \eqref{ABS} directly, we need to suppose that the given data are small, which is different from the existence results on weak solutions to \eqref{ABS} established by Wang \cite{Wang17} and Ke and Zheng \cite{KeZheng19} (cf.~\cite{Winkler18, WangWinklerXiang18}).
However, in contrast to \cite{Wang17, KeZheng19, Winkler18, WangWinklerXiang18}, we may obtain the \textit{uniqueness} of the solution.
To close the a priori estimate, the space-time $L^{\infty}$-boundedness of the nonlinear terms are crucial; this leads us to suppose the conditions $r>N$ and $q>2$ since these restrictions imply the regularities \eqref{solbound} of solutions $(n,c,\bm{u})$.
Namely, those restrictions stem from the nonlinear problem rather than linear analyses.
Note that the second author \cite{watanabe} showed the stability of the constant stationary solutions $(\overline{n}_0,0,0)$ to \eqref{ABS} in the case that $S \equiv I$, $\bm{f} \equiv 0$, and the second equation of \eqref{ABS} is replaced by \eqref{consumption}.
The main ingredient of the analysis in \cite{watanabe} is the maximal regularity theorem as well, but System \eqref{ABS} contains the nonlinear boundary condition $\nabla n \cdot \bm{\nu}=nS(t,x)\nabla c \cdot \bm{\nu}$, which induces difficulty. Hence, our method essentially \textit{differs} from that of \cite{watanabe}.

Section \ref{sec-5} is devoted to the proof of Theorem \ref{solreg}.
Note that the regularities \eqref{ncureg} of $(n,c,\bm{u})$ are sufficient to verify that $(n,c,\bm{u})$ are classical solutions to \eqref{ABS}.
However, at least in our result, we do not know whether \eqref{ncureg} holds for a large $\theta>0$ even if $\varphi$, $\bm{f}$, and $S$ are smooth.
The regularities of $\bm{u}$ may be verified by the classical semigroup theory with the aid of the mapping properties of the Stokes semigroup $e^{- tA}$ in the Besov spaces, whose proof is given in Appendix \ref{ap-A}.
However, in contrast to proving the additional regularity for $\bm{u}$, some ingenuity is required to show the additional regularity of solutions $n$ and $c$ due to the nonlinearity of the boundary condition.
In fact, as it has been stated in Remark \ref{regrema} (i), the additional regularity of solutions $n$ and $c$ may be obtained by the standard bootstrap argument, but we have to pay attention to the nonlinear boundary conditions appearing in \eqref{ABS}.
Here, in general, we may expect that solutions obtained by the maximal regularity theorem (cf.~Theorem \ref{globalsol}) are real-analytic in space and time due to the parameter trick introduced by Angenent \cite{angenent}, see, e.g., \cite{prusssimonett}*{Section 9.4} and \cite{denk}*{Section 7.2} for its details.
However, at least in our situation, it seems to be \textit{not} so simple to apply the parameter trick. In fact, the parameter trick is essentially based on the implicit function theorem in Banach spaces together with a scaling argument with respect to the variables $x$ and $t$.
In our case, there are two main difficulties:
The first one is that the function $n$ should satisfy the condition $\int_\Omega n (t,x) \, dx = \int_\Omega n_0 (x) \, dx$ so that it seems to be difficult to introduce a suitable coordinate transform that possesses the mean value condition on $n$.
Here, a similar problem arises in the divergence-free condition on $\bm{u}$ as well.
The second difficulty stems from the maximal regularity theorem of the heat equation with a Neumann boundary condition (cf.~Theorem \ref{maxneu0} below).
In fact, it seems to be difficult to apply this theorem to verify an isomorphism assumption that is required in the implicit function theorem.
This is due to the condition that, in Theorem \ref{maxneu0} below, each term on the right-hand side of \eqref{nheat} has to be given as the divergence form $- \nabla \cdot \bm{F}_{\mathrm{B}}$ or satisfy the mean value condition $\int_\Omega F_{\mathrm{E}}(t,x) \, dx = 0$.
Hence, the analyticity of the solution is a difficult problem even with respect to the time variable $t$ only when we adapt Angenent's parameter trick \cite{angenent}.

To avoid the aforementioned difficulty, we rely on the smoothing effects of the heat semigroup $e^{t \Delta}$ on $\mathbb{R}^N$.
To reformulate the equations as integral forms, we will use the cut-off technique to ignore the boundary conditions.
In fact, we take $0<\delta<1$ sufficiently small and introduce the cut-off function $\chi_{\delta} \in C_0^{\infty}(\Omega)$ by setting
\begin{equation}\label{chidef}
\chi_{\delta}(x) \coloneqq \left\{\begin{array}{cl}
1 & \displaystyle \text{if } \inf_{y \in \partial\Omega}|x-y| \ge \delta, \\
0 & \displaystyle \text{if } \inf_{y \in \partial\Omega}|x-y| \le \delta/2.		
\end{array}\right.		
\end{equation}
Here, we note that for all $\psi \in C(\overline{\Omega})$, there holds $\chi_{\delta}(x)\psi(x)=0$ if $x \in \partial\Omega$, which means that we may ignore the aforementioned nonlinear boundary condition.
Then, extending the unknowns, which are multiplied by $\chi_{\delta}$, to $\mathbb{R}^N$, we may have the integral system; see \eqref{INT} below.
Together with the mapping properties of the heat semigroup $e^{t \Delta}$ on $\mathbb{R}^N$ and the bootstrap argument, we may obtain the additional regularities of global strong solutions in space and time.
After verifying that the solution $(\bm{u}, n, c)$ is a classical solution to \eqref{ABS}, we will prove the non-negativity of $n$ and $c$, which is required in light of biological models.
Note again that System \eqref{ABS} includes the nonlinear boundary condition $\nabla n\cdot\bm{\nu}=nS(t,x)\nabla c\cdot\bm{\nu}$ on $\partial \Omega$, we may \textit{not} expect to show $n, c \ge 0$ just relying on the standard maximum principle for the parabolic problems.
Regarding this problem, we introduce the maximum principle for the parabolic problems of a new type, which may be regarded as a variant assertion of Proposition 52.8 and Remark 52.9 in \cite{quittner}.
The proof is given in Appendix \ref{ap-B}.

\section{Preliminaries}
\label{sec-2}
\subsection{Function spaces}

In this subsection, we give the definitions of function spaces. We begin with classical function spaces:
Let $C(\Omega)$ be the set of all continuous functions on $\Omega$.
We also set $C^k(\Omega) \coloneqq \{\psi \in C(\Omega) \mid \partial_x^{\alpha}\psi \in C(\Omega), \, |\alpha| \le k\}$ for $k \in \mathbb{N}$.
Note that if $\Omega$ is replaced by $K$, where $K$ is $K=\overline{\Omega}$ or the compact subset $K \subset \Omega$, then $C(K)$ and $C^k(K)$ become the Banach spaces with the norms $\|\,\cdot\,\|_{L^{\infty}(K)}$ and $\|\,\cdot\,\|_{C^k(K)} \coloneqq \sum_{|\alpha| \le k}\|\partial_x^{\alpha}\,\cdot\,\|_{L^{\infty}(K)}$, respectively.
We next introduce the H\"older spaces:
Let $0<\theta<1$.
Then the H\"older spaces are given by $C^{\theta}(\overline{\Omega}) \coloneqq \{\psi \in C(\overline{\Omega}) \mid \|\psi\|_{C^{\theta}(\overline{\Omega})} <\infty\}$ with the norm
$$\|\psi\|_{C^{\theta}(\overline{\Omega})} \coloneqq \|\psi\|_{L^{\infty}(\Omega)}+\sup_{x,y \in \Omega, \, x \neq y}\frac{|\psi(x)-\psi(y)|}{|x-y|^{\theta}}.$$
We also set $C^{k+\theta}(\overline{\Omega}) \coloneqq \{\psi \in C^k(\overline{\Omega}) \mid \|\psi\|_{C^{k+\theta}(\overline{\Omega})} \coloneqq \|\psi\|_{C^k(\overline{\Omega})}+\sum_{|\alpha|=k}\|\partial_x^{\alpha}\psi\|_{C^{\theta}(\overline{\Omega})}<\infty\}$ for $k \in \mathbb{N}$.

Concerning the Besov spaces $B_{r,q}^s(\Omega)$, we define via the real interpolation method; let $r,q \in [1,\infty]$ and $0<s<\infty$.
Then the Besov spaces are defined by $B_{r,q}^s(\Omega) \coloneqq (L^r(\Omega),W^{k,r}(\Omega))_{s/k,q}$, where $k \in \mathbb{N}$ is any number satisfying $k>s$ and $(\,\cdot\,,\,\cdot\,)_{\theta,q}$ denotes the real interpolation functor.
For the definition of the real interpolation spaces, we refer to, e.g., \cite{triebel}*{Chapter 1} and \cite{lunardi}*{Chapter 1}.
Note that the basic property of the real interpolation spaces \cite{lunardi}*{Proposition 1.3} implies the continuous embeddings
$$W^{k,r}(\Omega) \subset B_{r,q}^s(\Omega) \subset L^r(\Omega)$$
for all $r,q \in [1,\infty]$, $0<s<\infty$, and $k \in \mathbb{N}$ such that $k>s$.
We also have the following continuous embeddings
$$B_{r,\infty}^{s_1}(\Omega) \subset B_{r,q_0}^{s_0}(\Omega) \subset B_{r,q_1}^{s_0}(\Omega)$$
for all $1 \le r \le \infty$, $0<s_0<s_1<\infty$, and $1 \le q_0 \le q_1 \le \infty$ due to Proposition 1.4 and Theorem 1.6 in \cite{lunardi}.
Furthermore, in the case of the bounded domain $\Omega$, since $W^{k,\infty}(\Omega) \subset W^{k,r}(\Omega)$ for all $1 \le r \le \infty$ and $k \in \mathbb{N}_0$, we see by \cite{lunardi}*{Theorem 1.6} that $B_{\infty,q}^s(\Omega) \subset B_{r,q}^s(\Omega)$ for all $r,q \in [1,\infty]$ and $0<s<\infty$ as well.
We remark that there are several ways to define the Besov spaces on domains $\Omega$; see, e.g., \cite{muramatu}, \cite{triebel}*{Chapter 4}, and \cite{amann}*{Section 2}, but most of them are equivalent to each other.

We next define the $X$-valued function spaces:
Let $X$ be a Banach space.
Then $L^q(0,\infty;X)$ for $1 \le q \le \infty$ denote the Lebesgue--Bochner spaces with the norm $\|\,\cdot\,\|_{L^q(X)} \coloneqq \|\|\,\cdot\,\|_X\|_{L^q(0,\infty)}$.
We also introduce $W^{1,q}(0,\infty;X)$ similarly.
Let $C((0,\infty);X)$ denote the set of all continuous $X$-valued functions on $\Omega$ and let $C^1((0,\infty);X) \coloneqq \{F \in C((0,\infty);X) \mid \partial_tF \in C((0,\infty);X)\}$.
In addition, $BUC([0,\infty);X)$ denotes the closed subspace of $L^{\infty}(0,\infty;X)$ whose function is bounded uniformly continuous.

Finally, for $r,q \in (1,\infty)$ and $k=0,1$, we introduce the following notations, which imply the usual maximal regularity classes:
\begin{equation}
\begin{aligned}
\mathbb{E}_{q,r}^k &\coloneqq L^q(0,\infty;W^{2+k,r}(\Omega)) \cap W^{1,q}(0,\infty;W^{k,r}(\Omega)), \\
\|\,\cdot\,\|_{\mathbb{E}_{q,r}^k} &\coloneqq \|\,\cdot\,\|_{L^q(W^{2+k,r}(\Omega)) \cap W^{1,q}(W^{k,r}(\Omega))} \coloneqq \|\,\cdot\,\|_{L^q(W^{2+k,r}(\Omega))}+\|\,\cdot\,\|_{W^{1,q}(W^{k,r}(\Omega))}.
\end{aligned}
\end{equation}
We also introduce
\begin{equation}
\begin{aligned}
\mathbb{E}_{q,r}^0(-\Delta_{\mathrm{N}}) &\coloneqq L^q(0,\infty;(W^{2,r} \cap L_0^r)(\Omega)) \cap W^{1,q}(0,\infty;L_0^r(\Omega)), \\
\mathbb{E}_{q,r}^1(1-\Delta_{\mathrm{N},1}) &\coloneqq L^q(0,\infty;W_{\nu}^{3,r}(\Omega)) \cap W^{1,q}(0,\infty;W^{1,r}(\Omega)), \\
\mathbb{E}_{q,r}^0(A) &\coloneqq L^q(0,\infty;D(A)) \cap W^{1,q}(0,\infty;L_{\sigma}^r(\Omega)).
\end{aligned}
\end{equation}

\subsection{Fundamental propositions}

We shall verify the fundamental propositions which will be frequently used in this paper.
In the following, we suppose that $\Omega \subset \mathbb{R}^N$, $N\ge 2$, is a bounded smooth domain.
We first recall the embedding properties of the Besov spaces, which give the relations with the classical spaces like the Sobolev and H\"older spaces:

\begin{prop}\label{besovemb}
Let $r,q \in [1,\infty]$ and $0<s<\infty$.
If $k \in \mathbb{N}_0$ satisfies $k<s$, then the continuous embedding $B_{r,q}^s(\Omega) \subset W^{k,r}(\Omega)$ holds.
In addition, if $s_0 \notin \mathbb{N}$ satisfies $0<s_0 \le s-N/r$, then the continuous embedding $B_{r,q}^s(\Omega) \subset C^{s_0}(\overline{\Omega})$ holds.
\end{prop}

For the proof, we refer to Muramatu \cite{muramatu}*{Theorem 2}.
We next recall a characterization of the trace space \cite{lunardi}*{Corollary 1.14}, which ensures that the maximal regularity class naturally includes the space of the initial data.

\begin{prop}\label{traceemb}
Let $r,q \in (1,\infty)$ and assume that $F_k \in \mathbb{E}_{q,r}^k$ for $k=0,1$.
It holds that $F_k \in BUC([0,\infty);B_{r,q}^{2+k-2/q}(\Omega))$ for $k=0,1$ with the estimate
\begin{equation}
\|F_k\|_{L^{\infty}(B_{r,q}^{2+k-2/q}(\Omega))} \le C\|F_k\|_{\mathbb{E}_{q,r}^k},
\end{equation}
where $C>0$ is a constant independent of $F_k$.
\end{prop}

Note that, by combining Propositions \ref{besovemb} and \ref{traceemb}, we observe that the conditions $r>N$ and $q>2$ in Theorem \ref{globalsol} ensure the following embeddings of initial data $n_0$ and solutions $(n,c,\bm{u})$ to \eqref{ABS}.

\begin{corr}\label{ncuemb}
Let $N<r<\infty$ and $2<q<\infty$.
Let $\overline{n}_0 \coloneqq |\Omega|^{-1}\int_{\Omega}n_0(x)\,dx$ be the mean value of $n_0$.
Then the following statements hold:
\begin{enumerate}
\item Assume that $n_0 \in B_{r,q}^{2-2/q}(\Omega)$.
Then there holds
$$|\overline{n}_0|+\|\overline{n}_0\|_{B_{r,q}^{2-2/q}(\Omega)} \le C\|n_0\|_{B_{r,q}^{2-2/q}(\Omega)},$$
where $C>0$ is a constant independent of $n_0$.

\item Assume that $n \in \mathbb{E}_{q,r}^0$ and $c \in \mathbb{E}_{q,r}^1$.
Then there hold
$$\|n\|_{L^{\infty}(W^{1,r}(\Omega))} \le C\|n\|_{\mathbb{E}_{q,r}^0}, \quad \|c\|_{L^{\infty}(W^{2,r}(\Omega))} \le C\|c\|_{\mathbb{E}_{q,r}^1},$$
where $C>0$ is a constant independent of $n$ and $c$.
In addition, it holds that
$$n \in BUC([0,\infty);C(\overline{\Omega})), \quad c \in BUC([0,\infty);C^1(\overline{\Omega})).$$
\end{enumerate}
\end{corr}

\begin{proof}
\textbf{(i)} Since Proposition \ref{besovemb} yields $B_{r,q}^{2-2/q}(\Omega) \subset C(\overline{\Omega})$ due to $r>N$ and $q>2$, we have
\begin{equation}
|\overline{n}_0| \le \frac{1}{|\Omega|}\int_{\Omega}|n_0(x)|\,dx \le \|n_0\|_{L^{\infty}(\Omega)} \le C\|n_0\|_{B_{r,q}^{2-2/q}(\Omega)}.
\end{equation}
In addition, since $\overline{n}_0$ is the constant, we see by $W^{2,r}(\Omega) \subset B_{r,q}^{2-2/q}(\Omega) \subset L^r(\Omega)$ that
\begin{equation}
\|\overline{n}_0\|_{B_{r,q}^{2-2/q}(\Omega)} \le C\|\overline{n}_0\|_{W^{2,r}(\Omega)} \le C\|\overline{n}_0\|_{L^r(\Omega)} \le C|\overline{n}_0|\|1\|_{L^r(\Omega)} \le C\|n_0\|_{B_{r,q}^{2-2/q}(\Omega)}.
\end{equation}

\noindent
\textbf{(ii)} Proposition \ref{besovemb} with the condition $q>2$ implies that $B_{r,q}^{2-2/q}(\Omega) \subset W^{1,r}(\Omega)$ and $B_{r,q}^{3-2/q}(\Omega) \subset W^{2,r}(\Omega)$.
Thus we have the desired result from Proposition \ref{traceemb} and the condition $r>N$.
\end{proof}

Finally, we verify the basic properties due to the Gauss divergence theorem.

\begin{prop}\label{gaussdiv}
Let $1<r<\infty$ and assume that $\psi \in W^{2,r}(\Omega)$.
If $\bm{\psi}_{\mathrm{B}} \in W^{1,r}(\Omega)^N$ satisfies $\nabla \psi\cdot\bm{\nu}=\bm{\psi}_{\mathrm{B}}\cdot\bm{\nu}$ on $\partial\Omega$, then there holds
$$\int_{\Omega}\Delta \psi(x)\,dx=\int_{\Omega}\nabla \cdot \bm{\psi}_{\mathrm{B}}(x)\,dx.$$
In addition, if $\bm{\psi}_{\sigma} \in D(A)$, then there holds
$$\int_{\Omega}(\bm{\psi}_{\sigma}\cdot\nabla \psi)(x)\,dx=0.$$
\end{prop}

\begin{proof}
The Gauss divergence theorem implies that
$$\int_{\Omega}\Delta \psi(x)\,dx=\int_{\partial\Omega}(\nabla\psi \cdot \bm{\nu})(x)\,d\sigma(x)=\int_{\partial\Omega}(\bm{\psi}_{\mathrm{B}}\cdot\bm{\nu})(x)\,d\sigma(x)=\int_{\Omega}\nabla \cdot \bm{\psi}_{\mathrm{B}}(x)\,dx.$$
In addition, since $\nabla \cdot \bm{\psi}_{\sigma}=0$ in $\Omega$ and $\bm{\psi}_{\sigma}=0$ on $\partial\Omega$, by applying the Gauss divergence theorem again, we have
$$\int_{\Omega}(\bm{\psi}_{\sigma}\cdot\nabla \psi)(x)\,dx=\int_{\Omega}\nabla \cdot (\psi\bm{\psi}_{\sigma})(x)\,dx=\int_{\partial\Omega}(\psi\bm{\psi}_{\sigma} \cdot \bm{\nu})(x)\,d\sigma(x)=0,$$
which proves the claim.
\end{proof}

\section{Maximal regularity theorems for the linearized equations}
\label{sec-3}

In this section, we describe maximal regularity theorems of the linear heat equation with a Neumann boundary condition and the Stokes system.

\subsection{Remark on the heat equation with a Neumann boundary condition}

The main purpose of this subsection is to show the following maximal regularity theorem for the linear heat equation, which is \textit{not} a simple consequence of the usual maximal regularity theory; compare our claim with Proposition \ref{maxneu} below.

\begin{theo}\label{maxneu0}
Let $r,q \in (1,\infty)$ satisfy $1/r+2/q \neq 1$ and let $0 \le \lambda<\lambda_{\mathrm{N}}(\Omega)/q$, where $\lambda_{\mathrm{N}}(\Omega)>0$ is given by \eqref{lamdef}.
Assume that $F_0$, $F_{\mathrm{E}}$, and $\bm{F}_{\mathrm{B}}$ satisfy
\begin{equation}\label{Fcond}
\left\{\begin{gathered}
F_0 \in (B_{r,q}^{2-2/q} \cap L_0^r)(\Omega), \quad e^{\lambda t}F_{\mathrm{E}} \in L^q(0,\infty;L_0^r(\Omega)), \\
e^{\lambda t}\bm{F}_{\mathrm{B}} \in  L^q(0,\infty;W^{1,r}(\Omega)^N) \cap W^{1,q}(0,\infty;L^r(\Omega)^N).
\end{gathered}\right.
\end{equation}
Furthermore, if $1/r+2/q<1$, assume that $F_0$ and $\bm{F}_{\mathrm{B}}$ satisfy the additional condition $\nabla F_0 \cdot \bm{\nu}=\bm{F}_{\mathrm{B}}(0,\,\cdot\,) \cdot \bm{\nu}$ on $\partial\Omega$.
Then there exists a unique global strong solution $U$ to the equations
\begin{equation}\label{nheat}
\left\{\begin{aligned}
\partial_tU-\Delta U &=-\nabla \cdot \bm{F}_{\mathrm{B}}+F_{\mathrm{E}}, & t>0, \, x &\in \Omega, \\
\nabla U \cdot \bm{\nu} &=\bm{F}_{\mathrm{B}} \cdot \bm{\nu}, & t>0, \, x &\in \partial\Omega, \\
U (0,x) &= F_0(x), & x &\in \Omega
\end{aligned}\right.
\end{equation}
such that $e^{\lambda t}U \in \mathbb{E}_{q,r}^0(-\Delta_{\mathrm{N}})$. The boundary condition is satisfied in the sense that
\begin{equation}\label{VBC}
(\gamma(\nabla U(t,\,\cdot\,)-\bm{F}_{\mathrm{B}}(t,\,\cdot\,))(x))\cdot \bm{\nu}(x)=0
\end{equation}
for a.e.~$(t,x) \in (0,\infty) \times \partial\Omega$. Moreover, it holds that
\begin{equation}\label{nheatest}
\|e^{\lambda t}U\|_{\mathbb{E}_{q,r}^0} \le C\left(\|F_0\|_{B_{r,q}^{2-2/q}(\Omega)}+\|e^{\lambda t}F_{\mathrm{E}}\|_{L^q(L^r(\Omega))}+\|e^{\lambda t}\bm{F}_{\mathrm{B}}\|_{L^q(W^{1,r}(\Omega)) \cap W^{1,q}(L^r(\Omega))}\right),
\end{equation}
where $C>0$ is a constant independent of $F_0$, $F_{\mathrm{E}}$, $\bm{F}_{\mathrm{B}}$, and $U$.
\end{theo}

Theorem \ref{maxneu0} is different from the standard result on the maximal regularity theorem for the linear heat equation with an inhomogeneous Neumann boundary condition (cf.~Pr\"uss and Simonett \cite{prusssimonett}*{Theorem 6.3.2}) since the left-hand side is given by $\partial_tU-\Delta U$ and since the solution to \eqref{nheat} decays exponentially as $t \to \infty$ by taking $\lambda>0$.
Indeed, recall that the Neumann--Laplacian defined on $L^r (\Omega)$ with the domain $W_{\nu}^{2,r} (\Omega)$ does \textit{not} admit 0 as its resolvent, but if one replaces $L^r (\Omega)$ with $L^r_0 (\Omega)$, then 0 becomes the resolvent of the Neumann--Laplacian.
Hence, one may expect to observe that the Neumann--Laplacian admits maximal regularity on the semi-infinite time interval $(0, \infty)$.
Although this fact seems to be standard, the proof of Theorem \ref{maxneu0} is partly hard to find in the common literature, and hence, we shall give its proof by following the approach due to Shibata \cite{shibata}*{Section 4} together with the general theory for maximal regularity \cite{prusssimonett}*{Theorem 6.3.2} in what follows.
Note that the crucial point here is that the domain of the Neumann--Laplacian is $(W_{\nu}^{2,r} \cap L_0^r)(\Omega)$ instead of $W_{\nu}^{2,r} (\Omega)$, which is completely different from the operator $-\Delta_{\mathrm{N},1}$ introduced in Section \ref{sec-1}.
Indeed, this replacement induces the \textit{invertibility} of the Neumann--Laplacian in $L^r_0 (\Omega)$. To this end, we first prepare the following proposition.

\begin{prop}\label{neumann0}
Let $1<r<\infty$ and define the operator $-\Delta_{\mathrm{N},0} \coloneqq -\Delta$ with the domain
$$D(-\Delta_{\mathrm{N},0}) \coloneqq (W_{\nu}^{2,r} \cap L_0^r)(\Omega)=\left\{\psi \in W^{2,r}(\Omega) \, \left| \, \nabla \psi \cdot \bm{\nu}=0 \text{ on $\partial\Omega$}, \quad \int_{\Omega}\psi(x)\,dx=0 \right.\right\}.$$
Then the following statements hold:
\begin{enumerate}
\item The resolvent set $\rho(\Delta_{\mathrm{N},0}) \subset \mathbb{C}$ of $\Delta_{\mathrm{N},0}$ satisfies
$$[0,\infty) \subset \mathbb{C} \setminus (-\infty,-\lambda_{\mathrm{N}}(\Omega)] \subset \rho(\Delta_{\mathrm{N},0}),$$
where $\lambda_{\mathrm{N}}(\Omega)>0$ is given by \eqref{lamdef}.

\item The domain $D(-\Delta_{\mathrm{N},0})$ is dense in $L_0^r(\Omega)$.
\end{enumerate}
\end{prop}

\begin{proof}
Although the assertion described in Proposition \ref{neumann0} seems to be standard, we will give a detailed proof for the convenience of the readers.

\noindent
\textbf{(i)} Let $\xi \in L_0^r(\Omega)$.
We aim to show that there exists a unique $\psi \in D(-\Delta_{\mathrm{N},0})$ such that $\lambda \psi-\Delta_{\mathrm{N},0}\psi=\xi$.
To this end, we first take $R \ge 1$ sufficiently large and apply \cite{grisvard}*{Theorem 2.4.1.3} to obtain a unique $\psi \in W_{\nu}^{2,r}(\Omega)$ such that $R\psi-\Delta \psi =\xi$.
Since $\xi \in L_0^r(\Omega)$ and $\psi \in W_{\nu}^{2,r}(\Omega)$, Proposition \ref{gaussdiv} yields $R\int_{\Omega}\psi(x)\,dx=0$.
Thus, we have $\psi \in D(-\Delta_{\mathrm{N},0})$ due to $R \neq 0$, so there holds $R \in \rho(\Delta_{\mathrm{N},0})$.

We next consider the case of $\lambda \in \mathbb{C} \setminus (-\infty,-\lambda_{\mathrm{N}}(\Omega)]$.
We show the injectivity of $\lambda-\Delta_{\mathrm{N},0}$:
Suppose that $\psi \in D(-\Delta_{\mathrm{N},0})$ satisfies $\lambda \psi-\Delta_{\mathrm{N},0}\psi=0$.
Since the spectrum of $\Delta_{\mathrm{N},0}$ is independent of $r$ due to the boundedness of $\Omega$, it suffices to consider the case of $r=2$.
Then, since $\nabla \psi \cdot \bm{\nu}=0$ on $\partial\Omega$, by multiplying $\psi$ and integrating over $\Omega$, we have
\begin{equation}
0=\lambda \int_{\Omega}|\psi(x)|^2\,dx-\int_{\Omega}(\psi\Delta_{\mathrm{N},0} \psi)(x)\,dx=\lambda \int_{\Omega}|\psi(x)|^2\,dx+\int_{\Omega}|\nabla \psi(x)|^2\,dx.
\end{equation}
Hence, there holds $\mathrm{Im} \, \lambda \|\psi\|_{L^2(\Omega)}^2=0$ and $\mathrm{Re} \, \lambda \|\psi\|_{L^2(\Omega)}^2+\|\nabla \psi\|_{L^2(\Omega)}^2=0$.
Here, noting that $\int_{\Omega}\psi(x)\,dx=0$, we see by the Poincar\'e inequality \cite{galdi}*{Theorem II 5.4} that
\begin{equation}
\begin{aligned}
0 &=\mathrm{Re} \, \lambda \|\psi\|_{L^2(\Omega)}^2+\|\nabla \psi\|_{L^2(\Omega)}^2 \\ 
& \ge \left(\mathrm{Re} \, \lambda +\inf_{\psi_* \in (W^{1,2} \cap L_0^2)(\Omega) \setminus \{0\}}\frac{\|\nabla\psi_*\|_{L^2(\Omega)}^2}{\|\psi_*\|_{L^2(\Omega)}^2}\right)\|\psi\|_{L^2(\Omega)}^2 \\
&=(\mathrm{Re} \, \lambda + \lambda_{\mathrm{N}}(\Omega))\|\psi\|_{L^2(\Omega)}^2,
\end{aligned}
\end{equation}
which yields $\psi=0$ due to $\mathrm{Im} \, \lambda \neq 0$ or $\mathrm{Re} \, \lambda>-\lambda_{\mathrm{N}}(\Omega)$.

Finally, we show the surjectivity of $\lambda-\Delta_{\mathrm{N},0}$:
Since we have observed that $R \in \rho(\Delta_{\mathrm{N},0})$, the operator $\mathcal{L}_R \coloneqq (R-\Delta_{\mathrm{N},0})^{-1}:L_0^r(\Omega) \to D(-\Delta_{\mathrm{N},0}) \subset L_0^r(\Omega)$ is bounded.
Note that since $\Omega$ is bounded and since $D(-\Delta_{\mathrm{N},0}) \subset W^{1,r}(\Omega)$, the Rellich--Kondrachov theorem \cite{brezis}*{Theorem 9.16} ensures that $\mathcal{L}_R$ is a compact operator.
In addition, we have shown that $\lambda-\Delta_{\mathrm{N},0}$ is injective for $\lambda \in \mathbb{C} \setminus (-\infty,-\lambda_{\mathrm{N}}(\Omega)]$, so $(\lambda-\Delta_{\mathrm{N},0})\mathcal{L}_R$ is injective as well.
Therefore, by the Fredholm alternative theorem \cite{brezis}*{Theorem 6.6 (c)} and the identity $(\lambda-\Delta_{\mathrm{N},0})\mathcal{L}_R=1+(\lambda-R)\mathcal{L}_R$, we see that given $\xi \in L_0^r(\Omega)$, we may find $\widetilde{\psi} \in L_0^r(\Omega)$ so that $(\lambda-\Delta_{\mathrm{N},0})\mathcal{L}_R\widetilde{\psi}=\xi$.
Hence, we have $\psi \coloneqq \mathcal{L}_R\widetilde{\psi} \in D(-\Delta_{\mathrm{N},0})$ and $\lambda \psi-\Delta_{\mathrm{N},0}\psi=\xi$, which yields $\mathbb{C} \setminus (-\infty,-\lambda_{\mathrm{N}}(\Omega)] \subset \rho(\Delta_{\mathrm{N},0})$.

\noindent
\textbf{(ii)} Suppose that $\psi \in L_0^r(\Omega)$.
Then we may take a sequence $\{\psi_j\}_{j \in \mathbb{N}} \subset C_0^{\infty}(\Omega)$ of functions satisfying $\|\psi-\psi_j\|_{L^r(\Omega)} \to 0$ as $j \to \infty$.
By setting
$$\widetilde{\psi}_j \coloneqq \psi_j-\frac{1}{|\Omega|}\int_{\Omega}\psi_j(x)\,dx, \quad j \in \mathbb{N},$$
we have $\{\widetilde{\psi}_j\}_{j \in \mathbb{N}} \subset D(-\Delta_{\mathrm{N},0})$.
Noting that $\int_{\Omega}\psi(x)\,dx=0$, we see by the H\"older inequality that
\begin{equation}
\|\widetilde{\psi}_j-\psi\|_{L^r(\Omega)} \le \|\psi_j-\psi\|_{L^r(\Omega)}+\frac{1}{|\Omega|}\int_{\Omega}|\psi_j(x)-\psi(x)|\,dx \cdot \|1\|_{L^r(\Omega)} \le 2\|\psi_j-\psi\|_{L^r(\Omega)},
\end{equation}
which yields the desired result.
\end{proof}

By virtue of Proposition \ref{neumann0}, we see that the Neumann--Laplacian $-\Delta_{\mathrm{N},0}:D(-\Delta_{\mathrm{N},0}) \to L_0^r(\Omega)$ generates a bounded analytic $C_0$-semigroup $e^{t\Delta_{\mathrm{N},0}}: L_0^r(\Omega) \to L_0^r(\Omega)$ such that
\begin{equation}\label{semiest}
\|e^{t\Delta_{\mathrm{N},0}}\psi\|_{L^r(\Omega)} \le Ce^{-\lambda_{\mathrm{N}}(\Omega)t}\|\psi\|_{L^r(\Omega)}
\end{equation}
for all $0<t<\infty$ and $\psi \in L_0^r(\Omega)$, where $C>0$ is a constant independent of $t$ and $\psi$. In fact, since Proposition \ref{neumann0} implies that $(-\infty,0) \subset \rho(-\lambda_{\mathrm{N}}(\Omega)-\Delta_{\mathrm{N},0})$, the well-known semigroup theory (see \cite{prusssimonett}*{Theorem 3.3.2}) ensures that $e^{(\lambda_{\mathrm{N}}(\Omega)+\Delta_{\mathrm{N},0})t}$ is uniformly bounded, which yields \eqref{semiest}.
It should be noted that although this consequence \eqref{semiest} has already been given by Winkler \cite{Winkler10}*{Lemma 1.3 (i)}, another approach by Proposition \ref{neumann0} is taken here so that the resolvent set $\rho(\Delta_{\mathrm{N},0})$ is characterized in detail.
Using $e^{t\Delta_{\mathrm{N},0}}$, we next show the following proposition
by simple calculations.
Note that the uniqueness assertion will play an important role in the proof of Theorem \ref{maxneu0} since it provides a \textit{gain} of regularities of the function $F$.
In addition, noting that a similar estimate to \eqref{semiest} is valid for the Stokes operator $e^{-tA}$, we give the corresponding result in the case of $e^{-tA}$ as well.

\begin{prop}\label{ipsipro}
Let $r,q \in (1,\infty)$ and let $\lambda_{\mathrm{N}}(\Omega),\lambda_{\mathrm{D}}(\Omega) \in (0,\infty)$ be given by \eqref{lamdef}.
Then the following statements hold:

\begin{enumerate}
\item Let $0<\lambda_*<\lambda_{\mathrm{N}}(\Omega)/q$ and $0 \le \lambda \le \lambda_*$.
For $F \in L^q(0,\infty;L_0^r(\Omega))$, define $\mathcal{I}F$ by setting
\begin{equation}
(\mathcal{I}F)(t,x) \coloneqq \int_0^t e^{\lambda (t-\tau)}e^{(t-\tau)\Delta_{\mathrm{N},0}}F(\tau,x) \, d\tau, \quad (t,x) \in (0,\infty) \times \Omega.
\end{equation}
Then it holds that $\mathcal{I}F \in L^q(0,\infty;L_0^r(\Omega))$ with the estimate
\begin{equation}\label{ipsiest}
\|\mathcal{I}F\|_{L^q(L^r(\Omega))} \le C_{\lambda_*}\|F\|_{L^q(L^r(\Omega))},
\end{equation}
where $C_{\lambda_*}>0$ is a constant independent of $\lambda$ and $F$.
Moreover, if there is a constant $\alpha \in \mathbb{R} \setminus \{0\}$ such that $\mathcal{I}F=\alpha F$ in $(0,\infty) \times \Omega$, then it holds that $F \equiv 0$.

\item Let $0<\lambda_*<\lambda_{\mathrm{D}}(\Omega)/q$ and $0 \le \lambda \le \lambda_*$.
For $F \in L^q(0,\infty;L_{\sigma}^r(\Omega))$, define $\mathcal{J}F$ by setting
\begin{equation}
(\mathcal{J}F)(t,x) \coloneqq \int_0^t e^{\lambda (t-\tau)}e^{-(t-\tau)A}F(\tau,x) \, d\tau, \quad (t,x) \in (0,\infty) \times \Omega.
\end{equation}
Then it holds that $\mathcal{J}F \in L^q(0,\infty;L_{\sigma}^r(\Omega))$ with the estimate
\begin{equation}
\|\mathcal{J}F\|_{L^q(L^r(\Omega))} \le C_{\lambda_*}\|F\|_{L^q(L^r(\Omega))}.
\end{equation}
Moreover, if there is a constant $\alpha \in \mathbb{R} \setminus \{0\}$ such that $\mathcal{J}F=\alpha F$ in $(0,\infty) \times \Omega$, then it holds that $F \equiv 0$.
\end{enumerate}
\end{prop}

\begin{proof}
\textbf{(i)} It follows from \eqref{semiest} and the H\"older inequality that
\begin{align}
\|(\mathcal{I}F)(t,\,\cdot\,)\|_{L^r(\Omega)} &\le C\int_0^t e^{\lambda (t-\tau)}e^{-\lambda_{\mathrm{N}}(\Omega)(t-\tau)}\|F(\tau,\,\cdot\,)\|_{L^r(\Omega)} \, d\tau \\
&\le C\bigg(\int_0^t e^{-\lambda_{\mathrm{N}}(\Omega)(t-\tau)} \, d\tau \bigg)^{1/q'}\bigg(\int_0^t e^{\lambda q(t-\tau)}e^{-\lambda_{\mathrm{N}}(\Omega)(t-\tau)}\|F(\tau,\,\cdot\,)\|_{L^r(\Omega)}^q \, d\tau \bigg)^{1/q} \\
&\le C(\lambda_{\mathrm{N}}(\Omega))^{-1/q'}\bigg(\int_0^t e^{-(\lambda_{\mathrm{N}}(\Omega)-\lambda q)(t-\tau)}\|F(\tau,\,\cdot\,)\|_{L^r(\Omega)}^q \, d\tau \bigg)^{1/q}
\end{align}
due to $1/q+1/q'=1$.
Since $0 \le \lambda<\lambda_{\mathrm{N}}(\Omega)/q$, together with the aforementioned inequality and the Fubini theorem, we obtain
\begin{align}
\int_T^{T+h}\|(\mathcal{I}F)(t,\,\cdot\,) \|_{L^r(\Omega)}^q \, dt &\le C\int_T^{T+h}\int_0^t e^{- (\lambda_{\mathrm{N}}(\Omega)-\lambda q)(t-\tau)}\|F(\tau,\,\cdot\,)\|_{L^r(\Omega)}^q \, d\tau \, dt \\
&= C\int_0^T \bigg(\int_T^{T+h} e^{-(\lambda_{\mathrm{N}}(\Omega)-\lambda q)(t-\tau)} dt \bigg)\|F(\tau,\,\cdot\,)\|_{L^r(\Omega)}^q \, d\tau \\
&\quad + C\int_T^{T+h} \bigg(\int_\tau^{T+h} e^{-(\lambda_{\mathrm{N}}(\Omega)-\lambda q)(t-\tau)} dt \bigg)\|F(\tau,\,\cdot\,)\|_{L^r(\Omega)}^q \, d\tau \\
&\le \frac{C}{\lambda_{\mathrm{N}}(\Omega)-\lambda q}(1-e^{-(\lambda_{\mathrm{N}}(\Omega)-\lambda q)h}) \int_0^{T+h} \|F(\tau,\,\cdot\,)\|_{L^r(\Omega)}^q \, d\tau
\end{align}
for all $T,h \in (0,\infty)$, which yields
$$\|\mathcal{I}F\|_{L^q(T,T+h;L^r(\Omega))} \le C_{\lambda_*}(1-e^{-(\lambda_{\mathrm{N}}(\Omega)-\lambda q)h})^{1/q}\|F\|_{L^q(0,T+h;L^r(\Omega))}.$$
Thus, we have \eqref{ipsiest} by letting $T \to +0$ and $h \to \infty$.
In addition, assume that $\mathcal{I}F=\alpha F$.
Then, by taking $h$ sufficiently small so that $C(1-e^{-(\lambda_{\mathrm{N}}(\Omega)-\lambda q)h})^{1/q} \le |\alpha|/2$, we obtain
$$|\alpha|\|F\|_{L^q(T,T+h;L^r(\Omega))} \le \frac{|\alpha|}{2}\|F\|_{L^q(0,T+h;L^r(\Omega))}.$$
Hence, taking $T=0$ implies that $F=0$ holds in $(0,h) \times \Omega$.
Next we take $T=h$.
Then it holds that
$$|\alpha|\|F\|_{L^q(h,2h;L^r(\Omega))} \le \frac{|\alpha|}{2}\|F\|_{L^q(0,2h;L^r(\Omega))}= \frac{|\alpha|}{2}\|F\|_{L^q(h,2h;L^r(\Omega))},$$
which yields $F=0$ in $(0,2h) \times \Omega$.
By repeating this argument, we have the desired result.

\noindent
\textbf{(ii)} Noting that \cite{amann}*{Remark 3.1} yields
$$\|e^{-tA}\bm{\psi}\|_{L^r(\Omega)} \le Ce^{-\lambda_{\mathrm{D}}(\Omega)t}\|\bm{\psi}\|_{L^r(\Omega)}$$
for all $0<t<\infty$ and $\bm{\psi} \in L_{\sigma}^r(\Omega)$, we simply follow the proof of the assertion of (i) with \eqref{semiest} replaced by the above estimate to conclude the desired result.
\end{proof}

Concerning the boundary condition of System \eqref{nheat} in Theorem \ref{maxneu0}, we shall introduce the mapping property of the trace operator $\gamma$.
To describe the result, the notations of the fractional Sobolev spaces $W^{\theta,r}(\partial\Omega)$ defined on the boundary $\partial\Omega$ and the $X$-valued Triebel--Lizorkin spaces $F_{q,r}^s(0,\infty;X)$ are used (cf.~\cite{grisvard}*{Definition 1.3.3.2}, \cite{galdi}*{Section II.4}, and \cite{hytonen}*{Section 14.6}), but we omit the precise definitions here; these function spaces appear only in the following claim.

\begin{prop}\label{maxtrace}
Let $r,q \in (1,\infty)$ and $\bm{F} \in L^q(0,\infty;W^{1,r}(\Omega)^N) \cap W^{1,q}(0,\infty;L^r(\Omega)^N)$.
Then it holds that
$$(\gamma \bm{F}) \cdot \bm{\nu} \in L^q(0,\infty;W^{1-1/r,r}(\partial\Omega)) \cap F_{q,r}^{(1/2)(1-1/r)}(0,\infty;L^r(\partial\Omega))$$
with the estimate
$$\|(\gamma\bm{F}) \cdot \bm{\nu}\|_{L^q(W^{1-1/r,r}(\partial\Omega)) \cap F_{q,r}^{(1/2)(1-1/r)}(L^r(\partial\Omega))} \le C\|\bm{F}\|_{L^q(W^{1,r}(\Omega)) \cap W^{1,q}(L^r(\Omega))},$$
where $\bm{\nu}$ denotes a unit outer normal to $\partial\Omega$ and $C>0$ is a constant independent of $\bm{F}$.
\end{prop}

\begin{proof}
Since the boundary $\partial\Omega$ is smooth, we may take $\widetilde{\bm{\nu}} \in C^1(\overline{\Omega})^N$ so that $\gamma\widetilde{\bm{\nu}}=\bm{\nu}$ from \cite{grisvard}*{Theorem 1.5.1.2}.
Thus we observe that
$$\bm{F} \cdot \widetilde{\bm{\nu}} \in L^q(0,\infty;W^{1,r}(\Omega)) \cap W^{1,q}(0,\infty;L^r(\Omega)).$$
Therefore, it holds by the result of Lindemulder \cite{lindemulder}*{Theorem 4.4} that
\begin{equation}
\begin{aligned}
(\gamma \bm{F}) \cdot \bm{\nu}=\gamma(\bm{F}\cdot\widetilde{\bm{\nu}}) &\in L^q(0,\infty;W^{1-1/r,r}(\partial\Omega)) \cap F_{q,r}^{1-1/r}(0,\infty;L^r(\partial\Omega)) \\
&\subset L^q(0,\infty;W^{1-1/r,r}(\partial\Omega)) \cap F_{q,r}^{(1/2)(1-1/r)}(0,\infty;L^r(\partial\Omega)).
\end{aligned}
\end{equation}
This proves Proposition \ref{maxtrace}.
\end{proof}

Finally, before showing Theorem \ref{maxneu0}, we verify the following proposition, which is an immediate consequence of the maximal regularity theorem given by Pr\"uss and Simonett \cite{prusssimonett}*{Theorem 6.3.2}:

\begin{prop}\label{maxneu}
Let $r,q \in (1,\infty)$ satisfy $1/r+2/q \neq 1$.
Assume that $F_0$, $F_{\mathrm{E}}$, and $\bm{F}_{\mathrm{B}}$ satisfy
\begin{gather}
F_0 \in B_{r,q}^{2-2/q}(\Omega), \quad F_{\mathrm{E}} \in L^q(0,\infty;L^r(\Omega)), \\
\bm{F}_{\mathrm{B}} \in  L^q(0,\infty;W^{1,r}(\Omega)^N) \cap W^{1,q}(0,\infty;L^r(\Omega)^N).
\end{gather}
Furthermore, if $1/r+2/q<1$, suppose that $F_0$ and $\bm{F}_{\mathrm{B}}$ satisfy the additional condition $\nabla F_0 \cdot \bm{\nu}=\bm{F}_{\mathrm{B}}(0,\,\cdot\,) \cdot \bm{\nu}$ on $\partial\Omega$.
Then there exists a constant $\omega_0>0$ independent of $F_0$, $F_{\mathrm{E}}$, and $\bm{F}_{\mathrm{B}}$ such that for every $\omega \ge \omega_0$, there exists a unique global strong solution $U_{\omega} \in \mathbb{E}_{q,r}^0$ to the equations
\begin{equation}
\left\{\begin{aligned}
\partial_tU_{\omega}+\omega U_{\omega}-\Delta U_{\omega} &=-\nabla \cdot \bm{F}_{\mathrm{B}}+F_{\mathrm{E}}, & t>0, \, x &\in \Omega, \\
\nabla U_{\omega} \cdot \bm{\nu} &=\bm{F}_{\mathrm{B}} \cdot \bm{\nu}, & t>0, \, x &\in \partial\Omega, \\
U_{\omega} (0,x) &= F_0(x), & x &\in \Omega.
\end{aligned}\right.
\end{equation}
The boundary condition is satisfied in the sense that \eqref{VBC}. Moreover, it holds that
$$\|U_{\omega}\|_{\mathbb{E}_{q,r}^0} \le C\left(\|F_0\|_{B_{r,q}^{2-2/q}(\Omega)}+\|F_{\mathrm{E}}\|_{L^q(L^r(\Omega))}+\|\bm{F}_{\mathrm{B}}\|_{L^q(W^{1,r}(\Omega)) \cap W^{1,q}(L^r(\Omega))}\right),$$
where $C>0$ is a constant independent of $F_0$, $F_{\mathrm{E}}$, $\bm{F}_{\mathrm{B}}$, and $U_{\omega}$.
\end{prop}

\begin{proof}
Note that Proposition \ref{maxtrace} yields
\begin{equation}
\begin{aligned}
\|\nabla \cdot \bm{F}_{\mathrm{B}}\|_{L^q(L^r(\Omega))} &\le C\|\bm{F}_{\mathrm{B}}\|_{L^q(W^{1,r}(\Omega))}, \\
\|(\gamma \bm{F}_{\mathrm{B}}) \cdot \bm{\nu}\|_{L^q(W^{1-1/r,r}(\partial\Omega)) \cap F_{q,r}^{(1/2)(1-1/r)}(L^r(\partial\Omega))} &\le C\|\bm{F}_{\mathrm{B}}\|_{L^q(W^{1,r}(\Omega)) \cap W^{1,q}(L^r(\Omega))}.
\end{aligned}
\end{equation}
We set $m=m_j=\mu=1$ in the assertion given by Pr\"uss and Simonett \cite{prusssimonett}*{Theorem 6.3.2}.
Then we have $\kappa_j \coloneqq 1-m_j/(2m)-1/(2mr)=(1/2)(1-1/r)$.
In case $1/r+2/q<1$, the compatibility condition $\nabla F_0 \cdot \bm{\nu}=\bm{F}_{\mathrm{B}}(0,\,\cdot\,) \cdot \bm{\nu}$ on $\partial\Omega$ appears from $\kappa_j =(1/2)(1-1/r)>1/2 \cdot 2/q=1/q+1-\mu$.
Hence, by \cite{prusssimonett}*{Theorem 6.3.2} we may find a sufficiently large $\omega_0>0$ and obtain the desired result.
\end{proof}

We now turn to the proof of Theorem \ref{maxneu0}.

\begin{proof}[Proof of Theorem \ref{maxneu0}]
In the following, let $0 \le \lambda<\lambda_{\mathrm{N}}(\Omega)/q$, where $\lambda_{\mathrm{N}}(\Omega)>0$ is given by \eqref{lamdef}.
By the assumptions \eqref{Fcond}, we may apply Proposition \ref{maxneu} and observe that there exist a constant $\omega>0$ and a unique global strong solution $U_{\omega} \in \mathbb{E}_{q,r}^0$ to the following equations
\begin{equation}\label{uoeq}
\left\{\begin{aligned}
\partial_tU_{\omega}+\omega U_{\omega}-\Delta U_{\omega} &= -e^{\lambda t}\nabla \cdot \bm{F}_{\mathrm{B}}+e^{\lambda t}F_{\mathrm{E}}, & t>0, \, x &\in \Omega, \\
\nabla U_{\omega} \cdot \bm{\nu} &=e^{\lambda t}\bm{F}_{\mathrm{B}} \cdot \bm{\nu}, & t>0, \, x &\in \partial\Omega, \\
U_{\omega} (0,x) &= F_0(x), & x &\in \Omega.
\end{aligned}\right.
\end{equation}
In addition, the solution $U_\omega$ satisfies
\begin{equation}\label{uoest}
\|U_{\omega}\|_{\mathbb{E}_{q,r}^0} \le C\left(\|F_0\|_{B_{r,q}^{2-2/q}(\Omega)}+\|e^{\lambda t}F_{\mathrm{E}}\|_{L^q(L^r(\Omega))}+\|e^{\lambda t}\bm{F}_{\mathrm{B}}\|_{L^q(W^{1,r}(\Omega)) \cap W^{1,q}(L^r(\Omega))}\right).
\end{equation}
Here, by integrating the first equation of \eqref{uoeq} over $\Omega$ and using $e^{\lambda t}F_{\mathrm{E}}(t,\,\cdot\,) \in L_0^r(\Omega)$ for all $0<t<\infty$, we deduce that
\begin{equation}
\partial_t\int_{\Omega}U_{\omega}(t,x)\,dx+\omega \int_{\Omega}U_{\omega}(t,x)\,dx-\int_{\Omega}\Delta U_{\omega}(t,x)\,dx=-\int_{\Omega}e^{\lambda t}\nabla \cdot \bm{F}_{\mathrm{B}}(t,x)\,dx.
\end{equation}
Moreover, Proposition \ref{gaussdiv} yields 
\begin{equation}
\partial_t\int_{\Omega}U_{\omega}(t,x)\,dx=-\omega \int_{\Omega}U_{\omega}(t,x)\,dx \qquad \text{for all $0<t<\infty$}.
\end{equation}
Therefore, by noting $F_0 \in L_0^r(\Omega)$, we observe that
$$\int_{\Omega}U_{\omega}(t,x)\,dx=e^{-\omega t}\int_{\Omega}F_0(x)\,dx \equiv 0,$$
i.e., $U_{\omega} \in \mathbb{E}_{q,r}^0(-\Delta_{\mathrm{N}})$.
Note that the solution $U_\omega$ to \eqref{uoeq} solves
\begin{equation}\label{uoeq2}
\left\{\begin{aligned}
\partial_t(e^{-\lambda t}U_{\omega})+(\omega+\lambda)e^{-\lambda t}U_{\omega}-\Delta (e^{-\lambda t}U_{\omega}) &= -\nabla \cdot \bm{F}_{\mathrm{B}}+F_{\mathrm{E}}, & t>0, \, x &\in \Omega, \\
\nabla (e^{-\lambda t}U_{\omega}) \cdot \bm{\nu} &=\bm{F}_{\mathrm{B}} \cdot \bm{\nu}, & t>0, \, x &\in \partial\Omega, \\
U_{\omega} (0,x) &= F_0(x), & x &\in \Omega.
\end{aligned}\right.
\end{equation}
We now set
\begin{equation}\label{un0def}
U_{\mathrm{N},0}(t,x) \coloneqq (\omega+\lambda)\int_0^t e^{\lambda (t-\tau)}e^{(t-\tau)\Delta_{\mathrm{N},0}}U_{\omega}(\tau,x) \, d\tau, \quad (t,x) \in (0,\infty) \times \Omega.
\end{equation}
Then it holds by Proposition \ref{ipsipro} (i) and \eqref{uoest} that $U_{\mathrm{N},0} \in L^q(0,\infty;L_0^r(\Omega))$ with the estimate
\begin{align}
\|U_{\mathrm{N},0}\|_{L^q(L^r(\Omega))} &\le C\|U_{\omega}\|_{L^q (L^r (\Omega))} \\
&\le C \Big(\|F_0\|_{B^{2-2/q}_{r,q}(\Omega)}+\|e^{\lambda t} F_{\mathrm{E}}\|_{L^q(L^r(\Omega))}+\|e^{\lambda t}\bm{F}_{\mathrm{B}}\|_{L^q(W^{1,r}(\Omega)) \cap W^{1,q}(L^r(\Omega))} \Big).
\end{align}
Hence, by applying Proposition \ref{maxneu}, we see that there exists a unique global strong solution $\widetilde{U}_{\mathrm{N},0} \in \mathbb{E}_{q,r}^0$ to the following equations
\begin{equation}\label{un0eq}
\left\{\begin{aligned}
\partial_t\widetilde{U}_{\mathrm{N},0}+\omega\widetilde{U}_{\mathrm{N},0}-\Delta\widetilde{U}_{\mathrm{N},0} &=(\omega+\lambda)(U_{\omega}+U_{\mathrm{N},0}), & t>0, \, x &\in \Omega, \\
\nabla \widetilde{U}_{\mathrm{N},0} \cdot \bm{\nu} &=0, & t>0, \, x &\in \partial\Omega, \\
\widetilde{U}_{\mathrm{N},0} (0,x) &=0, & x &\in \Omega.
\end{aligned}\right.
\end{equation}
In addition, we obtain
\begin{align}
\|\widetilde{U}_{\mathrm{N},0}\|_{\mathbb{E}^0_{q,r}} &\le C\Big(\|U_{\omega}\|_{L^q(L^r(\Omega))}+\|U_{\mathrm{N},0}\|_{L^q(L^r(\Omega))} \Big) \\
&\le C\Big(\|F_0\|_{B^{2-2/q}_{r,q}(\Omega)}+\|e^{\lambda t} F_{\mathrm{E}}\|_{L^q(L^r(\Omega))}+\|e^{\lambda t}\bm{F}_{\mathrm{B}}\|_{L^q(W^{1,r}(\Omega)) \cap W^{1,q}(L^r(\Omega))} \Big).
\end{align}
Note that $(U_{\omega}+U_{\mathrm{N},0})(t,\,\cdot\,) \in L_0^r(\Omega)$ for all $0<t<\infty$, and hence we see by \eqref{un0eq} and Proposition \ref{gaussdiv} that 
\begin{equation}
\partial_t\int_{\Omega}\widetilde{U}_{\mathrm{N},0}(t,x)\,dx+\omega\int_{\Omega}\widetilde{U}_{\mathrm{N},0}(t,x)\,dx=0 \qquad \text{for all $0<t<\infty$}.
\end{equation}
Thus we have $\widetilde{U}_{\mathrm{N},0} \in L^q(0,\infty;D(-\Delta_{\mathrm{N},0}))$.
We also see by \eqref{un0eq} that
\begin{equation}\label{un0eq2}
\left\{\begin{aligned}
\partial_t(e^{-\lambda t}\widetilde{U}_{\mathrm{N},0})-\Delta_{\mathrm{N},0}(e^{-\lambda t}\widetilde{U}_{\mathrm{N},0}) &=(\omega+\lambda)e^{-\lambda t}(U_{\omega}+U_{\mathrm{N},0}-\widetilde{U}_{\mathrm{N},0}), & t>0, \, x &\in \Omega, \\
\widetilde{U}_{\mathrm{N},0} (0,x) &=0, & x &\in \Omega,
\end{aligned}\right.
\end{equation}
which yields
\begin{equation}
e^{-\lambda t}\widetilde{U}_{\mathrm{N},0} (t,x) = (\omega+\lambda)\int_0^te^{(t-\tau)\Delta_{\mathrm{N},0}}e^{-\lambda \tau}(U_{\omega}+U_{\mathrm{N},0}-\widetilde{U}_{\mathrm{N},0})(\tau,x) \, d\tau
\end{equation}
for all $(t,x) \in (0,\infty) \times \Omega$ from the Duhamel principle.
Since \eqref{un0def} implies that
\begin{equation}
\begin{aligned}
\widetilde{U}_{\mathrm{N},0}(t,x) &= (\omega+\lambda)\int_0^te^{\lambda (t-\tau)}e^{(t-\tau)\Delta_{\mathrm{N},0}}(U_{\omega}+U_{\mathrm{N},0}-\widetilde{U}_{\mathrm{N},0})(\tau,x) \, d\tau \\
&= U_{\mathrm{N},0}(t,x)+(\omega+\lambda)\int_0^te^{\lambda (t-\tau)}e^{(t-\tau)\Delta_{\mathrm{N},0}}(U_{\mathrm{N},0}-\widetilde{U}_{\mathrm{N},0})(\tau,x) \, d\tau,
\end{aligned}
\end{equation}
we infer from the uniqueness assertion in Proposition \ref{ipsipro} (i) that $\widetilde{U}_{\mathrm{N},0}=U_{\mathrm{N},0}$, and hence we have
\begin{align}
\|U_{\mathrm{N},0}\|_{\mathbb{E}^0_{q, r}} \le C\Big(\|F_0\|_{B^{2-2/q}_{r,q}(\Omega)}+\|e^{\lambda t} F_{\mathrm{E}} \|_{L^q(L^r(\Omega))}+\|e^{\lambda t} \bm{F}_{\mathrm{B}}\|_{L^q(W^{1,r}(\Omega)) \cap W^{1,q}(L^r(\Omega))} \Big).
\end{align}
Together with \eqref{uoest}, we observe that the function $U$ defined by
$$U(t,x) \coloneqq e^{-\lambda t}\{U_{\omega}(t,x) + U_{\mathrm{N},0}(t,x)\}, \quad (t,x) \in (0,\infty) \times \Omega$$
satisfies $e^{\lambda t}U \in \mathbb{E}_{q,r}^0(-\Delta_{\mathrm{N}})$ with the estimate \eqref{nheatest}.
In addition, since $U_{\omega}$ and $U_{\mathrm{N},0}=\widetilde{U}_{\mathrm{N},0}$ satisfy \eqref{uoeq2} and \eqref{un0eq2}, respectively, we deduce that $U$ is a solution to \eqref{nheat}.
The proof is complete.
\end{proof}

\subsection{Shifted Neumann heat equation and the Stokes system}

We next prepare the maximal regularity results for a shifted Neumann heat equation and the Stokes system.

\begin{lemm}\label{maxheatstokes}
Let $r,q \in (1,\infty)$ satisfy $1/r+2/q<2$.
Then the following statements hold:
\begin{enumerate}
\item Let $0<\lambda<\infty$ and assume that $F_0 \in B_{r,q}^{3-2/q}(\Omega)$ with the condition $\nabla F_0 \cdot \bm{\nu}=0$ on $\partial\Omega$ and $F_{\mathrm{E}} \in L^q(0,\infty;W^{1,r}(\Omega))$.
Then there exists a unique global strong solution $U \in \mathbb{E}_{q,r}^1(1-\Delta_{\mathrm{N},1})$ to the equation
\begin{equation}
\left\{\begin{aligned}
\partial_tU+\lambda U-\Delta_{\mathrm{N},1}U &= F_{\mathrm{E}}, & t>0, \, x &\in \Omega, \\
U (0,x) &= F_0(x), & x &\in \Omega.
\end{aligned}\right.
\end{equation}
Moreover, it holds that
$$\|U\|_{\mathbb{E}_{q,r}^1} \le C\left(\|F_0\|_{B_{r,q}^{3-2/q}(\Omega)}+\|F_{\mathrm{E}}\|_{L^q(W^{1,r}(\Omega))}\right),$$
where $C>0$ is a constant independent of $F_0$, $F_{\mathrm{E}}$, and $U$.

\item Let $0 \le \lambda <\lambda_{\mathrm{D}}(\Omega)/q$, where $\lambda_{\mathrm{D}}(\Omega)>0$ is given by \eqref{lamdef}.
Assume that $\bm{F}_0 \in L_{\sigma}^r(\Omega) \cap B_{r,q}^{2-2/q}(\Omega)^N$ with the condition $\bm{F}_0=0$ on $\partial\Omega$ and $e^{\lambda t}\bm{F}_{\mathrm{E}} \in L^q(0,\infty;L_{\sigma}^r(\Omega))$.
Then there exists a unique global strong solution $\bm{U} \in \mathbb{E}_{q,r}^0(A)$ to the Stokes system
\begin{equation}
\left\{\begin{aligned}
\partial_t\bm{U}+A\bm{U} &= \bm{F}_{\mathrm{E}}, & t>0, \, x &\in \Omega, \\
\bm{U}(0,x) &= \bm{F}_0(x), & x &\in \Omega.
\end{aligned}\right.
\end{equation}
Moreover, it holds that
\begin{equation}
\|e^{\lambda t}\bm{U}\|_{\mathbb{E}_{q,r}^0} \le C\left(\|\bm{F}_0\|_{B_{r,q}^{2-2/q}(\Omega)}+\|e^{\lambda t}\bm{F}_{\mathrm{E}}\|_{L^q(L^r(\Omega))}\right),
\end{equation}
where $C>0$ is a constant independent of $\bm{F}_0$, $\bm{F}_{\mathrm{E}}$, and $\bm{U}$.
\end{enumerate}
\end{lemm}

\begin{proof}
\textbf{(i)} Applying \cite{prusssimonett}*{Theorem 6.3.3}, we have the desired result.
More precisely, we set $m=m_j=\mu=1$ in \cite{prusssimonett}*{Theorem 6.3.3}.
Then we have 
\begin{equation}
\kappa_j \coloneqq 1-\frac{m_j}{2m}-\frac{1}{2mr}=\frac12 \bigg(1-\frac1r\bigg),
\end{equation}
and thus the compatibility condition $\nabla F_0 \cdot \bm{\nu}=0$ on $\partial\Omega$ appears from the condition 
\begin{equation}
\kappa_j =\frac12 \bigg(1-\frac1r\bigg)>\frac12\bigg(\frac2q-1\bigg)=\frac1q+1-\mu-\frac{1}{2m}.
\end{equation}
Concerning the constant $\omega_0 \in \mathbb{R}$ appearing in \cite{prusssimonett}*{Theorem 6.3.3}, we infer from the discussion in \cite{prusssimonett}*{Section 6.3.4} that $\omega_0=\sup\{\sigma(\Delta_{\mathrm{N},1}) \cap \mathbb{R}\}=0$, where $\sigma(\Delta_{\mathrm{N},1}) \subset \mathbb{C}$ denotes the spectrum of $\Delta_{\mathrm{N},1}$.
Hence we may take $\omega=\lambda$ in \cite{prusssimonett}*{Theorem 6.3.3} due to $\lambda>0=\omega_0$.

\noindent
\textbf{(ii)} The assertion in the case of $\lambda=0$ is a direct consequence of \cite{prusssimonett}*{Theorem 7.3.2} with $\mu=1$ and $\omega=0$ due to $\omega_0=\sup\{\sigma(-A) \cap \mathbb{R}\}<0$.
Note that the condition of the initial data $\bm{F}_0$ is characterized by $\bm{F}_0 \in (L_{\sigma}^r(\Omega),D(A))_{1-1/q,q}$ from Propositions 3.4.2 and 3.4.4 in \cite{prusssimonett}.
Here, it is well-known that this real interpolation space is equivalent to $L_{\sigma}^r(\Omega) \cap B_{r,q}^{2-2/q}(\Omega)^N$ due to $1/r<2-2/q<2$, see Theorem 3.4 and Remark 3.7 (b) in \cite{amann}.

In the case of $0<\lambda <\lambda_{\mathrm{D}}(\Omega)/q$, we simply rely on a similar method to the proof of Theorem \ref{maxneu0} by applying Proposition \ref{ipsipro} (ii) instead of the assertion of (i).
More precisely, to complete the proof, we construct a unique global strong solution $\bm{U}_1 \in \mathbb{E}_{q,r}^0(A)$ to
\begin{equation}
\left\{\begin{aligned}
\partial_t\bm{U}_1+A\bm{U}_1 &= e^{\lambda t}\bm{F}_{\mathrm{E}}, & t>0, \, x &\in \Omega, \\
\bm{U}_1(0,x) &= \bm{F}_0(x), & x &\in \Omega
\end{aligned}\right.
\end{equation}
with a suitable maximal regularity estimate and define
$$\bm{U}_2(t,x) \coloneqq \lambda\int_0^te^{\lambda(t-\tau)}e^{-(t-\tau)A}\bm{U}_1(\tau,x) \, d\tau, \quad (t,x) \in (0,\infty) \times \Omega.$$
Since we may show the existence of a global strong solution $\widetilde{\bm{U}}_2 \in \mathbb{E}_{q,r}^0(A)$ to
\begin{equation}
\left\{\begin{aligned}
\partial_t\widetilde{\bm{U}}_2+A\widetilde{\bm{U}}_2 &= \lambda (\bm{U}_1+\bm{U}_2), & t>0, \, x &\in \Omega, \\
\widetilde{\bm{U}}_2(0,x) &= 0, & x &\in \Omega
\end{aligned}\right.
\end{equation}
with a suitable maximal regularity estimate and since the uniqueness assertion in Proposition \ref{ipsipro} (ii) ensures that $\widetilde{\bm{U}}_2=\bm{U}_2$, by setting
$$\bm{U}(t,x) \coloneqq e^{-\lambda t}\{\bm{U}_1(t,x) + \bm{U}_2(t,x)\}, \quad (t,x) \in (0,\infty) \times \Omega,$$
we have the desired estimate.
\end{proof}

\section{Maximal regularity theorem for the Keller--Segel--Navier--Stokes system}
\label{sec-4}

In this section, we give the proof Theorem \ref{globalsol}, which gives the existence of unique global strong solutions to the Keller--Segel--Navier--Stokes system \eqref{ABS}.
Since the initial density $n_0$ does not satisfy the mean value condition $\int_\Omega n_0 (x)\, dx = 0$ as in Theorem \ref{maxneu0}, we first introduce $\widetilde{n}$ and $\widetilde{c}$ by setting
$$\widetilde{n}(t,x) \coloneqq n(t,x)-\overline{n}_0, \quad \widetilde{c}(t,x) \coloneqq c(t,x)-(1-e^{-t})\overline{n}_0$$
for $(t,x) \in (0,\infty) \times \Omega$.
If $(n,c,\bm{u})$ is a solution to System \eqref{ABS}, then we observe that $(\widetilde{n},\widetilde{c},\bm{u})$ satisfies
\begin{equation}\label{shift}
\left\{\begin{aligned}
\partial_t\widetilde{n}-\Delta \widetilde{n} &= -\nabla \cdot ((\widetilde{n}+\overline{n}_0)S(t,x)\nabla \widetilde{c})-\bm{u} \cdot \nabla \widetilde{n}, & t>0, \, x &\in \Omega, \\
\partial_t\widetilde{c}+(1-\Delta_{\mathrm{N},1})\widetilde{c} &=\widetilde{n}-\bm{u} \cdot \nabla \widetilde{c}, & t>0, \, x &\in \Omega, \\
\partial_t\bm{u}+A\bm{u} &= -P(\bm{u}\cdot\nabla)\bm{u}+P(\widetilde{n}\nabla \varphi)+P\bm{f}, & t>0, \, x &\in \Omega, \\
\nabla \widetilde{n} \cdot \bm{\nu} &=(\widetilde{n}+\overline{n}_0)S(t,x)\nabla \widetilde{c} \cdot \bm{\nu}, & t>0, \, x &\in \partial\Omega, \\
(\widetilde{n},\widetilde{c},\bm{u})(0,x) &= (n_0-\overline{n}_0,c_0,\bm{u}_0)(x), & x &\in \Omega.
\end{aligned}\right.
\end{equation}
In the following we shall construct global strong solutions $(\widetilde{n},\widetilde{c},\bm{u})$ to System \eqref{shift} by applying Theorem \ref{maxneu0} and Lemma \ref{maxheatstokes} together with the Banach fixed point theorem.
Concerning the nonlinear terms appearing in \eqref{shift}, we will make use of the following estimates. 

\begin{lemm}\label{nonlest}
Let $N<r<\infty$ and $2<q<\infty$ and let $\varphi$ satisfy $\nabla\varphi \in L^r(\Omega)^N$.
Assume that $S$ satisfies \eqref{scond} and $\widetilde{n} \in \mathbb{E}_{q,r}^0$, $\widetilde{c} \in \mathbb{E}_{q,r}^1$, and $\bm{u},\bm{u}^* \in (\mathbb{E}_{q,r}^0)^N$.
Then the following statements hold:
\begin{enumerate}
\item It holds that
\begin{equation}
\begin{aligned}
\bm{u} \cdot \nabla\widetilde{n} &\in L^q(0,\infty;L^r(\Omega)), \quad \bm{u} \cdot \nabla \widetilde{c} \in L^q(0,\infty;W^{1,r}(\Omega)), \\
(\bm{u}\cdot\nabla)\bm{u}^*,\widetilde{n}\nabla\varphi &\in L^q(0,\infty;L^r(\Omega)^N)
\end{aligned}
\end{equation}
with the estimates
\begin{equation}
\begin{aligned}
\|\bm{u}\cdot\nabla\widetilde{n}\|_{L^q(L^r(\Omega))} &\le C\|\bm{u}\|_{\mathbb{E}_{q,r}^0}\|\widetilde{n}\|_{\mathbb{E}_{q,r}^0}, & \|\bm{u}\cdot\nabla \widetilde{c}\|_{L^q(W^{1,r}(\Omega))} &\le C\|\bm{u}\|_{\mathbb{E}_{q,r}^0}\|\widetilde{c}\|_{\mathbb{E}_{q,r}^1}, \\
\|(\bm{u}\cdot\nabla)\bm{u}^*\|_{L^q(L^r(\Omega))} &\le C\|\bm{u}\|_{\mathbb{E}_{q,r}^0}\|\bm{u}^*\|_{\mathbb{E}_{q,r}^0}, & \|\widetilde{n}\nabla\varphi\|_{L^q(L^r(\Omega))} &\le C\|\widetilde{n}\|_{\mathbb{E}_{q,r}^0}\|\nabla\varphi\|_{L^r(\Omega)},
\end{aligned}
\end{equation}
where $C>0$ is a constant independent of $\varphi$, $\widetilde{n}$, $\widetilde{c}$, $\bm{u}$, and $\bm{u}^*$.

\item Let $\alpha \in \mathbb{R}$.
Then it holds that
\begin{equation}
(\widetilde{n}+\alpha)S\nabla \widetilde{c} \in L^q(0,\infty;W^{1,r}(\Omega)^N) \cap W^{1,q}(0,\infty;L^r(\Omega)^N)
\end{equation}
with the estimate
\begin{equation}
\|(\widetilde{n}+\alpha)S\nabla \widetilde{c}\|_{L^q(W^{1,r}(\Omega)) \cap W^{1,q}(L^r(\Omega))} \le CM_S(\|\widetilde{n}\|_{\mathbb{E}_{q,r}^0}+|\alpha|)\|\widetilde{c}\|_{\mathbb{E}_{q,r}^1},
\end{equation}
where
\begin{equation}\label{msdef}
M_S \coloneqq 1+\|S\|_{L^{\infty}(W^{1,r}(\Omega))}+\|\partial_tS\|_{L^q(L^r(\Omega))}
\end{equation}
and $C>0$ is a constant independent of $\alpha$, $S$, $\widetilde{n}$, and $\widetilde{c}$.
\end{enumerate}
\end{lemm}

\begin{proof}
\textbf{(i)} Since $r>N$, the Sobolev embedding yields $W^{1,r}(\Omega) \subset C(\overline{\Omega})$.
In particular, the space $W^{1,r}(\Omega)$ is a Banach algebra.
Hence, applying these properties implies that
\begin{equation}
\begin{alignedat}{2}
\|(\bm{u}\cdot\nabla \widetilde{n})(t,\,\cdot\,)\|_{L^r(\Omega)} &\le C\|\bm{u}(t,\,\cdot\,)\|_{L^r(\Omega)}\|\nabla \widetilde{n}(t,\,\cdot\,)\|_{L^{\infty}(\Omega)} &&\le C\|\bm{u}(t,\,\cdot\,)\|_{L^r(\Omega)}\|\widetilde{n}(t,\,\cdot\,)\|_{W^{2,r}(\Omega)}, \\
\|(\bm{u} \cdot \nabla \widetilde{c})(t,\,\cdot\,)\|_{W^{1,r}(\Omega)} &\le C\|\bm{u}(t,\,\cdot\,)\|_{W^{1,r}(\Omega)}\|\nabla \widetilde{c}(t,\,\cdot\,)\|_{W^{1,r}(\Omega)} &&\le C\|\bm{u}(t,\,\cdot\,)\|_{W^{1,r}(\Omega)}\|\widetilde{c}(t,\,\cdot\,)\|_{W^{2,r}(\Omega)}, \\
\|(\bm{u}\cdot\nabla )\bm{u}^*(t,\,\cdot\,)\|_{L^r(\Omega)} &\le C\|\bm{u}(t,\,\cdot\,)\|_{L^r(\Omega)}\|\nabla \bm{u}^*(t,\,\cdot\,)\|_{L^{\infty}(\Omega)} &&\le C\|\bm{u}(t,\,\cdot\,)\|_{L^r(\Omega)}\|\bm{u}^*(t,\,\cdot\,)\|_{W^{2,r}(\Omega)}, \\
\|\widetilde{n}(t,\,\cdot\,)\nabla \varphi\|_{L^r(\Omega)} &\le C\|\widetilde{n}(t,\,\cdot\,)\|_{L^{\infty}(\Omega)}\|\nabla \varphi\|_{L^r(\Omega)} &&\le C\|\widetilde{n}(t,\,\cdot\,)\|_{W^{1,r}(\Omega)}\|\nabla \varphi\|_{L^r(\Omega)}
\end{alignedat}
\end{equation}
for all $0<t<\infty$.
By taking the $L^q$-norm in time and applying Corollary \ref{ncuemb} (ii), we have
\begin{equation}
\begin{alignedat}{2}
\|\bm{u}\cdot\nabla \widetilde{n}\|_{L^q(L^r(\Omega))} &\le C\|\bm{u}\|_{L^{\infty}(L^r(\Omega))}\|\widetilde{n}\|_{L^q(W^{2,r}(\Omega))} &&\le C\|\bm{u}\|_{\mathbb{E}_{q,r}^0}\|\widetilde{n}\|_{\mathbb{E}_{q,r}^0}, \\
\|\bm{u} \cdot \nabla \widetilde{c}\|_{L^q(W^{1,r}(\Omega))} &\le C\|\bm{u}\|_{L^{\infty}(W^{1,r}(\Omega))}\|\widetilde{c}\|_{L^q(W^{2,r}(\Omega))} &&\le C\|\bm{u}\|_{\mathbb{E}_{q,r}^0}\|\widetilde{c}\|_{\mathbb{E}_{q,r}^1}, \\
\|(\bm{u}\cdot\nabla)\bm{u}^*\|_{L^q(L^r(\Omega))} &\le C\|\bm{u}\|_{L^{\infty}(L^r(\Omega))}\|\bm{u}^*\|_{L^q(W^{2,r}(\Omega))} &&\le C\|\bm{u}\|_{\mathbb{E}_{q,r}^0}\|\bm{u}^*\|_{\mathbb{E}_{q,r}^0}, \\
\|\widetilde{n}\nabla \varphi\|_{L^q(L^r(\Omega))} &\le C\|\widetilde{n}\|_{L^q(W^{1,r}(\Omega))}\|\nabla \varphi\|_{L^r(\Omega)} &&\le C\|\widetilde{n}\|_{\mathbb{E}_{q,r}^0}\|\nabla \varphi\|_{L^r(\Omega)}.
\end{alignedat}
\end{equation}

\noindent
\textbf{(ii)} Since the space $W^{1,r}(\Omega)$ is a Banach algebra, we obtain
\begin{equation}
\|((\widetilde{n}+\alpha)S\nabla \widetilde{c})(t,\,\cdot\,)\|_{W^{1,r}(\Omega)} \le C(\|\widetilde{n}(t,\,\cdot\,)\|_{W^{1,r}(\Omega)}+|\alpha|)\|S(t,\,\cdot\,)\|_{W^{1,r}(\Omega)}\|\widetilde{c}(t,\,\cdot\,)\|_{W^{2,r}(\Omega)}
\end{equation}
for all $0<t<\infty$.
Thus we see by Corollary \ref{ncuemb} (ii) that
\begin{equation}
\begin{aligned}
\|(\widetilde{n}+\alpha)S\nabla \widetilde{c}\|_{L^q(W^{1,r}(\Omega))} &\le C(\|\widetilde{n}\|_{L^{\infty}(W^{1,r}(\Omega))}+|\alpha|)\|S\|_{L^{\infty}(W^{1,r}(\Omega))}\|\widetilde{c}\|_{L^q(W^{2,r}(\Omega))} \\
&\le C\|S\|_{L^{\infty}(W^{1,r}(\Omega))}(\|\widetilde{n}\|_{\mathbb{E}_{q,r}^0}+|\alpha|)\|\widetilde{c}\|_{\mathbb{E}_{q,r}^1}.
\end{aligned}
\end{equation}
Moreover, there hold
\begin{equation}
\begin{aligned}
\|((\partial_t(\widetilde{n}+\alpha))S\nabla \widetilde{c})(t,\,\cdot\,)\|_{L^r(\Omega)} &\le C\|\partial_t\widetilde{n}(t,\,\cdot\,)\|_{L^r(\Omega)}\|S(t,\,\cdot\,)\|_{W^{1,r}(\Omega)}\|\widetilde{c}(t,\,\cdot\,)\|_{W^{2,r}(\Omega)}, \\
\|((\widetilde{n}+\alpha)(\partial_tS)\nabla \widetilde{c})(t,\,\cdot\,)\|_{L^r(\Omega)} &\le C(\|\widetilde{n}(t,\,\cdot\,)\|_{W^{1,r}(\Omega)}+|\alpha|)\|\partial_tS(t,\,\cdot\,)\|_{L^r(\Omega)}\|\widetilde{c}(t,\,\cdot\,)\|_{W^{2,r}(\Omega)}, \\
\|((\widetilde{n}+\alpha)S\partial_t\nabla \widetilde{c})(t,\,\cdot\,)\|_{L^r(\Omega)} &\le C(\|\widetilde{n}(t,\,\cdot\,)\|_{W^{1,r}(\Omega)}+|\alpha|)\|S(t,\,\cdot\,)\|_{W^{1,r}(\Omega)}\|\partial_t\widetilde{c}(t,\,\cdot\,)\|_{W^{1,r}(\Omega)}
\end{aligned}
\end{equation}
for all $0<t<\infty$, which yield
\begin{equation}
\begin{aligned}
\|(\partial_t(\widetilde{n}+\alpha))S\nabla \widetilde{c}\|_{L^q(L^r(\Omega))} &\le C\|\widetilde{n}\|_{W^{1,q}(L^r(\Omega))}\|S\|_{L^{\infty}(W^{1,r}(\Omega))}\|\widetilde{c}\|_{L^{\infty}(W^{2,r}(\Omega))}, \\
\|(\widetilde{n}+\alpha)(\partial_tS)\nabla \widetilde{c}\|_{L^q(L^r(\Omega))} &\le C(\|\widetilde{n}\|_{L^{\infty}(W^{1,r}(\Omega))}+|\alpha|)\|\partial_tS\|_{L^q(L^r(\Omega))}\|\widetilde{c}\|_{L^{\infty}(W^{2,r}(\Omega))}, \\
\|(\widetilde{n}+\alpha)S\partial_t\nabla \widetilde{c}\|_{L^q(L^r(\Omega))} &\le C(\|\widetilde{n}\|_{L^{\infty}(W^{1,r}(\Omega))}+|\alpha|)\|S\|_{L^{\infty}(W^{1,r}(\Omega))}\|\widetilde{c}\|_{W^{1,q}(W^{1,r}(\Omega))}.
\end{aligned}
\end{equation}
Hence it holds by Corollary \ref{ncuemb} (ii) that
\begin{equation}
\|\partial_t((\widetilde{n}+\alpha)S\nabla \widetilde{c})\|_{L^q(L^r(\Omega))} \le C(\|S\|_{L^{\infty}(W^{1,r}(\Omega))}+\|\partial_tS\|_{L^q(L^r(\Omega))})(\|\widetilde{n}\|_{\mathbb{E}_{q,r}^0}+|\alpha|)\|\widetilde{c}\|_{\mathbb{E}_{q,r}^1},
\end{equation}
which gives the desired estimate.
\end{proof}

In the following, we shall prove Theorem \ref{globalsol}.
We fix $0<\lambda_1<\min\{1,\lambda_{\mathrm{N}}(\Omega)/q\}$ and $\lambda_2 \in [0,\lambda_1] \cap [0,\lambda_{\mathrm{D}}(\Omega)/q)$ and introduce the function space
$$\mathbb{E} \coloneqq \left\{(\widetilde{n},\widetilde{c},\bm{u}) \, \left| \, e^{\lambda_1 t}\widetilde{n} \in \mathbb{E}_{q,r}^0(-\Delta_{\mathrm{N}}), \, e^{\lambda_1 t}\widetilde{c} \in \mathbb{E}_{q,r}^1(1-\Delta_{\mathrm{N},1}), \, e^{\lambda_2 t}\bm{u} \in \mathbb{E}_{q,r}^0(A) \right.\right\}$$
with the norm
\begin{equation}
\begin{aligned}
\|v\|_{\mathbb{E}} &\coloneqq \|e^{\lambda_1 t}\widetilde{n}\|_{\mathbb{E}_{q,r}^0}+\|e^{\lambda_1 t}\widetilde{c}\|_{\mathbb{E}_{q,r}^1}+\|e^{\lambda_2 t}\bm{u}\|_{\mathbb{E}_{q,r}^0} \\
&=\|e^{\lambda_1 t}\widetilde{n}\|_{L^q(W^{2,r}(\Omega)) \cap W^{1,q}(L^r(\Omega))}+\|e^{\lambda_1 t}\widetilde{c}\|_{L^q(W^{3,r}(\Omega)) \cap W^{1,q}(W^{1,r}(\Omega))} \\
&\quad +\|e^{\lambda_2 t}\bm{u}\|_{L^q(W^{2,r}(\Omega)) \cap W^{1,q}(L^r(\Omega))},
\end{aligned}
\end{equation}
where $v \coloneqq (\widetilde{n},\widetilde{c},\bm{u})$ and $\lambda_{\mathrm{N}}(\Omega),\lambda_{\mathrm{D}}(\Omega) \in (0,\infty)$ are given by \eqref{lamdef}.
In addition, we also set
\begin{equation}\label{ininorm}
\llbracket (n_0,c_0,\bm{u}_0,\bm{f}) \rrbracket \coloneqq \|n_0\|_{B_{r,q}^{2-2/q}(\Omega)}+\|c_0\|_{B_{r,q}^{3-2/q}(\Omega)}+\|\bm{u}_0\|_{B_{r,q}^{2-2/q}(\Omega)}+\|e^{\lambda_2 t}\bm{f}\|_{L^q(L^r(\Omega))}.
\end{equation}
Then Corollary \ref{ncuemb} (i) yields
\begin{equation}\label{n0varest}
|\overline{n}_0|+\|n_0-\overline{n}_0\|_{B_{r,q}^{2-2/q}(\Omega)} \le C\llbracket (n_0,c_0,\bm{u}_0,\bm{f}) \rrbracket.
\end{equation}

Let $v \coloneqq (\widetilde{n},\widetilde{c},\bm{u}) \in \mathbb{E}$ be arbitrary.
Then, Lemma \ref{nonlest} and the estimate \eqref{n0varest} imply that
\begin{equation}
\begin{aligned}
e^{\lambda_1 t}(\widetilde{n}+\overline{n}_0)S\nabla \widetilde{c} &\in L^q(0,\infty;W^{1,r}(\Omega)^N) \cap W^{1,q}(0,\infty;L^r(\Omega)^N), \\
e^{\lambda_1 t}\bm{u}\cdot\nabla\widetilde{n} &\in L^q(0,\infty;L^r(\Omega))
\end{aligned}
\end{equation}
with the estimates
\begin{equation}\label{1nonlest}
\left\{\begin{aligned}
\|e^{\lambda_1 t}(\widetilde{n}+\overline{n}_0)S\nabla \widetilde{c}\|_{L^q(W^{1,r}(\Omega)) \cap W^{1,q}(L^r(\Omega))} &\le CM_S(\|\widetilde{n}\|_{\mathbb{E}_{q,r}^0}+|\overline{n}_0|)\|e^{\lambda_1 t}\widetilde{c}\|_{\mathbb{E}_{q,r}^1} \\
&\le CM_S(\|v\|_{\mathbb{E}}+\llbracket (n_0,c_0,\bm{u}_0,\bm{f}) \rrbracket)\|v\|_{\mathbb{E}}, \\
\|e^{\lambda_1 t}\bm{u}\cdot\nabla\widetilde{n}\|_{L^q(L^r(\Omega))} &\le C\|\bm{u}\|_{\mathbb{E}_{q,r}^0}\|e^{\lambda_1 t}\widetilde{n}\|_{\mathbb{E}_{q,r}^0} \le C\|v\|_{\mathbb{E}}^2.
\end{aligned}\right.
\end{equation}
In addition, we see that $\int_{\Omega}(n_0(x)-\overline{n}_0)\,dx=0$ and that Proposition \ref{gaussdiv} yields $\int_{\Omega}(\bm{u}\cdot\nabla \widetilde{n})(t,x)\,dx=0$ for all $0<t<\infty$, and thus we may apply Theorem \ref{maxneu0} to obtain a unique global strong solution $\Phi_1v$ to
\begin{equation}
\left\{\begin{aligned}
\partial_t\Phi_1v-\Delta \Phi_1v &= -\nabla \cdot ((\widetilde{n}+\overline{n}_0)S\nabla \widetilde{c})-\bm{u}\cdot\nabla\widetilde{n} && \text{in $(0,\infty) \times \Omega$}, \\
\nabla \Phi_1v \cdot \bm{\nu} &=(\widetilde{n}+\overline{n}_0)S\nabla \widetilde{c} \cdot \bm{\nu} && \text{in $(0,\infty) \times \partial\Omega$}, \\
(\Phi_1v) (0,\,\cdot\,) &= n_0-\overline{n}_0 && \text{in $\Omega$}
\end{aligned}\right.
\end{equation}
such that $e^{\lambda_1t}\Phi_1v \in \mathbb{E}_{q,r}^0(-\Delta_{\mathrm{N}})$.
The boundary condition is satisfied in the sense that
\begin{equation}\label{PNBC}
(\gamma(\nabla \Phi_1v(t,\,\cdot\,)-(\widetilde{n}(t,\,\cdot\,)+\overline{n}_0)S(t,\,\cdot\,)\nabla \widetilde{c}(t,\,\cdot\,))(x)) \cdot \bm{\nu}(x)=0
\end{equation}
for a.e.~$(t,x) \in (0,\infty) \times \partial\Omega$.
Moreover, it holds that
\begin{equation}\label{phi1est}
\begin{aligned}
\|e^{\lambda_1 t}\Phi_1v\|_{\mathbb{E}_{q,r}^0} &\le C\left(\|n_0-\overline{n}_0\|_{B_{r,q}^{2-2/q}(\Omega)}+\|e^{\lambda_1 t}\bm{u}\cdot\nabla\widetilde{n}\|_{L^q(L^r(\Omega))}\right. \\
&\left.\quad{}+\|e^{\lambda_1 t}(\widetilde{n}+\overline{n}_0)S\nabla \widetilde{c}\|_{L^q(W^{1,r}(\Omega)) \cap W^{1,q}(L^r(\Omega))}\right) \\
&\le C\llbracket (n_0,c_0,\bm{u}_0,\bm{f}) \rrbracket+CM_S(\|v\|_{\mathbb{E}}+\llbracket (n_0,c_0,\bm{u}_0,\bm{f}) \rrbracket)\|v\|_{\mathbb{E}}
\end{aligned}
\end{equation}
with the aid of \eqref{n0varest} and \eqref{1nonlest}.
Next, since $1-\lambda_1>0$, by combining Lemmas \ref{maxheatstokes} (i) and \ref{nonlest} (i), we observe that there exists a unique global strong solution $\widetilde{\Phi}_2v \in \mathbb{E}_{q,r}^1(1-\Delta_{\mathrm{N},1})$ to
\begin{equation}
\left\{\begin{aligned}
\partial_t\widetilde{\Phi}_2v+(1-\lambda_1-\Delta_{\mathrm{N},1})\widetilde{\Phi}_2v &= e^{\lambda_1 t}\widetilde{n}-e^{\lambda_1 t}\bm{u} \cdot \nabla \widetilde{c} && \text{in $(0,\infty) \times \Omega$}, \\
(\widetilde{\Phi}_2v)(0,\,\cdot\,) &=c_0 && \text{in $\Omega$}.
\end{aligned}\right.
\end{equation}
In addition, it holds by the estimates in Lemma \ref{nonlest} (i) that
\begin{equation}
\begin{aligned}
\|\widetilde{\Phi}_2v\|_{\mathbb{E}_{q,r}^1} &\le C\left(\|c_0\|_{B_{r,q}^{3-2/q}(\Omega)}+\|e^{\lambda_1 t}\widetilde{n}\|_{L^q(W^{1,r}(\Omega))}+\|e^{\lambda_1 t}\bm{u} \cdot \nabla \widetilde{c}\|_{L^q(W^{1,r}(\Omega))} \right) \\
&\le C\llbracket (n_0,c_0,\bm{u}_0,\bm{f}) \rrbracket+C\|v\|_{\mathbb{E}}^2,
\end{aligned}
\end{equation}
where we have used the notation introduced in \eqref{ininorm}.
Hence, the function $\Phi_2v \coloneqq e^{-\lambda_1 t}\widetilde{\Phi}_2v$ solves
\begin{equation}
\left\{\begin{aligned}
\partial_t\Phi_2v+(1-\Delta_{\mathrm{N},1})\Phi_2v &= \widetilde{n}-\bm{u} \cdot \nabla \widetilde{c} && \text{in $(0,\infty) \times \Omega$}, \\
(\Phi_2v)(0,\,\cdot\,) &=c_0 && \text{in $\Omega$}
\end{aligned}\right.
\end{equation}
and satisfies
\begin{equation}
\|e^{\lambda_1 t}\Phi_2v\|_{\mathbb{E}_{q,r}^1} \le C\llbracket (n_0,c_0,\bm{u}_0,\bm{f}) \rrbracket+C\|v\|_{\mathbb{E}}^2.
\end{equation}
By applying Lemma \ref{maxheatstokes} (ii) with the aid of Lemma \ref{nonlest} (i), we may construct a unique global strong solution $\bm{\Phi}_3v \in \mathbb{E}_{q,r}^0(A)$ to
\begin{equation}
\left\{\begin{aligned}
\partial_t\bm{\Phi}_3v+A\bm{\Phi}_3v &= -P(\bm{u}\cdot\nabla)\bm{u}+P((\Phi_1v)\nabla \varphi)+P\bm{f} && \text{in $(0,\infty) \times \Omega$}, \\
(\bm{\Phi}_3v)(0,\,\cdot\,) &= \bm{u}_0 && \text{in $\Omega$}.
\end{aligned}\right.
\end{equation}
Moreover, since $\lambda_1 \ge \lambda_2$, it holds by the estimates in Lemma \ref{nonlest} (i) 
and the estimate \eqref{phi1est} that
\begin{equation}
\begin{aligned}
&\|e^{\lambda_2 t}\bm{\Phi}_3v\|_{\mathbb{E}_{q,r}^0} \\
&\le C\left(\|\bm{u}_0\|_{B_{r,q}^{2-2/q}(\Omega)}+\|e^{\lambda_2 t}(\bm{u} \cdot \nabla)\bm{u}\|_{L^q(L^r(\Omega))}+\|e^{\lambda_2 t}(\Phi_1v)\nabla\varphi\|_{L^q(L^r(\Omega))}+\|e^{\lambda_2 t}\bm{f}\|_{L^q(L^r(\Omega))}\right) \\
&\le C\left(\llbracket (n_0,c_0,\bm{u}_0,\bm{f}) \rrbracket+\|v\|_{\mathbb{E}}^2+\|\nabla\varphi\|_{L^r(\Omega)}\|e^{\lambda_1 t}\Phi_1v\|_{\mathbb{E}_{q,r}^0}\right) \\
&\le C(1+\|\nabla\varphi\|_{L^r(\Omega)})\llbracket (n_0,c_0,\bm{u}_0,\bm{f}) \rrbracket +CM_S(1+\|\nabla\varphi\|_{L^r(\Omega)})(\|v\|_{\mathbb{E}}+\llbracket (n_0,c_0,\bm{u}_0,\bm{f}) \rrbracket)\|v\|_{\mathbb{E}}.
\end{aligned}
\end{equation}
Hence, by combining the estimates of $\Phi_1$, $\Phi_2$, and $\bm{\Phi}_3$, we have
\begin{equation}\label{phiest}
\begin{aligned}
\|\Phi v\|_{\mathbb{E}} &\le C(1+\|\nabla\varphi\|_{L^r(\Omega)})\llbracket (n_0,c_0,\bm{u}_0,\bm{f}) \rrbracket \\
&\quad +CM_S(1+\|\nabla\varphi\|_{L^r(\Omega)})(\|v\|_{\mathbb{E}}+\llbracket (n_0,c_0,\bm{u}_0,\bm{f}) \rrbracket)\|v\|_{\mathbb{E}},
\end{aligned}
\end{equation}
where $\Phi v \coloneqq (\Phi_1v,\Phi_2v,\bm{\Phi}_3v) \in \mathbb{E}$.
In addition, by summarizing the aforementioned discussions, we also observe that for every $v \in \mathbb{E}$, there exists a unique $\Phi v \coloneqq (\Phi_1v,\Phi_2v,\bm{\Phi}_3v) \in \mathbb{E}$ such that
\begin{equation}\label{phisol}
\left\{\begin{aligned}
\partial_t\Phi_1v-\Delta \Phi_1v &= -\nabla \cdot ((\widetilde{n}+\overline{n}_0)S\nabla \widetilde{c})-\bm{u}\cdot\nabla\widetilde{n} && \text{in $(0,\infty) \times \Omega$}, \\
\partial_t\Phi_2v+(1-\Delta_{\mathrm{N},1}) \Phi_2v &=\widetilde{n}-\bm{u} \cdot \nabla \widetilde{c} && \text{in $(0,\infty) \times \Omega$}, \\
\partial_t\bm{\Phi}_3v+A\bm{\Phi}_3v &= -P(\bm{u}\cdot\nabla)\bm{u}+P((\Phi_1v)\nabla \varphi)+P\bm{f} && \text{in $(0,\infty) \times \Omega$}, \\
\nabla\Phi_1v \cdot \bm{\nu} &= (\widetilde{n}+\overline{n}_0)S\nabla \widetilde{c} \cdot \bm{\nu} && \text{in $(0,\infty) \times \partial\Omega$}, \\
(\Phi_1v,\Phi_2v,\bm{\Phi}_3v)(0,\,\cdot\,) &= (n_0-\overline{n}_0,c_0,\bm{u}_0) && \text{in $\Omega$}.
\end{aligned}\right.
\end{equation}

Since it holds that
\begin{equation}
\begin{aligned}
\nabla \cdot ((\widetilde{n}+\overline{n}_0)S\nabla \widetilde{c})-\nabla \cdot ((\widetilde{n}^*+\overline{n}_0)S\nabla \widetilde{c}^*) &= \nabla \cdot ((\widetilde{n}-\widetilde{n}^*)S\nabla \widetilde{c})+\nabla \cdot ((\widetilde{n}^*+\overline{n}_0)S\nabla (\widetilde{c}-\widetilde{c}^*)), \\
(\widetilde{n}+\overline{n}_0)S\nabla \widetilde{c}\cdot\bm{\nu}-(\widetilde{n}^*+\overline{n}_0)S\nabla \widetilde{c}^*\cdot\bm{\nu} &= (\widetilde{n}-\widetilde{n}^*)S\nabla \widetilde{c} \cdot \bm{\nu}+(\widetilde{n}^*+\overline{n}_0)S\nabla (\widetilde{c}-\widetilde{c}^*) \cdot \bm{\nu}, \\
\bm{u}\cdot\nabla\widetilde{n}-\bm{u}^*\cdot\nabla\widetilde{n}^* &= \bm{u}\cdot\nabla (\widetilde{n}-\widetilde{n}^*)+(\bm{u}-\bm{u}^*)\cdot\nabla\widetilde{n}^*, \\
\widetilde{n}-\bm{u}\cdot\nabla \widetilde{c}-\widetilde{n}^*+\bm{u}^*\cdot\nabla \widetilde{c}^* &= \widetilde{n}-\widetilde{n}^*+(\bm{u}-\bm{u}^*) \cdot \nabla \widetilde{c}+\bm{u}^* \cdot \nabla (\widetilde{c}-\widetilde{c}^*), \\
(\bm{u}\cdot\nabla)\bm{u}-(\bm{u}^*\cdot\nabla)\bm{u}^* &= ((\bm{u}-\bm{u}^*)\cdot\nabla)\bm{u}+(\bm{u}^*\cdot\nabla)(\bm{u}-\bm{u}^*), \\
(\Phi_1v)\nabla \varphi-(\Phi_1v^*)\nabla \varphi &= (\Phi_1v-\Phi_1v^*)\nabla \varphi,
\end{aligned}
\end{equation}
in a similar manner to the aforementioned argument, we observe that
\begin{equation}
\begin{aligned}
\|e^{\lambda_1 t}(\Phi_1v-\Phi_1v^*)\|_{\mathbb{E}_{q,r}^0} &\le CM_S(\|v\|_{\mathbb{E}}+\|v^*\|_{\mathbb{E}}+\llbracket (n_0,c_0,\bm{u}_0,\bm{f}) \rrbracket)\|v-v^*\|_{\mathbb{E}}, \\
\|e^{\lambda_1 t}(\Phi_2v-\Phi_2v^*)\|_{\mathbb{E}_{q,r}^1} &\le C(\|v\|_{\mathbb{E}}+\|v^*\|_{\mathbb{E}})\|v-v^*\|_{\mathbb{E}}, \\
\|e^{\lambda_2 t}(\bm{\Phi}_3v-\bm{\Phi}_3v^*)\|_{\mathbb{E}_{q,r}^0} &\le C(\|v\|_{\mathbb{E}}+\|v^*\|_{\mathbb{E}})\|v-v^*\|_{\mathbb{E}}+C\|e^{\lambda_1 t}(\Phi_1v-\Phi_1v^*)\|_{\mathbb{E}_{q,r}^0}\|\nabla\varphi\|_{L^r(\Omega)} \\
&\le CM_S(1+\|\nabla\varphi\|_{L^r(\Omega)})(\|v\|_{\mathbb{E}}+\|v^*\|_{\mathbb{E}}+\llbracket (n_0,c_0,\bm{u}_0,\bm{f}) \rrbracket)\|v-v^*\|_{\mathbb{E}}
\end{aligned}
\end{equation}
for all $v,v^* \in \mathbb{E}$, which yield
\begin{equation}\label{phisubest}
\|\Phi v-\Phi v^*\|_{\mathbb{E}} \le CM_S(1+\|\nabla\varphi\|_{L^r(\Omega)})(\|v\|_{\mathbb{E}}+\|v^*\|_{\mathbb{E}}+\llbracket (n_0,c_0,\bm{u}_0,\bm{f}) \rrbracket)\|v-v^*\|_{\mathbb{E}}.
\end{equation}
Here, it should be emphasized that the aforementioned constant $C$ is independent of $S$ and $\varphi$.
Now we assume that $n_0$, $c_0$, $\bm{u}_0$, $\bm{f}$, and $\varphi$ satisfy
\begin{equation}\label{givensmall}
\llbracket (n_0,c_0,\bm{u}_0,\bm{f}) \rrbracket \le \frac{1}{8C^2M_S(1+\|\nabla\varphi\|_{L^r(\Omega)})^2}.
\end{equation}
We also introduce the closed subspace $\mathscr{B}(\mathbb{E})$ of $\mathbb{E}$ by setting
\begin{equation}\label{bedef}
\mathscr{B}(\mathbb{E}) \coloneqq \left\{ v \in \mathbb{E} \, \left| \, \|v\|_{\mathbb{E}} \le \frac{1}{4CM_S(1+\|\nabla\varphi\|_{L^r(\Omega)})} \right.\right\}.
\end{equation}
Here, the uniqueness of $\Phi v \in \mathbb{E}$ ensures that we may define the mapping $\Phi : \mathscr{B}(\mathbb{E}) \to \mathbb{E}$ by taking the unique solution $\Phi v \coloneqq (\Phi_1v,\Phi_2v,\bm{\Phi}_3v) \in \mathbb{E}$ to \eqref{phisol} for given $v \in \mathscr{B}(\mathbb{E})$.
Let us take $v,v^* \in \mathscr{B}(\mathbb{E})$.
Then we see by \eqref{phiest} and \eqref{phisubest} that
\begin{equation}
\begin{aligned}
\|\Phi v\|_{\mathbb{E}} &\le C(1+\|\nabla\varphi\|_{L^r(\Omega)}) \cdot \frac{1}{8C^2M_S(1+\|\nabla\varphi\|_{L^r(\Omega)})^2}+CM_S(1+\|\nabla\varphi\|_{L^r(\Omega)}) \\
&\quad \times \left\{ \frac{1}{4CM_S(1+\|\nabla\varphi\|_{L^r(\Omega)})}+\frac{1}{8C^2M_S(1+\|\nabla\varphi\|_{L^r(\Omega)})^2} \right\} \cdot \frac{1}{4CM_S(1+\|\nabla\varphi\|_{L^r(\Omega)})} \\
&\le \frac{1}{4CM_S(1+\|\nabla\varphi\|_{L^r(\Omega)})}
\end{aligned}
\end{equation}
and
\begin{equation}
\begin{aligned}
&\|\Phi v-\Phi v^*\|_{\mathbb{E}} \\
&\le CM_S(1+\|\nabla\varphi\|_{L^r(\Omega)}) \left\{\frac{1}{2CM_S(1+\|\nabla\varphi\|_{L^r(\Omega)})}+\frac{1}{8C^2M_S(1+\|\nabla\varphi\|_{L^r(\Omega)})^2}\right\}\|v-v^*\|_{\mathbb{E}} \\
&\le \frac{5}{8}\|v-v^*\|_{\mathbb{E}},
\end{aligned}
\end{equation}
which yield $\Phi v \in \mathscr{B}(\mathbb{E})$ for all $v \in \mathscr{B}(\mathbb{E})$.
Hence, the Banach fixed point theorem ensures the existence of a unique $v \in \mathscr{B}(\mathbb{E})$ such that $\Phi v=v$, i.e., $\Phi_1v=\widetilde{n}$, $\Phi_2v=\widetilde{c}$, and $\bm{\Phi}_3v=\bm{u}$.
Consequently, we may obtain a unique global strong solution $v=(\widetilde{n},\widetilde{c},\bm{u}) \in \mathscr{B}(\mathbb{E})$ to \eqref{shift} from the definition \eqref{phisol} of the mapping $\Phi$.

Noting that $v \in \mathscr{B}(\mathbb{E})$, we see by \eqref{phiest} that
\begin{equation}
\begin{aligned}
\|v\|_{\mathbb{E}} &\le C(1+\|\nabla\varphi\|_{L^r(\Omega)})\llbracket (n_0,c_0,\bm{u}_0,\bm{f}) \rrbracket \\
&\quad +CM_S(1+\|\nabla\varphi\|_{L^r(\Omega)})\left\{ \frac{1}{4CM_S(1+\|\nabla\varphi\|_{L^r(\Omega)})}+\frac{1}{8C^2M_S(1+\|\nabla\varphi\|_{L^r(\Omega)})^2} \right\}\|v\|_{\mathbb{E}} \\
&\le C(1+\|\nabla\varphi\|_{L^r(\Omega)})\llbracket (n_0,c_0,\bm{u}_0,\bm{f}) \rrbracket+\frac{3}{8}\|v\|_{\mathbb{E}},
\end{aligned}
\end{equation}
which yields
$$\|v\|_{\mathbb{E}} \le C(1+\|\nabla\varphi\|_{L^r(\Omega)})\llbracket (n_0,c_0,\bm{u}_0,\bm{f}) \rrbracket.$$
Concerning the Lipschitz continuity of the solution mapping, we suppose that
\begin{equation}\label{givenstarsmall}
\llbracket (n_0^*,c_0^*,\bm{u}_0^*,\bm{f}^*) \rrbracket \le \frac{1}{8C^2M_S(1+\|\nabla\varphi\|_{L^r(\Omega)})^2}
\end{equation}
like \eqref{givensmall} and also suppose that $v^*=(\widetilde{n}^*,\widetilde{c}^*,\bm{u}^*) \in \mathscr{B}(\mathbb{E})$ is a global solution to \eqref{shift} with the given datum $(n_0-\overline{n}_0,c_0,\bm{u}_0,\bm{f})$ replaced by $(n_0^*-\overline{n}_0^*,c_0^*,\bm{u}_0^*,\bm{f}^*)$.
Then, similarly to the derivation of the estimate \eqref{phisubest}, we deduce that
\begin{equation}
\begin{aligned}
\|e^{\lambda_1 t}(\widetilde{n}-\widetilde{n}^*)\|_{\mathbb{E}_{q,r}^0} &\le C\llbracket (n_0-n_0^*,c_0-c_0^*,\bm{u}_0-\bm{u}_0^*,\bm{f}-\bm{f}^*) \rrbracket \\
&\quad +CM_S(\|v\|_{\mathbb{E}}+\|v^*\|_{\mathbb{E}}+\llbracket (n_0,c_0,\bm{u}_0,\bm{f}) \rrbracket+\llbracket (n_0^*,c_0^*,\bm{u}_0^*,\bm{f}^*) \rrbracket)\|v-v^*\|_{\mathbb{E}}, \\
\|e^{\lambda_1 t}(\widetilde{c}-\widetilde{c}^*)\|_{\mathbb{E}_{q,r}^1} &\le C\llbracket (n_0-n_0^*,c_0-c_0^*,\bm{u}_0-\bm{u}_0^*,\bm{f}-\bm{f}^*) \rrbracket+C(\|v\|_{\mathbb{E}}+\|v^*\|_{\mathbb{E}})\|v-v^*\|_{\mathbb{E}},
\end{aligned}
\end{equation}
and
\begin{equation}
\begin{aligned}
&\|e^{\lambda_2 t}(\bm{u}-\bm{u}^*)\|_{\mathbb{E}_{q,r}^0} \\
&\le C\llbracket (n_0-n_0^*,c_0-c_0^*,\bm{u}_0-\bm{u}_0^*,\bm{f}-\bm{f}^*) \rrbracket \\
&\quad +C(\|v\|_{\mathbb{E}}+\|v^*\|_{\mathbb{E}})\|v-v^*\|_{\mathbb{E}}+C\|e^{\lambda_1 t}(\widetilde{n}-\widetilde{n}^*)\|_{\mathbb{E}_{q,r}^0}\|\nabla\varphi\|_{L^r(\Omega)} \\
&\le C(1+\|\nabla\varphi\|_{L^r(\Omega)})\llbracket (n_0-n_0^*,c_0-c_0^*,\bm{u}_0-\bm{u}_0^*,\bm{f}-\bm{f}^*) \rrbracket \\
&\quad +CM_S(1+\|\nabla\varphi\|_{L^r(\Omega)})(\|v\|_{\mathbb{E}}+\|v^*\|_{\mathbb{E}}+\llbracket (n_0,c_0,\bm{u}_0,\bm{f}) \rrbracket+\llbracket (n_0^*,c_0^*,\bm{u}_0^*,\bm{f}^*) \rrbracket)\|v-v^*\|_{\mathbb{E}}.
\end{aligned}
\end{equation}
Since $v,v^* \in \mathscr{B}(\mathbb{E})$, we see by \eqref{givensmall} and \eqref{givenstarsmall} that
\begin{equation}
\begin{aligned}
&\|v-v^*\|_{\mathbb{E}} \\
&\le C(1+\|\nabla\varphi\|_{L^r(\Omega)})\llbracket (n_0-n_0^*,c_0-c_0^*,\bm{u}_0-\bm{u}_0^*,\bm{f}-\bm{f}^*) \rrbracket \\
&\quad +CM_S(1+\|\nabla\varphi\|_{L^r(\Omega)})(\|v\|_{\mathbb{E}}+\|v^*\|_{\mathbb{E}}+\llbracket (n_0,c_0,\bm{u}_0,\bm{f}) \rrbracket+\llbracket (n_0^*,c_0^*,\bm{u}_0^*,\bm{f}^*) \rrbracket)\|v-v^*\|_{\mathbb{E}} \\
&\le C(1+\|\nabla\varphi\|_{L^r(\Omega)})\llbracket (n_0-n_0^*,c_0-c_0^*,\bm{u}_0-\bm{u}_0^*,\bm{f}-\bm{f}^*) \rrbracket \\
&\quad +CM_S(1+\|\nabla\varphi\|_{L^r(\Omega)})\left\{\frac{1}{2CM_S(1+\|\nabla\varphi\|_{L^r(\Omega)})}+\frac{1}{4C^2M_S(1+\|\nabla\varphi\|_{L^r(\Omega)})^2}\right\}\|v-v^*\|_{\mathbb{E}} \\
&\le C(1+\|\nabla\varphi\|_{L^r(\Omega)})\llbracket (n_0-n_0^*,c_0-c_0^*,\bm{u}_0-\bm{u}_0^*,\bm{f}-\bm{f}^*) \rrbracket+\frac{3}{4}\|v-v^*\|_{\mathbb{E}},
\end{aligned}
\end{equation}
which implies that
$$\|v-v^*\|_{\mathbb{E}} \le C(1+\|\nabla\varphi\|_{L^r(\Omega)})\llbracket (n_0-n_0^*,c_0-c_0^*,\bm{u}_0-\bm{u}_0^*,\bm{f}-\bm{f}^*) \rrbracket.$$

Finally, since our solution $v=(\widetilde{n},\widetilde{c},\bm{u}) \in \mathscr{B}(\mathbb{E})$ satisfies \eqref{shift}, by setting
$$n(t,x) \coloneqq \widetilde{n}(t,x)+\overline{n}_0, \quad c(t,x) \coloneqq \widetilde{c}(t,x)+(1-e^{-t})\overline{n}_0$$
for $(t,x) \in (0,\infty) \times \Omega$, we observe that $(n,c,\bm{u})$ satisfies System \eqref{ABS} with the properties \eqref{solclass}, \eqref{solest}, and \eqref{sollip}.
The validity of the nonlinear boundary condition \eqref{VNBC} follows from \eqref{PNBC}.
This completes the proof of Theorem \ref{globalsol}.

\section{Smoothness and non-negativity of global solutions}
\label{sec-5}

In this section, we show Theorem \ref{solreg}.
We first verify that the global strong solution $(n,c,\bm{u})$ to System \eqref{ABS} has higher regularities so that $(n,c,\bm{u})$ satisfies \eqref{ABS} in a classical sense, provided that $\varphi$, $\bm{f}$, and $S$ belong to better spaces such that \eqref{givenreg1} and \eqref{givenreg2}.
By virtue of the smoothness of the solution, we also obtain the non-negativity of $n$ and $c$.
To this end, we go back to the solution $(\widetilde{n},\widetilde{c},\bm{u})$ to \eqref{shift} and investigate these regularities.

\subsection{Regularities of $\bm{u}$}

Here, let us give the regularities of $\bm{u}$, which may be achieved via the classical semigroup theory.
Precisely speaking, we will make use of the mapping properties of the Stokes semigroup $e^{-tA}$ in the Besov spaces.
In our argument, the boundedness of the Helmholtz projection $P$ in the Besov spaces is necessary.
We will give the proof in the appendix for the convenience of the readers.

\begin{prop}\label{besovstokes}
Let $1<r_*<\infty$ and $0<s<\infty$. Then the following statements hold:

\begin{enumerate}
\item The Helmholtz projection operator $P$ is bounded from $B_{r_*,\infty}^s(\Omega)^N$ onto itself.

\item Let $0<s_0<\infty$. Then for every $\bm{\psi} \in L^{r_*}(\Omega)^N$, it holds that $e^{-tA}P\bm{\psi} \in B_{r_*,1}^{s_0}(\Omega)^N$ for all $0<t<\infty$ with the estimate
$$\|e^{-tA}P\bm{\psi}\|_{B_{r_*,1}^{s_0}(\Omega)} \le Ct^{-s_0/2}\|\bm{\psi}\|_{L^{r_*}(\Omega)},$$
where $C>0$ is a constant independent of $t$ and $\bm{\psi}$.
In addition, if $\bm{\psi} \in B_{r_*,\infty}^s(\Omega)^N$, then there holds $e^{-tA}P\bm{\psi} \in B_{r_*,1}^{s+s_0}(\Omega)^N$ with
$$\|e^{-tA}P\bm{\psi}\|_{B_{r_*,1}^{s+s_0}(\Omega)} \le Ct^{-s_0/2}\|\bm{\psi}\|_{B_{r_*,\infty}^s(\Omega)}.$$
\end{enumerate}
\end{prop}

By applying Proposition \ref{besovstokes}, we obtain the regularities of $\bm{u}$ as follows.

\begin{lemm}\label{ureg}
Assume that all assumptions in Theorem \ref{globalsol} are satisfied.
In addition, suppose \eqref{givenreg1} with $0<\theta_0<\min\{1,2-2/q-N/r\}$.
Then the function $\bm{u}$ obtained in Theorem \ref{globalsol} satisfies
$$\bm{u} \in BUC([0,\infty);C(\overline{\Omega})^N) \cap C((0,\infty);C^{2+\theta}(\overline{\Omega})^N) \cap C^1((0,\infty);C^{\theta}(\overline{\Omega})^N)$$
for all $0<\theta<\theta_0$.
\end{lemm}

\begin{proof}
First we note that the regularity $\bm{u} \in BUC([0,\infty);C(\overline{\Omega})^N)$ has been obtained from \eqref{solclass} and Corollary \ref{ncuemb} (ii).
Since Proposition \ref{besovemb} yields $B_{r,q}^{2-2/q}(\Omega) \subset C^{\theta_0}(\overline{\Omega})$, we see by Proposition \ref{traceemb} that $\widetilde{n} \in BUC([0,\infty);C^{\theta_0}(\overline{\Omega}))$ with the estimate
\begin{equation}\label{ntraceemb}
\|\widetilde{n}\|_{L^{\infty}(C^{\theta_0}(\overline{\Omega}))} \le C\|\widetilde{n}\|_{\mathbb{E}_{q,r}^0}.
\end{equation}
We rewrite the third equation of \eqref{shift} as an integral form:
\begin{equation}
\begin{aligned}
\bm{u}(t,\,\cdot\,) &= e^{-tA}\bm{u}_0-\int_0^te^{-(t-\tau)A}P((\bm{u}\cdot\nabla)\bm{u}-\widetilde{n}\nabla \varphi-\bm{f})(\tau,\,\cdot\,) \, d\tau \\
&= e^{-(t/2)A}P\bm{u}(t/2,\,\cdot\,)-\int_{t/2}^te^{-(t-\tau)A}P((\bm{u}\cdot\nabla)\bm{u}-\widetilde{n}\nabla \varphi-\bm{f})(\tau,\,\cdot\,) \, d\tau.
\end{aligned}
\end{equation}
Corollary \ref{ncuemb} (ii) and the estimate \eqref{ntraceemb} imply that
\begin{equation}
\begin{aligned}
&\|((\bm{u}\cdot\nabla)\bm{u}-\widetilde{n}\nabla \varphi-\bm{f})(\tau,\,\cdot\,)\|_{L^r(\Omega)} \\
&\le C\left(\|\bm{u}(\tau,\,\cdot\,)\|_{L^{\infty}(\Omega)}\|\bm{u}(\tau,\,\cdot\,)\|_{W^{1,r}(\Omega)}+\|\widetilde{n}(\tau,\,\cdot\,)\|_{L^{\infty}(\Omega)}\|\nabla\varphi\|_{L^r(\Omega)}+\|\bm{f}(\tau,\,\cdot\,)\|_{L^r(\Omega)}\right) \\
&\le C\left(\|\bm{u}\|_{\mathbb{E}_{q,r}^0}^2+\|\widetilde{n}\|_{\mathbb{E}_{q,r}^0}\|\varphi\|_{C^{1+\theta_0}(\overline{\Omega})}+\|\bm{f}(\tau,\,\cdot\,)\|_{C^{\theta_0}(\overline{\Omega})} \right)
\end{aligned}
\end{equation}
for all $0<\tau<\infty$.
Therefore, it holds by Proposition \ref{besovstokes} (ii) that
\begin{equation}
\begin{aligned}
&\|\bm{u}(t,\,\cdot\,)\|_{B_{r,1}^{2-\eta}(\Omega)} \\
&\le Ct^{-(2-\eta)/2}\|\bm{u}(t/2,\,\cdot\,)\|_{L^r(\Omega)}+C\int_{t/2}^t(t-\tau)^{-(2-\eta)/2}\|((\bm{u}\cdot\nabla)\bm{u}-\widetilde{n}\nabla \varphi-\bm{f})(\tau,\,\cdot\,)\|_{L^r(\Omega)} \, d\tau \\
&\le Ct^{-1+\eta/2}\|\bm{u}\|_{\mathbb{E}_{q,r}^0}+Ct^{\eta/2}\left(\|\bm{u}\|_{\mathbb{E}_{q,r}^0}^2+\|\widetilde{n}\|_{\mathbb{E}_{q,r}^0}\|\varphi\|_{C^{1+\theta_0}(\overline{\Omega})}+\sup_{t/2 \le \tau \le t}\|\bm{f}(\tau,\,\cdot\,)\|_{C^{\theta_0}(\overline{\Omega})} \right)
\end{aligned}
\end{equation}
for all $0<t<\infty$ and $0<\eta<1$.
Since $r>N$, by taking $\eta>0$ sufficiently small and applying Proposition \ref{besovemb}, we observe that
$$\bm{u} \in C((0,\infty);B_{r,1}^{2-\eta}(\Omega)^N) \subset C((0,\infty);C^{1+\theta_*}(\overline{\Omega})^N)$$
with some $0<\theta_* \le \theta_0$.
Moreover, similarly to the derivation of the above estimate, we see by \eqref{ntraceemb} that
\begin{equation}
\begin{aligned}
&\|((\bm{u}\cdot\nabla)\bm{u}-\widetilde{n}\nabla \varphi-\bm{f})(\tau,\,\cdot\,)\|_{B_{r_*,\infty}^{\theta_*}(\Omega)} \\
&\le C\|((\bm{u}\cdot\nabla)\bm{u}-\widetilde{n}\nabla \varphi-\bm{f})(\tau,\,\cdot\,)\|_{C^{\theta_*}(\overline{\Omega})} \\
&\le C\left(\|\bm{u}(\tau,\,\cdot\,)\|_{C^{\theta_*}(\overline{\Omega})}\|\nabla\bm{u}(\tau,\,\cdot\,)\|_{C^{\theta_*}(\overline{\Omega})}+\|\widetilde{n}(\tau,\,\cdot\,)\|_{C^{\theta_0}(\overline{\Omega})}\|\nabla\varphi\|_{C^{\theta_0}(\overline{\Omega})}+\|\bm{f}(\tau,\,\cdot\,)\|_{C^{\theta_0}(\overline{\Omega})} \right) \\
&\le C\left(\|\bm{u}(\tau,\,\cdot\,)\|_{C^{1+\theta_*}(\overline{\Omega})}^2+\|\widetilde{n}\|_{\mathbb{E}_{q,r}^0}\|\varphi\|_{C^{1+\theta_0}(\overline{\Omega})}+\|\bm{f}(\tau,\,\cdot\,)\|_{C^{\theta_0}(\overline{\Omega})} \right)
\end{aligned}
\end{equation}
for all $0<\tau<\infty$ and $1<r_*<\infty$.
Hence Corollary \ref{ncuemb} (ii) and Proposition \ref{besovstokes} (ii) yield
\begin{equation}
\begin{aligned}
&\|\bm{u}(t,\,\cdot\,)\|_{B_{r*,1}^{2+\theta_*-\eta}(\Omega)} \\
&\le Ct^{-(2+\theta_*-\eta)/2}\|\bm{u}(t/2,\,\cdot\,)\|_{L^{\infty}(\Omega)}+C\int_{t/2}^t(t-\tau)^{-(2-\eta)/2}\|((\bm{u}\cdot\nabla)\bm{u}-\widetilde{n}\nabla \varphi-\bm{f})(\tau,\,\cdot\,)\|_{B_{r_*,\infty}^{\theta_*}(\Omega)} \, d\tau \\
&\le Ct^{-1-\theta_*/2+\eta/2}\|\bm{u}\|_{\mathbb{E}_{q,r}^0}+Ct^{\eta/2}\sup_{t/2 \le \tau \le t}\left(\|\bm{u}(\tau,\,\cdot\,)\|_{C^{1+\theta_*}(\overline{\Omega})}^2+\|\widetilde{n}\|_{\mathbb{E}_{q,r}^0}\|\varphi\|_{C^{1+\theta_0}(\overline{\Omega})}+\|\bm{f}(\tau,\,\cdot\,)\|_{C^{\theta_0}(\overline{\Omega})}\right)
\end{aligned}
\end{equation}
for all $0<t<\infty$ and $0<\eta<1$.
Thus we have
$$\bm{u} \in C((0,\infty);B_{r_*,1}^{2+\theta_*-\eta}(\Omega)^N) \subset C((0,\infty);C^{1+\theta_0}(\overline{\Omega})^N)$$
by choosing $r_*$ sufficiently large and using Proposition \ref{besovemb}.
Repeating the above argument with $\theta_*$ replaced by $\theta_0$, we also see by Proposition \ref{besovemb} that
$$\bm{u} \in C((0,\infty);B_{r_*,1}^{2+\theta_0-\eta}(\Omega)^N) \subset C((0,\infty);C^{2+\theta_0-2\eta}(\overline{\Omega})^N)$$
for arbitrarily small $0<\eta<1$.
In addition, using the third equation of \eqref{shift} and Proposition \ref{besovstokes} (i), we obtain
\begin{equation}
\begin{aligned}
&\|\partial_t\bm{u}(t,\,\cdot\,)\|_{B_{r_*,\infty}^{\theta_0-2\eta}(\Omega)} \\
&\le C\left(\|\bm{u}(t,\,\cdot\,)\|_{C^{2+\theta_0-2\eta}(\overline{\Omega})}+\|P((\bm{u}\cdot\nabla)\bm{u}-\widetilde{n}\nabla \varphi-\bm{f})(t,\,\cdot\,)\|_{B_{r_*,\infty}^{\theta_0}(\Omega)} \right) \\
&\le C\left(\|\bm{u}(t,\,\cdot\,)\|_{C^{2+\theta_0-2\eta}(\overline{\Omega})}+\|\bm{u}(t,\,\cdot\,)\|_{C^{1+\theta_0}(\overline{\Omega})}^2+\|\widetilde{n}\|_{\mathbb{E}_{q,r}^0}\|\varphi\|_{C^{1+\theta_0}(\overline{\Omega})}+\|\bm{f}(t,\,\cdot\,)\|_{C^{\theta_0}(\overline{\Omega})} \right)
\end{aligned}
\end{equation}
for all $0<t<\infty$ and $0<\eta<\theta_0/2$.
Thus we have
$$\bm{u} \in C^1((0,\infty);B_{r_*,\infty}^{\theta_0-2\eta}(\Omega)^N) \subset C^1((0,\infty);C^{\theta_0-3\eta}(\overline{\Omega})^N)$$
due to Proposition \ref{besovemb}, which yields the desired result.
\end{proof}

\subsection{Regularities of $n$ and $c$}

It remains to show the smoothness of $n$ and $c$ with the aid of mapping properties of the heat semigroup $e^{t\Delta}$ on $\mathbb{R}^N$.
To this end, we shall use the cut-off technique and extensions of the functions to the whole space so that we may avoid considering the nonlinear boundary condition $\nabla \widetilde{n}\cdot\bm{\nu}=(\widetilde{n}+\overline{n}_0)S(t,x)\nabla \widetilde{c} \cdot \bm{\nu}$ on $\partial \Omega$, see \eqref{shift}.
Recall the cut-off function $\chi_{\delta} \in C_0^{\infty}(\Omega)$ defined by \eqref{chidef} and recall that there holds $\chi_{\delta}(x)\psi(x)=0$ for all $\psi \in C(\overline{\Omega})$ if $x \in \partial\Omega$.
Furthermore, we introduce the zero-extension operator $\mathcal{E}$ by setting
$$\mathcal{E}[\psi](x) \coloneqq \left\{\begin{array}{cl}
\psi(x) & \text{if $x \in \Omega$}, \\
0 & \text{if $x \in \mathbb{R}^N \setminus \Omega$}
\end{array}\right.$$
for $\psi \in C(\overline{\Omega})$.
Then it holds by \eqref{shift} that
\begin{equation}\label{cutsys}
\left\{\begin{aligned}
\partial_t\mathcal{E}[\chi_{\delta}\widetilde{n}]-\Delta\mathcal{E}[\chi_{\delta}\widetilde{n}] &= \mathcal{E}[\chi_{\delta}(\partial_t\widetilde{n}-\Delta\widetilde{n}) -2\nabla\chi_{\delta} \cdot \nabla\widetilde{n}-\widetilde{n}\Delta\chi_{\delta}] \\
&= -\mathcal{E}[\chi_{\delta}\nabla \cdot ((\widetilde{n}+\overline{n}_0)S\nabla \widetilde{c})+\chi_{\delta}(\bm{u} \cdot \nabla \widetilde{n})+2\nabla\chi_{\delta} \cdot \nabla\widetilde{n}+\widetilde{n}\Delta\chi_{\delta}], \\
\partial_t\mathcal{E}[\chi_{\delta}\widetilde{c}]+(1-\Delta)\mathcal{E}[\chi_{\delta}\widetilde{c}] &=\mathcal{E}[\chi_{\delta}(\partial_t\widetilde{c}+(1-\Delta)\widetilde{c})-2\nabla\chi_{\delta} \cdot \nabla\widetilde{c}-\widetilde{c}\Delta\chi_{\delta}] \\
&= \mathcal{E}[\chi_{\delta}\widetilde{n}-\chi_{\delta}(\bm{u} \cdot \nabla \widetilde{c})-2\nabla\chi_{\delta} \cdot \nabla\widetilde{c}-\widetilde{c}\Delta\chi_{\delta}]
\end{aligned}\right.
\end{equation}
in $(0,\infty) \times \mathbb{R}^N$.
Hence, setting the right-hand sides by $F_{\delta}$ and $G_{\delta}$, i.e.,
\begin{equation}\label{FG}
\left\{\begin{aligned}
F_{\delta}(t,x) &\coloneqq (\chi_{\delta}\nabla \cdot ((\widetilde{n}+\overline{n}_0)S\nabla \widetilde{c})+\chi_{\delta}(\bm{u} \cdot \nabla \widetilde{n})+2\nabla\chi_{\delta} \cdot \nabla\widetilde{n}+\widetilde{n}\Delta\chi_{\delta})(t,x), \\
G_{\delta}(t,x) &\coloneqq (\chi_{\delta}\widetilde{n}-\chi_{\delta}(\bm{u} \cdot \nabla \widetilde{c})-2\nabla\chi_{\delta} \cdot \nabla\widetilde{c}-\widetilde{c}\Delta\chi_{\delta})(t,x)
\end{aligned}\right.
\end{equation}
for $(t,x) \in (0,\infty) \times \Omega$, respectively, we have the following integral system
\begin{equation}\label{INT}
\left\{\begin{aligned}
\mathcal{E}[\chi_{\delta}\widetilde{n}(t,\,\cdot\,)] &= e^{t\Delta}\mathcal{E}[\chi_{\delta}(n_0-\overline{n}_0)] -\int_0^te^{(t-\tau)\Delta}\mathcal{E}[F_{\delta}(\tau,\,\cdot\,)] \, d\tau \\
&=e^{(t/2)\Delta}\mathcal{E}[\chi_{\delta}\widetilde{n}(t/2,\,\cdot\,)]-\int_{t/2}^te^{(t-\tau)\Delta}\mathcal{E}[F_{\delta}(\tau,\,\cdot\,)] \, d\tau, \\
\mathcal{E}[\chi_{\delta}\widetilde{c}(t,\,\cdot\,)] &= e^{-t}e^{t\Delta}\mathcal{E}[\chi_{\delta}c_0]+\int_0^te^{-(t-\tau)}e^{(t-\tau)\Delta}\mathcal{E}[G_{\delta}(\tau,\,\cdot\,)] \, d\tau \\
&= e^{-t/2}e^{(t/2)\Delta}\mathcal{E}[\chi_{\delta}\widetilde{c}(t/2,\,\cdot\,)]+\int_{t/2}^te^{-(t-\tau)}e^{(t-\tau)\Delta}\mathcal{E}[G_{\delta}(\tau,\,\cdot\,)] \, d\tau
\end{aligned}\right.
\end{equation}
for $0<t<\infty$, where $e^{t\Delta}$ denotes the heat semigroup defined by
$$(e^{t\Delta}\psi)(x) \coloneqq \frac{1}{(4\pi t)^{N/2}}\int_{\mathbb{R}^N}e^{-|x-y|^2/(4t)}\psi(y)\,dy, \quad (t,x) \in (0,\infty) \times \mathbb{R}^N$$
for the functions $\psi \in C(\mathbb{R}^N)$ having a compact support.
In what follows we show the regularities of $\widetilde{n}$ and $\widetilde{c}$ on an arbitrary compact subset $K \subset \Omega$ by using the integral form \eqref{INT}.
We begin with the estimates of the heat semigroup and the fractional Leibniz rule in Besov spaces, see, e.g., Kozono--Ogawa--Taniuchi \cite{kozonoogawataniuchi}*{Lemma 2.2 (i)}, and Chae \cite{chae}*{Lemma 2.2}.

\begin{prop}\label{besovest}
Let $1 \le r_* \le \infty$. Then the following statements hold:

\begin{enumerate}
\item Let $0<s_0 \le s_1<\infty$.
If $\psi \in L^{r_*}(\mathbb{R}^N)$, then it holds that $e^{t\Delta}\psi \in B_{r_*,\infty}^{s_1}(\mathbb{R}^N)$ for all $0<t<\infty$ with the estimate
$$\|e^{t\Delta}\psi\|_{B_{r_*,\infty}^{s_1}(\mathbb{R}^N)} \le C(1+t^{-s_1/2})\|\psi\|_{L^{r_*}(\mathbb{R}^N)},$$
where $C>0$ is a constant independent of $t$ and $\psi$.
In addition, if $\psi \in B_{r_*,\infty}^{s_0}(\mathbb{R}^N)$, then it holds that $e^{t\Delta}\psi \in B_{r_*,\infty}^{s_1}(\mathbb{R}^N)$ for all $0<t<\infty$ with the estimate
$$\|e^{t\Delta}\psi\|_{B_{r_*,\infty}^{s_1}(\mathbb{R}^N)} \le C(1+t^{-(s_1-s_0)/2})\|\psi\|_{B_{r_*,\infty}^{s_0}(\mathbb{R}^N)}.$$

\item Let $0<s<\infty$.
Assume that $\psi_0 \in B_{\infty,\infty}^s(\mathbb{R}^N)$ and $\psi_1 \in B_{r_*,\infty}^s(\mathbb{R}^N)$.
Then it holds that $\psi_0\psi_1 \in B_{r_*,\infty}^s(\mathbb{R}^N)$ with the estimate
$$\|\psi_0\psi_1\|_{B_{r_*,\infty}^s(\mathbb{R}^N)} \le C\left(\|\psi_0\|_{B_{\infty,\infty}^s(\mathbb{R}^N)}\|\psi_1\|_{L^{r_*}(\mathbb{R}^N)}+\|\psi_0\|_{L^{\infty}(\mathbb{R}^N)}\|\psi_1\|_{B_{r_*,\infty}^s(\mathbb{R}^N)}\right),$$
where $C>0$ is a constant independent of $\psi_0$ and $\psi_1$.
\end{enumerate}
\end{prop}

We also verify the following estimate to deal with the term $\chi_{\delta}\nabla\psi$.

\begin{prop}\label{cutest}
Let $1 \le r_* \le \infty$ and $0<s<\infty$.
Suppose that $\psi \in B_{r_*,\infty}^{s+1}(K)$, where $K \subset \Omega$ is an arbitrary compact subset.
Then it holds that
\begin{equation}
\|\mathcal{E}[\chi_{\delta}\nabla\psi]\|_{B_{r_*,\infty}^s(\mathbb{R}^N)} \le C_{\delta}\|\mathcal{E}[\chi_{\delta/2}\psi]\|_{B_{r_*,\infty}^{s+1}(\mathbb{R}^N)}
\end{equation}
for arbitrarily small $0<\delta<1$, where $C_{\delta}>0$ is a constant independent of $\psi$.
\end{prop}

\begin{proof}
Since $\chi_{\delta/2}(x)=1$ for all $x \in \Omega$ such that $\inf_{y \in \partial\Omega}|x-y| \ge \delta/2$ from the definition \eqref{chidef}, we have
\begin{equation}\label{halfdelta}
\chi_{\delta}(x)\nabla\psi(x)=\chi_{\delta}(x)\nabla(\chi_{\delta/2}(x)\psi(x))
\end{equation}
for all $x \in \Omega$.
Hence, Proposition \ref{besovest} (ii) implies that
\begin{equation}
\begin{aligned}
\|\mathcal{E}[\chi_{\delta}\nabla\psi]\|_{B_{r_*,\infty}^s(\mathbb{R}^N)} &=\|\mathcal{E}[\chi_{\delta}\nabla(\chi_{\delta/2}\psi)]\|_{B_{r_*,\infty}^s(\mathbb{R}^N)} =\|\mathcal{E}[\chi_{\delta}]\nabla\mathcal{E}[\chi_{\delta/2}\psi]\|_{B_{r_*,\infty}^s(\mathbb{R}^N)} \\
&\le C\left(\|\chi_{\delta}\|_{B_{\infty,\infty}^s(\Omega)}\|\mathcal{E}[\chi_{\delta/2}\psi]\|_{B_{r_*,1}^1(\mathbb{R}^N)}+\|\chi_{\delta}\|_{L^{\infty}(\Omega)}\|\mathcal{E}[\chi_{\delta/2}\psi]\|_{B_{r_*,\infty}^{s+1}(\mathbb{R}^N)}\right) \\
&\le C_{\delta}\|\mathcal{E}[\chi_{\delta/2}\psi]\|_{B_{r_*,\infty}^{s+1}(\mathbb{R}^N)},
\end{aligned}
\end{equation}
which yields the desired result.
\end{proof}

Next we give the estimates of $F_{\delta}$ and $G_{\delta}$ defined by \eqref{FG}.

\begin{lemm}\label{FGest}
Assume that all assumptions in Theorem \ref{globalsol} are satisfied. Let $K \subset \Omega$ be an arbitrary compact subset. Then the following statements hold:
\begin{enumerate}
\item Let $r \le r_* \le \infty$, $0<\theta_0<\min\{1,2-2/q-N/r\}$, and $0<s \le \theta_0$.
Moreover, suppose that
\begin{equation}
\begin{aligned}
S &\in C((0,\infty);C^{1+\theta_0}(K)^{N^2}), & \bm{u} &\in C((0,\infty);C^{\theta_0}(K)^N), \\
\widetilde{n} &\in C((0,\infty);B_{r_*,\infty}^{s+1}(K)), & \widetilde{c} &\in C((0,\infty);B_{r_*,\infty}^{s+2}(K)).
\end{aligned}
\end{equation}
Then it holds that
\begin{equation}
\sup_{t/2 \le \tau \le t}\|\mathcal{E}[F_{\delta}(\tau,\,\cdot\,)]\|_{B_{r_*,\infty}^s(\mathbb{R}^N)} \le C_{\delta}M_{1,\delta}(r_*,s,t)
\end{equation}
for all $0<t<\infty$ and for arbitrarily small $0<\delta<1$, where
\begin{equation}\label{m1def}
\begin{aligned}
M_{1,\delta}(r_*,s,t) &\coloneqq \sup_{t/2 \le \tau \le t}\left\{\left(\|\chi_{\delta/2}\widetilde{n}(\tau,\,\cdot\,)\|_{B_{r_*,\infty}^{s+1}(\Omega)}+|\overline{n}_0|+1\right) \right. \\
&\quad\left.{}\times\left(\|\chi_{\delta/2}S(\tau,\,\cdot\,)\|_{C^{1+\theta_0}(\overline{\Omega})}\|\chi_{\delta/2}\widetilde{c}(\tau,\,\cdot\,)\|_{B_{r_*,\infty}^{s+2}(\Omega)}+\|\chi_{\delta}\bm{u}(\tau,\,\cdot\,)\|_{C^{\theta_0}(\overline{\Omega})}+1\right)\right\}
\end{aligned}
\end{equation}
and $C_{\delta}>0$ is a constant independent of $t$, $n_0$, $S$, $\bm{u}$, $\widetilde{n}$, and $\widetilde{c}$.

\item Let $N<r_* \le \infty$ and $0<s<\infty$.
Moreover, suppose that
\begin{equation}
\bm{u} \in C((0,\infty);B_{\infty,\infty}^{s+1}(K)^N), \quad \widetilde{n} \in C((0,\infty);B_{r_*,\infty}^{s+1}(K)), \quad \widetilde{c} \in C((0,\infty);B_{r_*,\infty}^{s+2}(K)).
\end{equation}
Then it holds that
\begin{equation}
\sup_{t/2 \le \tau \le t}\|\mathcal{E}[G_{\delta}(\tau,\,\cdot\,)]\|_{B_{r_*,\infty}^{s+1}(\mathbb{R}^N)} \le C_{\delta}M_{2,\delta}(r_*,s,t)
\end{equation}
for all $0<t<\infty$ and for arbitrarily small $0<\delta<1$, where
\begin{equation}\label{m2def}
\begin{aligned}
&M_{2,\delta}(r_*,s,t) \\
&\coloneqq \sup_{t/2 \le \tau \le t}\left\{\left(\|\chi_{\delta}\bm{u}(\tau,\,\cdot\,)\|_{B_{\infty,\infty}^{s+1}(\Omega)}+1\right)\|\chi_{\delta/2}\widetilde{c}(\tau,\,\cdot\,)\|_{B_{r_*,\infty}^{s+2}(\Omega)}+\|\chi_{\delta/2}\widetilde{n}(\tau,\,\cdot\,)\|_{B_{r_*,\infty}^{s+1}(\Omega)}\right\}
\end{aligned}
\end{equation}
and $C_{\delta}>0$ is a constant independent of $t$, $\bm{u}$, $\widetilde{n}$, and $\widetilde{c}$.
\end{enumerate}
\end{lemm}

\begin{proof}
\textbf{(i)} Noting the identity \eqref{halfdelta}, we have a similar identity such that
\begin{align}
&\chi_{\delta}(x)\nabla\cdot ((\widetilde{n}(\tau,x)+\overline{n}_0)S(\tau,x)\nabla\widetilde{c}(\tau,x)) \\
&=\chi_{\delta}(x)\nabla\cdot \left\{(\widetilde{n}(\tau,x)+\overline{n}_0)\chi_{\delta/2}(x)S(\tau,x)\nabla(\chi_{\delta/2}(x)\widetilde{c}(\tau,x)) \right\}
\end{align}
for all $(\tau,x) \in (0,\infty) \times \Omega$.
Since $B_{r_*,\infty}^{s+1}(\mathbb{R}^N)$ is a Banach algebra due to $r_*>N$, we see by Proposition \ref{cutest} that
\begin{equation}
\begin{aligned}
&\|\mathcal{E}[\chi_{\delta}\nabla\cdot ((\widetilde{n}+\overline{n}_0)S\nabla\widetilde{c})](\tau,\,\cdot\,)\|_{B_{r_*,\infty}^s(\mathbb{R}^N)} \\
&=\|\mathcal{E}[\chi_{\delta}\nabla\cdot ((\widetilde{n}+\overline{n}_0)\chi_{\delta/2}S\nabla(\chi_{\delta/2}\widetilde{c}))](\tau,\,\cdot\,)\|_{B_{r_*,\infty}^s(\mathbb{R}^N)} \\
&\le C_{\delta}\|\mathcal{E}[\chi_{\delta/2}(\widetilde{n}+\overline{n}_0)\chi_{\delta/2}S\nabla(\chi_{\delta/2}\widetilde{c})](\tau,\,\cdot\,)\|_{B_{r_*,\infty}^{s+1}(\mathbb{R}^N)} \\
&\le C_{\delta}(\|\chi_{\delta/2}\widetilde{n}(\tau,\,\cdot\,)\|_{B_{r_*,\infty}^{s+1}(\Omega)}+|\overline{n}_0|)\|\chi_{\delta/2}S(\tau,\,\cdot\,)\|_{B_{r_*,\infty}^{s+1}(\Omega)}\|\chi_{\delta/2}\widetilde{c}(\tau,\,\cdot\,)\|_{B_{r_*,\infty}^{s+2}(\Omega)} \\
&\le C_{\delta}M_{1,\delta}(r_*,s,t)
\end{aligned}
\end{equation}
for all $0<t<\infty$ and $t/2 \le \tau \le t$.
Similarly to the derivation of the above estimate, it holds by Proposition \ref{besovest} (ii) that
\begin{equation}
\begin{aligned}
&\|\mathcal{E}[\chi_{\delta}(\bm{u}\cdot\nabla\widetilde{n})](\tau,\,\cdot\,)\|_{B_{r_*,\infty}^s(\mathbb{R}^N)} \\
&= \|\mathcal{E}[\chi_{\delta}\bm{u}(\tau,\,\cdot\,)]\cdot\nabla\mathcal{E}[\chi_{\delta/2}\widetilde{n}(\tau,\,\cdot\,)]\|_{B_{r_*,\infty}^s(\mathbb{R}^N)} \\
&\le C\left(\|\chi_{\delta}\bm{u}(\tau,\,\cdot\,)\|_{B_{\infty,\infty}^s(\Omega)}\|\chi_{\delta/2}\widetilde{n}(\tau,\,\cdot\,)\|_{B_{r_*,1}^1(\Omega)}+\|\chi_{\delta}\bm{u}(\tau,\,\cdot\,)\|_{L^{\infty}(\Omega)}\|\chi_{\delta/2}\widetilde{n}(\tau,\,\cdot\,)\|_{B_{r_*,\infty}^{s+1}(\Omega)}\right) \\
&\le CM_{1,\delta}(r_*,s,t)
\end{aligned}
\end{equation}
and
\begin{equation}
\begin{aligned}
&\|2\mathcal{E}[\nabla\chi_{\delta}\cdot\nabla\widetilde{n}(\tau,\,\cdot\,)]\|_{B_{r_*,\infty}^s(\mathbb{R}^N)}+\|\mathcal{E}[\widetilde{n}(\tau,\,\cdot\,)\Delta\chi_{\delta}]\|_{B_{r_*,\infty}^s(\mathbb{R}^N)} \\
&= 2\|\nabla\mathcal{E}[\chi_{\delta}]\cdot\nabla\mathcal{E}[\chi_{\delta/2}\widetilde{n}(\tau,\,\cdot\,)]\|_{B_{r_*,\infty}^s(\mathbb{R}^N)}+\|\mathcal{E}[\chi_{\delta/2}\widetilde{n}(\tau,\,\cdot\,)]\Delta\mathcal{E}[\chi_{\delta}]\|_{B_{r_*,\infty}^s(\mathbb{R}^N)} \\
&\le C\left(\|\chi_{\delta}\|_{B_{\infty,\infty}^{s+1}(\Omega)}\|\chi_{\delta/2}\widetilde{n}(\tau,\,\cdot\,)\|_{B_{r_*,1}^1(\Omega)}+\|\chi_{\delta}\|_{B_{\infty,1}^1(\Omega)}\|\chi_{\delta/2}\widetilde{n}(\tau,\,\cdot\,)\|_{B_{r_*,\infty}^{s+1}(\Omega)}\right) \\
&\quad +C\left(\|\chi_{\delta}\|_{B_{\infty,\infty}^{s+2}(\Omega)}\|\chi_{\delta/2}\widetilde{n}(\tau,\,\cdot\,)\|_{L^{r_*}(\Omega)}+\|\chi_{\delta}\|_{B_{\infty,1}^2(\Omega)}\|\chi_{\delta/2}\widetilde{n}(\tau,\,\cdot\,)\|_{B_{r_*,\infty}^s(\Omega)}\right) \\
&\le C_{\delta}M_{1,\delta}(r_*,s,t)
\end{aligned}
\end{equation}
for all $0<t<\infty$ and $t/2 \le \tau \le t$.
By combining the above estimates, we obtain the desired result.

\noindent
\textbf{(ii)} In the same way as in the case of (i), we may show that
\begin{equation}
\begin{aligned}
\|\mathcal{E}[\chi_{\delta}(\bm{u}\cdot\nabla\widetilde{c})](\tau,\,\cdot\,)\|_{B_{r_*,\infty}^{s+1}(\mathbb{R}^N)} &\le C\|\chi_{\delta}\bm{u}(\tau,\,\cdot\,)\|_{B_{\infty,\infty}^{s+1}(\Omega)}\|\chi_{\delta/2}\widetilde{c}(\tau,\,\cdot\,)\|_{B_{r_*,\infty}^{s+2}(\Omega)} \\
&\le CM_{2,\delta}(r_*,s,t), \\
\|\mathcal{E}[2\nabla\chi_{\delta}\cdot\nabla\widetilde{c}+\widetilde{c}\Delta\chi_{\delta}](\tau,\,\cdot\,)\|_{B_{r_*,\infty}^{s+1}(\mathbb{R}^N)} &\le C_{\delta}\|\chi_{\delta/2}\widetilde{c}(\tau,\,\cdot\,)\|_{B_{r_*,\infty}^{s+2}(\Omega)} \\
&\le C_{\delta}M_{2,\delta}(r_*,s,t)
\end{aligned}
\end{equation}
for all $0<t<\infty$ and $t/2 \le \tau \le t$, which yield the desired estimate.
\end{proof}

\subsection{Proof of Theorem \ref{solreg}}

The non-negativity of $n$ and $c$ may essentially be proved by the following maximum principle for the parabolic problems, which is a variant assertion of Proposition 52.8 (cf.~Remark 52.9) in \cite{quittner}.

\begin{lemm}\label{nonneg}
Let $\Omega$ be a smooth (possibly unbounded) domain and let $\alpha>0$ be a constant.
Suppose that
$$U \in BUC([0,\infty);C^1(\overline{\Omega})) \cap C((0,\infty);C^2(\Omega)) \cap C^1((0,\infty);C(\Omega))$$
satisfies
\begin{equation}\label{uineq}
\left\{\begin{aligned}
\partial_tU-\Delta U &\ge -\alpha (|\nabla U|+|U|), & t>0, \, x &\in \Omega, \\
\nabla U\cdot \bm{\nu} &\ge -\alpha|U|, & t>0, \, x &\in \partial\Omega, \\
U(0,x) &\ge 0, & x &\in \Omega.
\end{aligned}\right.
\end{equation}
Then there holds $U \ge 0$ in $(0,\infty) \times \Omega$.
\end{lemm}

The proof of  Lemma \ref{nonneg} is given in the appendix.
Now we are ready to show Theorem \ref{solreg}.

\begin{proof}[Proof of Theorem \ref{solreg}]
We have $n \in BUC([0,\infty);C(\overline{\Omega}))$ and $c \in BUC([0,\infty);C^1(\overline{\Omega}))$ due to \eqref{solclass} and Corollary \ref{ncuemb} (ii).
In addition, since $q>2$, Proposition \ref{traceemb} yields
\begin{equation}
\begin{aligned}
\widetilde{n} &\in BUC([0,\infty);B_{r,q}^{2-2/q}(\Omega)) \subset BUC([0,\infty);B_{r,\infty}^{1+\eta(1-2/q)}(\Omega)), \\
\widetilde{c} &\in BUC([0,\infty);B_{r,q}^{3-2/q}(\Omega)) \subset BUC([0,\infty);B_{r,\infty}^{2+\eta(1-2/q)}(\Omega))
\end{aligned}
\end{equation}
for all $0<\eta<1$ and we obtain
\begin{equation}\label{nctraceemb2}
\|\widetilde{n}\|_{L^{\infty}(B_{r,q}^{2-2/q}(\Omega))} \le C\|\widetilde{n}\|_{\mathbb{E}_{q,r}^0}, \quad \|\widetilde{c}\|_{L^{\infty}(B_{r,q}^{3-2/q}(\Omega))} \le C\|\widetilde{c}\|_{\mathbb{E}_{q,r}^1}.
\end{equation}
Let $K$ be an arbitrary compact subset of $\Omega$. As $\delta_0 \coloneqq \inf_{x \in \partial K, \, y \in \partial\Omega}|x-y|>0$, we note that $\widetilde{n}=\chi_{\delta_*}\widetilde{n}$ and $\widetilde{c}=\chi_{\delta_*}\widetilde{c}$ in $(0,\infty) \times K$, provided that $0<\delta_*<\delta_0$. We now choose $0<\delta<4\delta_0$.
Since we have assumed \eqref{givenreg1} and \eqref{givenreg2} and since $\bm{u} \in C((0,\infty) ;C^2(\overline{\Omega})^N)$ from Lemma \ref{ureg}, by taking $0<\eta<1$ so that $\eta(1-2/q)<\theta_0$, we deduce that
$$M_{1,\delta}(r,\eta(1-2/q),t)+M_{2,\delta}(r,\eta(1-2/q),t)<\infty$$
for all $0<t<\infty$, where $M_{1,\delta}$ and $M_{2,\delta}$ are given by \eqref{m1def} and \eqref{m2def}, respectively.
Hence, by noting \eqref{nctraceemb2} and estimating the integral form \eqref{INT} with the aid of Proposition \ref{besovest} (i) and Lemma \ref{FGest}, we obtain
\begin{equation}
\begin{aligned}
\|\mathcal{E}[\chi_{\delta}\widetilde{n}(t,\,\cdot\,)]\|_{B_{r,\infty}^2(\mathbb{R}^N)} &\le C(1+t^{-1})\|\chi_{\delta}\widetilde{n}(t/2,\,\cdot\,)\|_{L^r(\Omega)} \\
&\quad +C\int_{t/2}^t\{1+(t-\tau)^{-(1/2)(2-\eta(1-2/q))}\} \|\mathcal{E}[F_{\delta}(\tau,\,\cdot\,)]\|_{B_{r,\infty}^{\eta(1-2/q)}(\mathbb{R}^N)} \, d\tau \\
&\le C(1+t^{-1})\|\widetilde{n}\|_{\mathbb{E}_{q,r}^0}+C_{\delta}(t+t^{(\eta/2)(1-2/q)})M_{1,\delta}(r,\eta(1-2/q),t)
\end{aligned}
\end{equation}
and
\begin{equation}
\begin{aligned}
\|\mathcal{E}[\chi_{\delta}\widetilde{c}(t,\,\cdot\,)]\|_{B_{r,\infty}^3(\mathbb{R}^N)} &\le C(1+t^{-3/2})\|\chi_{\delta}\widetilde{c}(t/2,\,\cdot\,)\|_{L^r(\Omega)} \\
&\quad +C\int_{t/2}^t\{1+(t-\tau)^{-(1/2)(2-\eta(1-2/q))}\}\|\mathcal{E}[G_{\delta}(\tau,\,\cdot\,)]\|_{B_{r,\infty}^{1+\eta(1-2/q)}(\mathbb{R}^N)} \, d\tau \\
&\le C(1+t^{-3/2})\|\widetilde{c}\|_{\mathbb{E}_{q,r}^1}+C_{\delta}(t+t^{(\eta/2)(1-2/q)})M_{2,\delta}(r,\eta(1-2/q),t)
\end{aligned}
\end{equation}
for all $0<t<\infty$.
Thus we see by Proposition \ref{besovemb} that
\begin{equation}
\begin{aligned}
\chi_{\delta}\widetilde{n} &\in C((0,\infty);B_{r,\infty}^2(\Omega)) \subset C((0,\infty);C^{1+\theta_*}(\overline{\Omega})), \\
\chi_{\delta}\widetilde{c} &\in C((0,\infty);B_{r,\infty}^3(\Omega)) \subset C((0,\infty);C^{2+\theta_*}(\overline{\Omega}))
\end{aligned}
\end{equation}
with some $0<\theta_* \le \theta_0$.
Therefore, it holds that
$$M_{1,2\delta}(\infty,\theta_*,t)+M_{2,2\delta}(\infty,\theta_*,t)<\infty$$
for all $0<t<\infty$, which yields
\begin{equation}
\begin{aligned}
\|\mathcal{E}[\chi_{2\delta}\widetilde{n}(t,\,\cdot\,)]\|_{B_{\infty,\infty}^{2+\theta_*-\eta}(\mathbb{R}^N)} &\le C(1+t^{-(2+\theta_*-\eta)/2})\|\chi_{2\delta}\widetilde{n}(t/2,\,\cdot\,)\|_{L^{\infty}(\Omega)} \\
&\quad +C\int_{t/2}^t\{1+(t-\tau)^{-(2-\eta)/2}\}\|\mathcal{E}[F_{2\delta}(\tau,\,\cdot\,)]\|_{B_{\infty,\infty}^{\theta_*}(\mathbb{R}^N)} \, d\tau \\
&\le C(1+t^{-1-\theta_*/2+\eta/2})\|\widetilde{n}\|_{\mathbb{E}_{q,r}^0}+C_{\delta}(t+t^{\eta/2})M_{1,2\delta}(\infty,\theta_*,t)
\end{aligned}
\end{equation}
and
\begin{equation}
\begin{aligned}
\|\mathcal{E}[\chi_{2\delta}\widetilde{c}(t,\,\cdot\,)]\|_{B_{\infty,\infty}^{3+\theta_*-\eta}(\mathbb{R}^N)} &\le C(1+t^{-(3+\theta_*-\eta)/2})\|\chi_{2\delta}\widetilde{c}(t/2,\,\cdot\,)\|_{L^{\infty}(\Omega)} \\
&\quad +C\int_{t/2}^t\{1+(t-\tau)^{-(2-\eta)/2}\}\|\mathcal{E}[G_{2\delta}(\tau,\,\cdot\,)]\|_{B_{\infty,\infty}^{1+\theta_*}(\mathbb{R}^N)} \, d\tau \\
&\le C(1+t^{-3/2-\theta_*/2+\eta/2})\|\widetilde{c}\|_{\mathbb{E}_{q,r}^1}+C(t+t^{\eta/2})M_{2,2\delta}(\infty,\theta_*,t)
\end{aligned}
\end{equation}
for all $0<\eta<\theta_*$.
Hence, there holds
\begin{equation}
\begin{aligned}
\chi_{2\delta}\widetilde{n} &\in C((0,\infty);B_{\infty,\infty}^{2+\theta_*-\eta}(\Omega)) \subset C((0,\infty);C^{1+\theta_0}(\overline{\Omega})), \\
\chi_{2\delta}\widetilde{c} &\in C((0,\infty);B_{\infty,\infty}^{3+\theta_*-\eta}(\Omega)) \subset C((0,\infty);C^{2+\theta_0}(\overline{\Omega})).
\end{aligned}
\end{equation}
By repeating the aforementioned argument with $\theta_*$ replaced by $\theta_0$, we see that
\begin{equation}\label{ncreg}
\chi_{4\delta}\widetilde{n} \in C((0,\infty);C^{2+\theta_0-\eta}(\overline{\Omega})), \quad \chi_{4\delta}\widetilde{c} \in C((0,\infty);C^{3+\theta_0-\eta}(\overline{\Omega}))
\end{equation}
for all $0<\eta<\theta_0$.
In addition, since \eqref{ncreg} implies that $M_{2,4\delta}(\infty,1+\theta_0-\eta,t)<\infty$ for all $0<t<\infty$ and $0<\eta<\theta_0$, we also observe that
\begin{equation}
\begin{aligned}
\|\mathcal{E}[\chi_{4\delta}\widetilde{c}(t,\,\cdot\,)]\|_{B_{\infty,\infty}^{4+\theta_0-2\eta}(\mathbb{R}^N)} &\le C(1+t^{-(4+\theta_0-2\eta)/2})\|\chi_{4\delta}\widetilde{c}(t/2,\,\cdot\,)\|_{L^{\infty}(\Omega)} \\
&\quad +C\int_{t/2}^t\{1+(t-\tau)^{-(2-\eta)/2}\} \|\mathcal{E}[G_{4\delta}(\tau,\,\cdot\,)]\|_{B_{\infty,\infty}^{2+\theta_0-\eta}(\mathbb{R}^N)} \, d\tau \\
&\le C(1+t^{-2-\theta_0/2+\eta})\|\widetilde{c}\|_{\mathbb{E}_{q,r}^1}+C(t+t^{\eta/2})M_{2,4\delta}(\infty,1+\theta_0-\eta,t),
\end{aligned}
\end{equation}
which yields
\begin{equation}\label{creg}
\chi_{4\delta}\widetilde{c} \in C((0,\infty);C^{4+\theta_0-2\eta}(\overline{\Omega}))
\end{equation}
for all $0<\eta<\theta_0/2$.
Therefore, we see by \eqref{cutsys} and Lemma \ref{FGest} that
\begin{equation}
\begin{aligned}
\|\partial_t\mathcal{E}[\chi_{4\delta}\widetilde{n}(t,\,\cdot\,)]\|_{C^{\theta_0-\eta}(\mathbb{R}^N)} &\le \|\Delta\mathcal{E}[\chi_{4\delta}\widetilde{n}(t,\,\cdot\,)]\|_{C^{\theta_0-\eta}(\mathbb{R}^N)}+\|\mathcal{E}[F_{4\delta}(t,\,\cdot\,)]\|_{C^{\theta_0-\eta}(\mathbb{R}^N)} \\
&\le C\left\{\|\chi_{4\delta}\widetilde{n}(t,\,\cdot\,)\|_{C^{2+\theta_0-\eta}(\Omega)}+M_{1,4\delta}(\infty,1+\theta_0-\eta,t)\right\}, \\
\|\partial_t\mathcal{E}[\chi_{4\delta}\widetilde{c}(t,\,\cdot\,)]\|_{C^{2+\theta_0-2\eta}(\mathbb{R}^N)} &\le \|(1-\Delta)\mathcal{E}[\chi_{4\delta}\widetilde{c}(t,\,\cdot\,)]\|_{C^{2+\theta_0-2\eta}(\mathbb{R}^N)}+\|\mathcal{E}[G_{4\delta}(t,\,\cdot\,)]\|_{C^{2+\theta_0-2\eta}(\mathbb{R}^N)} \\
&\le C\left\{\|\chi_{4\delta}\widetilde{c}(t,\,\cdot\,)\|_{C^{4+\theta_0-2\eta}(\Omega)}+M_{2,4\delta}(\infty,1+\theta_0-2\eta,t)\right\}
\end{aligned}
\end{equation}
for all $0<t<\infty$ and for sufficiently small $0<\eta<1$.
Together with \eqref{ncreg} and \eqref{creg}, we observe that
\begin{equation}
\chi_{4\delta}\widetilde{n} \in C^1((0,\infty);C^{\theta_0-\eta}(\overline{\Omega})), \quad \chi_{4\delta}\widetilde{c} \in C^1((0,\infty);C^{2+\theta_0-2\eta}(\overline{\Omega})),
\end{equation}
which yield
\begin{equation}
\widetilde{n} \in C^1((0,\infty);C^{\theta_0-\eta}(K)), \quad \widetilde{c} \in C^1((0,\infty);C^{2+\theta_0-2\eta}(K))
\end{equation}
due to $0<\delta<4\delta_0$. Thus, we have \eqref{ncureg} with the aid of Lemma \ref{ureg} and \eqref{ncreg} and \eqref{creg}.

In addition, by virtue of the regularities \eqref{ncureg}, we see that $(n,c,\bm{u})$ satisfies \eqref{ABS} in a classical sense.
Hence, integrating the first equation of \eqref{ABS} over $\Omega$ and applying Proposition \ref{gaussdiv}, we have $\partial_t\int_{\Omega}n(t,x)\,dx=0$ for all $0<t<\infty$, which yields $\int_{\Omega}n(t,x)\,dx=\int_{\Omega}n_0(x)\,dx$.
Moreover, integrating the second equation of \eqref{ABS} over $\Omega$ and applying Proposition \ref{gaussdiv} again, we observe that
$$\partial_t\int_{\Omega}c(t,x)\,dx+\int_{\Omega}c(t,x)\,dx=\int_{\Omega}n(t,x)\,dx$$
for all $0<t<\infty$.
Since $\int_{\Omega}n(t,x)\,dx=\int_{\Omega}n_0(x)\,dx$ implies that 
\begin{equation}
\partial_tI(t)+I(t)=\int_{\Omega}n_0(x)\,dx, \qquad I(t) \coloneqq \int_{\Omega}c(t,x)\,dx,
\end{equation}
we obtain $I(t)=me^{-t}+\int_{\Omega}n_0(x)\,dx$ with some constant $m \in \mathbb{R}$.
Noting that
$$m=I(0)-\int_{\Omega}n_0(x)\,dx=\int_{\Omega}c_0(x)\,dx-\int_{\Omega}n_0(x)\,dx,$$
we have the desired identity.

Concerning the non-negativity of solutions, we assume that $N/r+2/q<1$. Then Proposition \ref{besovemb} yields $B_{r,q}^{2-2/q}(\Omega) \subset C^1(\overline{\Omega})$.
Thus we see by \eqref{solclass} and Proposition \ref{traceemb} that
$$e^{\lambda t}(n-\overline{n}_0) \in BUC([0,\infty);B_{r,q}^{2-2/q}(\Omega)) \subset BUC([0,\infty);C^1(\overline{\Omega})).$$
Therefore, by noting \eqref{solclass} again, we have
$$n,c \in BUC([0,\infty);C^1(\overline{\Omega})) \cap C((0,\infty);C^2(\Omega)) \cap C^1((0,\infty);C(\Omega)).$$
Notice that the nonlinear boundary condition $\nabla n \cdot \bm{\nu}=nS(t,x)\nabla c \cdot \bm{\nu}$ makes sense in $C(\partial\Omega)$ because of \eqref{VNBC} (see Remark \ref{regrema} (ii)).
In addition, since $(n,c,\bm{u})$ satisfies \eqref{ABS}, there holds
\begin{equation}
\begin{aligned}
&\partial_tn-\Delta n \\
&= -\nabla n \cdot (S\nabla c)-n\nabla \cdot (S\nabla c)-\bm{u}\cdot\nabla n \\
&\ge -C\left(|\nabla n|\|S\|_{L^{\infty}(L^{\infty}(\Omega))}\|c\|_{L^{\infty}(C^1(\overline{\Omega}))}+|n|\|S\|_{L^{\infty}(C^1(\overline{\Omega}))}\|c\|_{L^{\infty}(C^2(\overline{\Omega}))}+|\nabla n|\|\bm{u}\|_{L^{\infty}(L^{\infty}(\Omega))}\right) \\
&\ge -C\left(\|S\|_{L^{\infty}(C^1(\overline{\Omega}))}\|c\|_{L^{\infty}(C^2(\overline{\Omega}))}+\|\bm{u}\|_{L^{\infty}(L^{\infty}(\Omega))}\right)(|\nabla n|+|n|)
\end{aligned}
\end{equation}
in $(0,\infty) \times \Omega$ and
$$\nabla n \cdot \bm{\nu} \ge -C|n|\|S\|_{L^{\infty}(L^{\infty}(\Omega))}\|c\|_{L^{\infty}(C^1(\overline{\Omega}))}$$
on $(0,\infty) \times \partial\Omega$.
By setting
$$\alpha \coloneqq C\left(\|S\|_{L^{\infty}(C^1(\overline{\Omega}))}\|c\|_{L^{\infty}(C^2(\overline{\Omega}))}+\|\bm{u}\|_{L^{\infty}(L^{\infty}(\Omega))}\right),$$
we may apply Lemma \ref{nonneg} to obtain $n \ge 0$ in $(0,\infty) \times \Omega$.
Concerning the function $c$, since we have $\nabla c \cdot \bm{\nu}=0$ on $(0,\infty) \times \partial\Omega$ and
\begin{equation}
\begin{aligned}
\partial_tc-\Delta_{\mathrm{N},1}c &=-c+n-\bm{u}\cdot\nabla c \ge -|c|-C|\nabla c|\|\bm{u}\|_{L^{\infty}(L^{\infty}(\Omega)} \\
&\ge -C\left(1+\|\bm{u}\|_{L^{\infty}(L^{\infty}(\Omega)}\right)(|\nabla c|+|c|)
\end{aligned}
\end{equation}
in $(0,\infty) \times \Omega$ from $n \ge 0$, by setting
$\alpha \coloneqq C(1+\|\bm{u}\|_{L^{\infty}(L^{\infty}(\Omega))}),$
we may apply Lemma \ref{nonneg} to obtain $c \ge 0$ in $(0,\infty) \times \Omega$.
This completes the proof of Theorem \ref{solreg}.
\end{proof}

\section*{Acknowledgment}

The authors would like to express their appreciation to the anonymous referees for their comments on the original manuscript, in particular for presenting fairly recent contributions on related chemotaxis models.

The first author was partially supported by JSPS KAKENHI Grant Numbers JP22J12100 and JP22KJ2930.
The second author was partially supported by JSPS KAKENHI Grant Number JP21K13826.

\section*{Conflict of interest statement}

On behalf of all authors, the corresponding author states that there is no conflict of interest.

\section*{Data availability statement}

Data sharing not applicable to this article as no datasets were generated or analysed during the current study.

\appendix
\section{The estimates of $P$ and $e^{-tA}$ in the Besov spaces}
\label{ap-A}

In this appendix, we show Proposition \ref{besovstokes}, which gives the boundedness results of the Helmholtz projection $P$ and the Stokes semigroup $e^{-tA}$ in the Besov spaces framework.
The proof is mainly based on a combination of classical estimates due to \cite{cattabriga} and \cite{gigamiyakawa} with the real interpolation method.

\begin{proof}[Proof of Proposition \ref{besovstokes}]
\textbf{(i)} Note that there holds
\begin{equation}\label{helmholtz}
\|P\bm{\psi}_*\|_{W^{k,r_*}(\Omega)} \le C\|\bm{\psi}_*\|_{W^{k,r_*}(\Omega)}
\end{equation}
for all $1<r_*<\infty$ and $k \in \mathbb{N} \cup \{0\}$ and $\bm{\psi}_* \in W^{k,r_*}(\Omega)^N$ from the result given by Giga and Miyakawa \cite{gigamiyakawa}*{Lemma 3.3 (i)}.
Given $0<s<\infty$, by taking $k \in \mathbb{N}$ so that $k>s$, we have $B_{r_*,\infty}^s(\Omega)=(L^{r_*}(\Omega),W^{k,r_*}(\Omega))_{s/k,\infty}$.
Thus we obtain the desired result by combining the above estimate with the usual property of the real interpolation spaces \cite{lunardi}*{Theorem 1.6}.

\noindent
\textbf{(ii)} We recall the estimate
$$\|\bm{\psi}_*\|_{W^{k+2,r_*}(\Omega)} \le C\|A\bm{\psi}_*\|_{W^{k,r_*}(\Omega)}$$
for all $1<r_*<\infty$ and $k \in \mathbb{N} \cup \{0\}$ and $\bm{\psi}_* \in D(A)$ such that $A\bm{\psi}_* \in W^{k,r_*}(\Omega)^N$ given by Cattabriga \cite{cattabriga}*{Theorem} (see also \cite{gigamiyakawa}*{Lemma 3.1}).
Here we take a sufficiently large $k \in \mathbb{N}$ and apply the above estimate to obtain
\begin{equation}\label{stokesest}
\|\bm{\psi}_*\|_{W^{2k,r_*}(\Omega)} \le C\|A\bm{\psi}_*\|_{W^{2k-2,r_*}(\Omega)} \le C\|A(A\bm{\psi}_*)\|_{W^{2k-4,r_*}(\Omega)} \le \cdots \le C\|A^k\bm{\psi}_*\|_{L^{r_*}(\Omega)}
\end{equation}
for all $\bm{\psi}_* \in D(A^k)$.
Since
\begin{equation}
\|Ae^{-tA}\bm{\psi}_*\|_{L^{r_*}(\Omega)} \le Ct^{-1}\|\bm{\psi}_*\|_{L^{r_*}(\Omega)}
\end{equation}
for all $1<r_*<\infty$ and $0<t<\infty$ and $\bm{\psi}_* \in L_{\sigma}^{r_*}(\Omega)$ \cite{gigamiyakawa}*{Proposition 1.2}, it holds by \eqref{stokesest} that
\begin{equation}\label{stokessemiest}
\begin{aligned}
\|e^{-tA}P\bm{\psi}\|_{W^{2k,r_*}(\Omega)} &\le C\|A^ke^{-tA}P\bm{\psi}\|_{L^{r_*}(\Omega)}=C\|(Ae^{-(t/k)A})^kP\bm{\psi}\|_{L^{r_*}(\Omega)} \\
&\le C(t/k)^{-1}\|(Ae^{-(t/k)A})^{k-1}P\bm{\psi}\|_{L^{r_*}(\Omega)} \\
&\le C(t/k)^{-2}\|(Ae^{-(t/k)A})^{k-2}P\bm{\psi}\|_{L^{r_*}(\Omega)} \\
&\le \cdots \le Ct^{-k}\|\bm{\psi}\|_{L^{r_*}(\Omega)}.
\end{aligned}
\end{equation}
Noting that $\|e^{-tA}P\bm{\psi}\|_{L^{r_*}(\Omega)} \le C\|\bm{\psi}\|_{L^{r_*}(\Omega)}$ and $B_{r_*,1}^{s_0}(\Omega)=(L^{r_*}(\Omega),W^{2k,r_*}(\Omega))_{s_0/(2k),1}$ for $0<s_0<\infty$ and $k \in \mathbb{N}$ such that $2k>s_0$, we see by \cite{lunardi}*{Theorem 1.6} that
$$\|e^{-tA}P\bm{\psi}\|_{B_{r_*,1}^{s_0}(\Omega)} \le C(t^{-k})^{s_0/(2k)}\|\bm{\psi}\|_{L^{r_*}(\Omega)}=Ct^{-s_0/2}\|\bm{\psi}\|_{L^{r_*}(\Omega)}.$$

Furthermore, recalling $A=-P\Delta$ and combining \eqref{helmholtz} and \eqref{stokesest}, we have
\begin{equation}
\begin{aligned}
\|e^{-tA}P\bm{\psi}\|_{W^{2k,r_*}(\Omega)} &\le C\|A^ke^{-tA}P\bm{\psi}\|_{L^{r_*}(\Omega)} \le C\|A^kP\bm{\psi}\|_{L^{r_*}(\Omega)} \le C\|A^{k-1}P\bm{\psi}\|_{W^{2,r_*}(\Omega)} \\
&\le C\|A^{k-2}P\bm{\psi}\|_{W^{4,r_*}(\Omega)} \le \cdots \le C\|P\bm{\psi}\|_{W^{2k,r_*}(\Omega)} \le C\|\bm{\psi}\|_{W^{2k,r_*}(\Omega)}.
\end{aligned}
\end{equation}
Given $s,s_0 \in (0,\infty)$, we take $k \in \mathbb{N}$ satisfying $2k>s+s_0$.
Then we see by $B_{r_*,\infty}^s(\Omega)=(L^{r_*}(\Omega),W^{2k,r_*}(\Omega))_{s/(2k),\infty}$ and \cite{lunardi}*{Theorem 1.6} that
$$\|e^{-tA}P\bm{\psi}\|_{B_{r_*,\infty}^s(\Omega)} \le C\|\bm{\psi}\|_{B_{r_*,\infty}^s(\Omega)}.$$
We apply \cite{lunardi}*{Theorem 1.6} again with the aid of \eqref{stokessemiest} to obtain
$$\|e^{-tA}P\bm{\psi}\|_{W^{2k,r_*}(\Omega)} \le C(t^{-k})^{1-s/(2k)}\|\bm{\psi}\|_{B_{r_*,\infty}^s(\Omega)}=Ct^{-(2k-s)/2}\|\bm{\psi}\|_{B_{r_*,\infty}^s(\Omega)}.$$
Therefore, noting that
\begin{equation}
\begin{aligned}
(B_{r_*,\infty}^s(\Omega),W^{2k,r_*}(\Omega))_{s_0/(2k-s),1} &= ((L^{r_*}(\Omega),W^{2k,r_*}(\Omega))_{s/(2k),\infty},W^{2k,r_*}(\Omega))_{s_0/(2k-s),1} \\
&=(L^{r_*}(\Omega),W^{2k,r_*}(\Omega))_{(s+s_0)/(2k),1} \\
&=B_{r_*,1}^{s+s_0}(\Omega)
\end{aligned}
\end{equation}
from \cite{lunardi}*{Corollary 1.24} and using \cite{lunardi}*{Corollary 1.7}, we deduce that
\begin{equation}
\begin{aligned}
\|e^{-tA}P\bm{\psi}\|_{B_{r_*,1}^{s+s_0}(\Omega)} &\le C\|e^{-tA}P\bm{\psi}\|_{B_{r_*,\infty}^s(\Omega)}^{1-s_0/(2k-s)} \|e^{-tA}P\bm{\psi}\|_{W^{2k,r_*}(\Omega)}^{s_0/(2k-s)} \\
&\le C\|\bm{\psi}\|_{B_{r_*,\infty}^s(\Omega)}^{1-s_0/(2k-s)}(t^{-(2k-s)/2}\|\bm{\psi}\|_{B_{r_*,\infty}^s(\Omega)})^{s_0/(2k-s)} \\
&=Ct^{-s_0/2}\|\bm{\psi}\|_{B_{r_*,\infty}^s(\Omega)},
\end{aligned}
\end{equation}
which completes the proof of Proposition \ref{besovstokes}.
\end{proof}

\section{Non-negativity of solutions}
\label{ap-B}

In this appendix, we give the proof of Lemma \ref{nonneg}, i.e., we show the non-negativity of solutions to a special parabolic problem \eqref{uineq}.
The claim and its proof are almost the same as in Proposition 52.8 and Remark 52.9 in \cite{quittner}, but our boundary condition $\nabla U\cdot\bm{\nu} \ge -\alpha|U|$ in \eqref{uineq} is slightly weaker than the assertion recorded in \cite{quittner} (in particular, the lower bound of $\nabla U \cdot \bm{\nu}$ on $\partial \Omega$ is given by $- \alpha|U|$ instead of $\alpha U$).
To prove Lemma \ref{nonneg}, we recall the following trace inequality.

\begin{prop}\label{traceineq}
Let $\Omega \subset \mathbb{R}^N$, $N \ge 2$, be a smooth (possibly unbounded) domain and let $\alpha>0$ be a constant.
Assume that $\psi \in W^{1,2}(\Omega)$.
Then there holds $\psi \in L^2(\partial\Omega)$ with the estimate
$$\int_{\partial\Omega}|\psi(x)|^2\,d\sigma(x) \le \frac{1}{2\alpha}\int_{\Omega}|\nabla\psi(x)|^2\,dx+C(1+\alpha)\int_{\Omega}|\psi(x)|^2\,dx,$$
where $C>0$ is a constant independent of $\alpha$ and $\psi$.
\end{prop}

\begin{proof}
By applying \cite{galdi}*{Theorem II.4.1} with $r=q=2$ and $m=1$, we have $\psi \in L^2(\partial\Omega)$ and
\begin{equation}
\begin{aligned}
\int_{\partial\Omega}|\psi(x)|^2\,d\sigma(x) &\le C\left\{\int_{\Omega}|\psi(x)|^2\,dx+\left(\int_{\Omega}|\psi(x)|^2\,dx\right)^{1/2}\left(\int_{\Omega}(|\psi(x)|^2+|\nabla\psi(x)|^2)\,dx\right)^{1/2}\right\} \\
&\le C\left\{2\int_{\Omega}|\psi(x)|^2\,dx+\left(\int_{\Omega}|\psi(x)|^2\,dx\right)^{1/2}\left(\int_{\Omega}|\nabla\psi(x)|^2\,dx\right)^{1/2}\right\}.
\end{aligned}
\end{equation}
Since the Young inequality yields
\begin{equation}
\left(\int_{\Omega}|\psi(x)|^2\,dx\right)^{1/2}\left(\int_{\Omega}|\nabla\psi(x)|^2\,dx\right)^{1/2} \le \frac{1}{2}\left\{C\alpha\int_{\Omega}|\psi(x)|^2\,dx+\frac{1}{C\alpha}\int_{\Omega}|\nabla\psi(x)|^2\,dx\right\},
\end{equation}
it holds that
\begin{equation}
\int_{\partial\Omega}|\psi(x)|^2\,d\sigma(x) \le \frac{1}{2\alpha}\int_{\Omega}|\nabla\psi(x)|^2\,dx+C\left(2+\frac{C\alpha}{2}\right)\int_{\Omega}|\psi(x)|^2\,dx,
\end{equation}
which completes the proof.
\end{proof}

\begin{proof}[Proof of Lemma \ref{nonneg}]
We follow the discussions in Proposition 52.8 and Remark 52.9 in \cite{quittner}.
Set $U_-(t,x) \coloneqq \min\{0,U(t,x)\} \le 0$ for $(t,x) \in (0,\infty) \times \Omega$.
Since
$$\partial_tU_-=\left\{\begin{array}{cl}
0 & \text{if $U \ge 0$}, \\
\partial_tU & \text{if $U<0$},
\end{array}\right. \quad \nabla U_-=\left\{\begin{array}{cl}
0 & \text{if $U \ge 0$}, \\
\nabla U & \text{if $U<0$},
\end{array}\right.$$
we have $\partial_tU_- \in L^{\infty}(0,\infty;L^{\infty}(\Omega))$ and $\nabla U_- \in L^{\infty}(0,\infty;L^{\infty}(\Omega)^N)$.
In addition, there holds $U_-|\nabla U|=-|U_-||\nabla U_-|$ and $U_-|U|=-|U_-|^2$.
Hence, noting that $\partial_t|U_-|^2 =2U_-\partial_tU$, we see by \eqref{uineq} and the Young inequality that
\begin{equation}\label{umest}
\begin{aligned}
\partial_t|U_-|^2 &\le 2U_-(\Delta U-\alpha|\nabla U|-\alpha|U|) \\
&=2U_-\Delta U+2\alpha|U_-||\nabla U_-|+2\alpha|U_-|^2 \\
&\le 2U_-\Delta U+\left(\alpha^2|U_-|^2+|\nabla U_-|^2\right)+2\alpha|U_-|^2 \\
&= 2U_-\Delta U+|\nabla U_-|^2+\alpha(2+\alpha)|U_-|^2
\end{aligned}
\end{equation}
in $(0,\infty) \times \Omega$.
Here, since \eqref{uineq} yields $U_-\nabla U\cdot \bm{\nu} \le -\alpha U_-|U|=\alpha |U_-|^2$, the integration by parts and Proposition \ref{traceineq} imply that
\begin{equation}
\begin{aligned}
2\int_{\Omega}U_-(t,x)\Delta U(t,x)\,dx &= 2\int_{\partial\Omega}U_-(t,x)\nabla U(t,x) \cdot \bm{\nu}(x)\,d\sigma (x)-2\int_{\Omega}\nabla U(t,x)\cdot \nabla U_-(t,x)\,dx \\
&\le 2\alpha\int_{\partial\Omega}|U_-(t,x)|^2\,d\sigma (x)-2\int_{\Omega}|\nabla U_-(t,x)|^2\,dx \\
&\le -\int_{\Omega}|\nabla U_-(t,x)|^2\,dx+C\alpha(1+\alpha)\int_{\Omega}|U_-(t,x)|^2\,dx
\end{aligned}
\end{equation}
for all $0<t<\infty$.
Therefore, integrating the inequality \eqref{umest} over $\Omega$ and using the above estimate, we observe that
\begin{equation}
\begin{aligned}
\partial_t\int_{\Omega}|U_-(t,x)|^2\,dx &\le -\int_{\Omega}|\nabla U_-(t,x)|^2\,dx+C\alpha(1+\alpha)\int_{\Omega}|U_-(t,x)|^2\,dx \\
&\quad +\int_{\Omega}|\nabla U_-(t,x)|^2\,dx+\alpha(2+\alpha)\int_{\Omega}|U_-(t,x)|^2\,dx \\
&\le C\alpha(1+\alpha)\int_{\Omega}|U_-(t,x)|^2\,dx
\end{aligned}
\end{equation}
for all $0<t<\infty$.
Since we obtain
$$\int_{\Omega}|U_-(t,x)|^2\,dx \le e^{C\alpha(1+\alpha)t}
\int_{\Omega}|U_-(0,x)|^2\,dx$$
and since $U_-(0,x) \equiv 0$ due to \eqref{uineq}, it necessarily holds $U \ge 0$ in $(0,\infty) \times \Omega$.
\end{proof}

\end{document}